\documentclass{amsart}

\usepackage{amsfonts, amssymb, amsmath, eucal, verbatim, amsthm, amscd, enumerate}

\usepackage{graphicx}
\usepackage{framed}
\usepackage{tikz}
\usepackage{enumitem}

\usepackage{ascmac}

\newtheorem{theorem}{Theorem}[section]
\newtheorem{lemma}[theorem]{Lemma}
\newtheorem{proposition}[theorem]{Proposition}

\newtheorem{corollary}[theorem]{Corollary}
\theoremstyle{definition}

\theoremstyle{remark}
\newtheorem*{remark}{Remark}

\newtheorem*{organisation}{Organisation}

\newcommand{\vertiii}[1]{{\left\vert\kern-0.25ex\left\vert\kern-0.25ex\left\vert #1 
\right\vert\kern-0.25ex\right\vert\kern-0.25ex\right\vert}}
\newcommand{\R}{{\mathbb R}}

\numberwithin{equation}{section}

\usepackage{setspace}

\def\1{\textbf{\rm 1}}

\def\XXint#1#2#3{{\setbox0=\hbox{$#1{#2#3}{\int}$}
\vcenter{\hbox{$#2#3$}}\kern-.5\wd0}}

\begin{document}

\date{\today}


\author[Bez]{Neal Bez}
\address[Neal Bez]{Department of Mathematics, Graduate School of Science and Engineering,
Saitama University, Saitama 338-8570, Japan}
\email{nealbez@mail.saitama-u.ac.jp}
\author[Nakamura]{Shohei Nakamura}
\address[Shohei Nakamura]{Department of Mathematics, Graduate School of Science, Osaka University, Toyonaka, Osaka 560-0043, Japan}
\email{srmkn@math.sci.osaka-u.ac.jp}
\author[Tsuji]{Hiroshi Tsuji}
\address[Hiroshi Tsuji]{Department of Mathematics, Graduate School of Science, Osaka University, Toyonaka, Osaka 560-0043, Japan}
\email{u302167i@ecs.osaka-u.ac.jp}


\title[Stability of hypercontractivity, the LSI and Talagrand's inequality]{Stability of hypercontractivity, the logarithmic Sobolev inequality, and Talagrand's cost inequality}

\begin{abstract}
We provide deficit estimates for Nelson's hypercontractivity inequality, the logarithmic Sobolev inequality, and Talagrand's transportation cost inequality under the restriction that the inputs are semi-log-subharmonic, semi-log-convex, or semi-log-concave. In particular, our result on the logarithmic Sobolev inequality complements a recently obtained result by Eldan, Lehec and Shenfeld concerning a deficit estimate for inputs with small covariance. Similarly, our result on Talagrand's transportation cost inequality complements and, for a large class of semi-log-concave inputs, improves a deficit estimate recently proved by Mikulincer. Our deficit estimates for hypercontractivity will be obtained by using a flow monotonicity scheme built on the Fokker--Planck equation, and our deficit estimates for the logarithmic Sobolev inequality will be derived as a corollary. For Talagrand's inequality, we use an optimal transportation argument. An appealing feature of our framework is robustness and this allows us to derive deficit estimates for the hypercontracivity inequality associated with the Hamilton--Jacobi equation,  the Poincar\'e inequality, and for Beckner's inequality.
\end{abstract}

\maketitle

\section{Introduction and main results}\label{S1}

In their work on stability of the celebrated gaussian logarithmic Sobolev inequality (LSI for short), Ledoux, Nourdin and Peccati \cite{LNP} identified that it is of interest to understand the relationship between the deficit (the difference between the two sides of the inequality) and the size of the covariance matrix of the input. Addressing this, Eldan, Lehec and Shenfeld \cite{ELS} established that for probability distributions whose covariance is bounded above by the identity (in the sense of positive definite transformations), then the deficit is minimised on an appropriate gaussian distribution. 

Although it is well known that the classical form of the LSI yields a multitude of fundamental inequalities, including Talagrand's transportation cost inequality and hypercontractivity, it is often the case that the corresponding deficit estimates do not enjoy such connections (at least not in a direct or obvious way). For example, under an assumption on the smallness of the covariance, it seems unclear how to derive a deficit estimate for Talagrand's inequality from the LSI deficit estimate in \cite{ELS}. Nevertheless, Mikulincer \cite{Miku} argued in a direct manner and, inspired by the approach in \cite{ELS}, was able to establish a certain deficit estimate for Talagrand's inequality. 

Whilst it was also shown in \cite{ELS,Miku} that such deficit results can dramatically fail for certain inputs (in fact, suitably chosen gaussian mixtures) with large covariance, one is certainly left wondering to what extent one may recover such failure by restricting to natural classes of inputs. Addressing this is one source of motivation for the present work.

We are also drawn to the problem of establishing deficit estimates of a similar nature for the hypercontractivity inequality associated with the Ornstein--Uhlenbeck semigroup; again, it seems completely unclear to us whether it is possible to derive such estimates from the LSI deficit estimate in \cite{ELS} under a small assumption on the covariance (see the forthcoming discussion after Theorem \ref{t:MainLSI}). In fact, if one thinks of the standard bridge between hypercontractivity and the LSI (the argument of Gross), it appears \emph{prima facie} that the former is the stronger statement. 

We address both of these issues by considering input functions which are semi-log-subharmonic\footnote{As far as we are aware, the idea of working with semi-log-subharmonic input functions in the context of hypercontractivity goes back to Graczyk, Kemp, and Loeb \cite{GKL} where they observed so-called strong hypercontractivity for the dilation semigroup $T_tf(x) = f(e^{-t}x)$ under the assumption of log-subharmonicity and thus improved some work of Janson \cite{Janson}.} or semi-log-convex. As we shall see, there are a number of advantages of working in such a framework. For example, whilst smallness of the covariance may not be preserved under Ornstein--Uhlenbeck flow, semi-log-subharmonicity and semi-log-convexity are known to be preserved under the flow (see the forthcoming Lemma \ref{l:LogPreserve} for precise details) and this fact will play a crucial role in this paper. In fact, semi-log-concavity is also preserved under the Ornstein--Uhlenbeck flow and from this we shall obtain results connected with the small covariance results mentioned above due to Eldan--Lehec--Shenfeld and Mikulincer; in the case of Talagrand's inequality, our result will be an improvement on Mikulincer's result with regards to the size of the deficit (albeit for a somewhat smaller class of inputs). 

An appealing feature of the framework in which we work is its robustness. As a result, we shall capitalise on the strength and wealth of connections that hypercontractivity enjoys in order to derive improved versions of a variety of functional inequalities, including hypercontractivity associated with the Hamilton--Jacobi semigroup (and consequently an improved form of a dual form of Talagrand's inequality), as well as the Poincar\'e inequality and Beckner's inequality.


\subsection{Functional inequalities}
Throughout the paper we shall frequently use the notation 
\[
\gamma_\beta (x) = \frac{1}{(2\pi \beta)^{\frac{n}{2}}} e^{ -\frac{|x|^2}{2\beta} }, \quad x \in \mathbb{R}^n,
\] 
for the centred and $L^1$-normalised gaussian with variance $\beta>0$. We abbreviate to $\gamma$ in the case $\beta=1$.   The Ornstein--Uhlenbeck semigroup $P_s$,  $s>0$,  is given by 
\begin{equation}\label{e:PsFormula}
P_sf(x) = \int_{\mathbb{R}^n} f(e^{-s} x + \sqrt{1-e^{-2s}} y)\, d\gamma(y)
\end{equation}
for non-negative $f \in L^1(\gamma)$. This $u_s:= P_s f$ solves the gaussian heat equation 
$$
\partial_s u_s = \mathcal{L}_1 u_s,\;\;\; u_0 = f,\;\;\; (s,x) \in (0,\infty) \times \mathbb{R}^n,
$$  
where in general $\mathcal{L}_\beta$ is given by 
\[
\mathcal{L}_\beta \phi(x) = \beta \Delta\phi(x) - x\cdot\nabla \phi(x)
\]
for $\beta>0$. For reasons that will become apparent, as above we shall often write the solution of time-dependent PDE in the form $u_s(x)$ rather than $u(s,x)$ (and so we shall never use the notation $u_s$ to mean $\partial_s u$).

Note that each $P_s$ can be extended to a bounded operator on $L^p(\gamma)$ for all $p\in[1,\infty]$. Nelson's fundamental gaussian hypercontractivity inequality is a quantification of the smoothing property of the Ornstein--Uhlenbeck semigroup $P_s$, and states that
\begin{equation}\label{e:HCClassic}
\big\| P_s\big[ f^\frac1p \big]  \big\|_{ L^q(\gamma) } \le \bigg( \int_{\mathbb{R}^n} f\, d\gamma\bigg)^\frac1p
\end{equation}
for all non-negative $f \in L^1(\gamma)$, where $1<p<q<\infty$ and $s>0$ satisfy $\frac{q-1}{p-1}  = e^{2s}$. This time threshold is sharp (for smaller $s$ the inequality fails even with a constant larger than $1$) and is often referred to as ``Nelson's time". We refer the reader to the survey paper \cite{DGS} and the book \cite{Faris} for detailed historical background and illuminating discussion of the vast connections of hypercontractivity to several branches of the mathematical sciences, including constructive quantum field theory. 

In the case $-\infty <q < p < 1$ with\footnote{If we reinterpret the inequality by writing $g = f^{1/p}$, then the case $p = 0$ may be included since $\|g\|_{L^p(\gamma)} \to \exp(\int \log g \, d\gamma)$ as $p \to 0$.} $p \neq 0$ (and the relation $\frac{q-1}{p-1}  = e^{2s}$ unchanged), the hypercontractivity inequality is known to hold in reverse form
\begin{equation}\label{e:RevHC}
\big\| P_s \big[ f^\frac1p \big] \big\|_{L^q(\gamma)} \ge \bigg( \int_{\mathbb{R}^n} f\, d\gamma \bigg)^\frac1p
\end{equation}
for all positive $f \in L^1(\gamma)$ and the constant is sharp (see Borell \cite{Borell}). The reverse form is also significant for several reasons; for example, from certain deficit estimates that we shall derive for reverse hypercontractivity, later in the paper (Section \ref{S4}) we deduce an improved version of an inequality dual to Talagrand's inequality. We also refer the reader to \cite{MOS} for wider context on reverse hypercontractivity.

As observed by Gross \cite{Gross}, the hypercontractivity inequality 
implies the LSI inequality
\begin{equation}\label{e:LogSob}
{\rm Ent}_{\gamma}(f) \le \frac12 {\rm I}_{\gamma}(f),
\end{equation}
by considering the special case $p = 2$ and $s\to0$ in \eqref{e:HCClassic}\footnote{Remarkably, Gross \cite{Gross} showed one can recover \eqref{e:HCClassic} for all admissible exponents from the LSI and hence hypercontractivity (forward and reverse, in fact) is indeed equivalent to the LSI when considering \emph{general} inputs.}.
Here, the entropy ${\rm Ent}_{\gamma}(f)$ and Fisher information ${\rm I}_{\gamma}(f)$ are defined by 
\begin{align}\label{e:EntFisher}
{\rm Ent}_{\gamma}(f)&:= \int_{\mathbb{R}^n} f\log\, f\, d\gamma - \bigg( \int_{\mathbb{R}^n} f\, d\gamma \bigg) \log\, \bigg( \int_{\mathbb{R}^n} f\, d\gamma \bigg),\\
{\rm I}_{\gamma}(f)&:= \int_{\mathbb{R}^n} \frac{|\nabla f|^2}{f}\, d\gamma. \nonumber 
\end{align}
The LSI inequality too enjoys a plethora of connections to a wide range of fields, including convex geometry, differential geometry, probability theory, and information theory; see, for example, \cite{BoLe}, \cite{BGL}, and the survey paper \cite{LedouxICM}. 
For instance, it is well-known that the LSI implies Talagrand's quadratic transportation cost inequality  which states that  
\begin{equation}\label{e:TalCl}
\frac12 W_2(\gamma, v )^2 \le {\rm Ent}_\gamma\big( \frac{v}{\gamma} \big)
\end{equation}
for $v:\mathbb{R}^n\to (0,\infty)$ satisfying $\int v\, dx =1$ and $\int |x|^2v\, dx < \infty$. Here, $W_2(\cdot,\cdot)$ is the quadratic Wasserstein distance; see Section \ref{S3.5} for details.

Deficit estimates for the LSI have been the subject of a large number of recent papers, including  \cite{BGRS,DoTo,ELS,FIL,FIPR,FMP,Goz,InKi,InMa,Kim,LNP,Tsu}. Among them, we are especially interested in the result by Eldan--Lehec--Shenfeld \cite{ELS}. We are also interested in the stability of Nelson's hypercontractivity inequality \eqref{e:HCClassic} and, unlike other papers on deficit estimates for the LSI, we simultaneously consider the stability of both \eqref{e:HCClassic} and \eqref{e:LogSob}; as we have already mentioned, there is no obvious reason that the \emph{equivalence} of these estimates should persist when restricting the class of input functions and typically one would expect \eqref{e:LogSob} to be the stronger estimate.

Before introducing the various statements of our results, it is instructive to discuss extremisers. Carlen \cite{Carlen} first showed that equality holds in \eqref{e:LogSob} if and only if $f(x) = e^{a\cdot x +b}$,  $a\in\mathbb{R}^n, b\in\mathbb{R}$.  Building on this, Ledoux \cite{LedouxJFA92} showed that (among smooth inputs) equality in \eqref{e:HCClassic} holds if and only if $f(x) = e^{ a\cdot x + b }$, $a\in\mathbb{R}^n$, $b\in\mathbb{R}$. In particular, $f = 1 = \frac{\gamma}{\gamma}$ is essentially the unique extremiser for both of these inequalities and consequently $f = \frac{\gamma_\beta}{\gamma}$ does not attain equality in \eqref{e:HCClassic} or \eqref{e:LogSob} unless $\beta=1$.  Moreover, it is possible to directly compute the deficits (in a multiplicative sense for \eqref{e:HCClassic}) for such inputs: for $\beta\neq 1$,  
\begin{equation}\label{e:DeficitHC}
\frac{\big\| P_s \big[ \big( \frac{\gamma_\beta}{\gamma} \big)^\frac1p \big] \big\|_{L^q(\gamma)}}{\big( \int_{\mathbb{R}^n} \frac{\gamma_\beta}{\gamma}\, d\gamma \big)^\frac1p}
= 
\beta^\frac{n}{2p'} \beta_s^{-\frac{n}{2q'}} < 1,
\end{equation}
where $\beta_s>0$ is given by  
\begin{equation}\label{e:beta2}
\beta_s := 1 + (\beta-1) \frac{q}{p}e^{-2s} 
\end{equation}
and 
\begin{equation}\label{e:DeficitLogSob}
{\rm Ent}_{\gamma}\big( \frac{\gamma_\beta}{\gamma} \big) - \frac12 {\rm I}_{\gamma}\big( \frac{\gamma_\beta}{\gamma} \big) = -\frac{n}{2} ( \log\, \beta - 1 + \frac1\beta ) < 0. 
\end{equation}
It was shown by Eldan--Lehec--Shenfeld \cite{ELS} that the LSI can be improved if the covariance of the input function is small enough. 
Let us recall the covariance matrix of a probability measure $\rho$ is given by 
$$
{\rm cov}\, (\rho) := \left( \int_{\mathbb{R}^n} x_ix_j\, d\rho - \bigg( \int_{\mathbb{R}^n} x_i \, d\rho \bigg)\bigg( \int_{\mathbb{R}^n} x_j \, d\rho \bigg) \right)_{1\le i,j \le n}
$$
and thus in particular we have $ {\rm cov}\, (\gamma_\beta) = \beta {\rm id}$.  It is also straightforward from \eqref{e:DeficitLogSob} to see that  
\begin{equation}\label{e:LSIDefGauss}
{\rm Ent}_{\gamma}\big( \frac{\gamma_a}{\gamma} \big) - \frac12 {\rm I}_{\gamma}\big( \frac{\gamma_a}{\gamma} \big)\le {\rm Ent}_{\gamma}\big( \frac{\gamma_\beta}{\gamma} \big) - \frac12 {\rm I}_{\gamma}\big( \frac{\gamma_\beta}{\gamma} \big)
\end{equation}
for all $0<a\le \beta \le1$. 
It is a consequence of \cite[Theorem 3]{ELS} that, as long as $\beta\le1$, then \eqref{e:LSIDefGauss} continues to hold even if $\gamma_a$ is replaced by any probability distribution whose covariance is $\le \beta{\rm id} $. 
\begin{theorem}[\cite{ELS}]\label{t:ELS}
Let $\beta < 1$. For any non-negative $v $ such that $\int_{\mathbb{R}^n} v\, dx = 1$ and ${\rm cov}\, (v) \le \beta {\rm id}$, we have
\begin{equation}\label{e:ELS}
{\rm Ent}_\gamma\big( \frac{v}{\gamma} \big) - \frac12 {\rm I}_\gamma\big( \frac{v}{\gamma} \big)
\le 
{\rm Ent}_{\gamma}\big( \frac{\gamma_\beta}{\gamma} \big) - \frac12 {\rm I}_{\gamma}\big( \frac{\gamma_\beta}{\gamma} \big). 
\end{equation}
\end{theorem}
It is clear that equality holds in \eqref{e:ELS} when $v = \gamma_\beta$ and an explicit calculation reveals that the constant in \eqref{e:ELS} is given by
\begin{equation*}\label{e:ConstLSI}
{\rm Ent}_{\gamma}\big( \frac{\gamma_\beta}{\gamma} \big) - \frac12 {\rm I}_{\gamma}\big( \frac{\gamma_\beta}{\gamma} \big)
=
-\frac{n}2
\big(
\log\, \beta - 1 + \frac1{\beta} 
\big). 
\end{equation*}

Shortly after \cite{ELS}, Mikulincer proved an analogous result for Talagrand's inequality. As we have already alluded to, it is not clear if one can derive deficit estimates for Talagrand's inequality directly from deficit estimates for the LSI in the setting of small covariance inputs, and the following result was obtained by utilising \emph{ideas from the approach} in \cite{ELS} (rather than a direct application of the results in \cite{ELS}).
\begin{theorem}[\cite{Miku}]\label{t:MikuWeak}
Let $\beta < 1$. 
For any non-negative $v$ such that $\int_{\mathbb{R}^n} v\, dx =1$, $\int_{\mathbb{R}^n} xv\, dx =0$ and ${\rm cov}\, (v) \le \beta {\rm id}$, 
\begin{equation}\label{e:MikuWeak}
\frac12 W_2( \gamma, v )^2 - {\rm Ent}_\gamma\big( \frac{v}{\gamma} \big) 
\le
- n 
\frac{2(1-\beta) + (\beta+1)\log\, \beta}{2(\beta-1)}.  
\end{equation}
\end{theorem}
To be precise, \cite[Theorem 3]{ELS} and \cite[Theorem 2]{Miku} give bounds  on the deficit in terms of the eigenvalues of the covariance matrix of the input from which we immediately obtain the above statements;  see the forthcoming Theorem \ref{t:ELS2} in Section \ref{section:consequences}.  

These two theorems naturally raise two questions. Firstly, is there a reasonable class of inputs with large covariance for which the deficit estimate holds? To this, Eldan--Lehec--Shenfeld and Mikulincer  gave  \textit{counterexamples} to the validity of \eqref{e:ELS} and \eqref{e:MikuWeak} in the case of large covariance and $\beta\gg1$ by mixing two gaussians. The covariance blows up whilst the deficit tends to zero. Nevertheless, the basic inequality \eqref{e:LSIDefGauss} for the LSI  remains true in the case $1<\beta < a$ and hence it is reasonable to expect some improvement like \eqref{e:ELS} and \eqref{e:MikuWeak} even when $\beta>1$ by imposing additional structure which excludes the counterexample in \cite{ELS,Miku}. 

Secondly, is the estimate \eqref{e:MikuWeak} sharp? Whilst the deficit estimate for the LSI in \eqref{e:ELS} is clearly sharp with equality when $v = \gamma_\beta$, the estimate \eqref{e:MikuWeak} is not sharp on such an input and this is suggestive that there is room for improvement. 

One of our aims in this paper is to address these questions by working in the framework of semi-log-subharmonicity and semi-log-convexity/concavity, and capitalising on invariance properties of such inputs under the Ornstein--Uhlenbeck flow. After introducing the framework within which we work, and its connection with covariance, we shall introduce our main results concerning stability estimates for hypercontractivity, Talagrand's inequality and the LSI.

\subsection{Our framework}
For $\beta>0$, we say that the twice differentiable function $v:\mathbb{R}^n\to(0,\infty)$ is $\beta$-\emph{semi-log-subharmonic} if 
\begin{equation}\label{e:BetaLogHar}
\Delta \log\, v \ge - \frac{n}\beta. 
\end{equation}
We were led to consider such input functions through work of Graczyk--Kemp--Loeb \cite{GKL} in which they proved so-called strong hypercontractivity for the dilation semigroup $T_sf(x) = f(e^{-s}x)$. More precisely, it was shown that
\begin{equation}\label{e:Janson}
\| T_s f \|_{L^q(\gamma)} \le \| f \|_{L^p(\gamma)}
\end{equation}
holds for $p,q,s$ satisfying $\frac{q}{p} = e^{2s}$ (we call such $s$ as ``Janson's time") under the condition that the input functions are log-subharmonic. The strong hypercontractivity \eqref{e:Janson} was firstly observed by Janson \cite{Janson}, namely he proved \eqref{e:Janson} for all $f:\mathbb{R}^{2d} \to \mathbb{C}$ such that $f(x,y) = F(x+iy)$ for some holomorphic function $F$ on $\mathbb{C}^d$, $d\ge1$. 
Notice that such $f:\mathbb{R}^{2d} \to \mathbb{C}$ are also harmonic $\Delta_{2d} f = 0$ and hence $P_sf = T_s f$. So Janson's result can be read as 
\begin{equation}\label{e:HC29Apr}
\big\| P_s f \|_{L^q(\gamma)} \le \| f \|_{L^p(\gamma)}
\end{equation}
for all holomorphic $f: \mathbb{R}^{2d} \to \mathbb{C} $ in the above sense and $\frac{q}{p} = e^{2s}$. Note that Janson's time (given by $\frac{q}{p} = e^{2s}$) improves Nelson's time (given by $\frac{q-1}{p-1} = e^{2s}$). Motivated by Janson's work, Graczyk--Kemp--Loeb \cite{GKL} extended the class of functions obeying \eqref{e:Janson} to log-subharmonic functions\footnote{If $f :\mathbb{R}^{2d} \to \mathbb{C}$ is holomorphic in the above sense, then $\log\, |f|$ becomes subharmonic, namely $|f|$ is log-subharmonic; see \cite[Example 2.1]{GKL}. Meanwhile it seems hopeless  to compare the result of Graczyk--Kemp--Loeb and the original hypercontractivity \eqref{e:HC29Apr} for $P_s$ since if $f$ is log-subharmonic and harmonic, then $f$ must be constant. }.


For $\beta>0$, we say that the twice differentiable function $v:\mathbb{R}^n\to(0,\infty)$ is $\beta$-\emph{semi-log-convex} if 
\begin{equation}\label{e:SemilogconvexScar}
\nabla^2 \log\, v \ge - \frac{1}{\beta}{\rm id},
\end{equation}
and $\beta$-\emph{semi-log-concave} if 
\begin{equation}\label{e:SemilogconcaveScar}
\nabla^2 \log\, v \le - \frac1\beta {\rm id}.
\end{equation}
Throughout the paper, our main results shall be stated in terms of either semi-log-subharmonicity or semi-log-convexity/semi-log-concavity. As we have already mentioned, these notions interact very nicely with the Ornstein--Uhlenbeck flow $P_s$ and this contributes to the robustness of our framework in terms of applications. Before exhibiting our main results, let us take a moment to compare these notions with each other and, to facilitate a comparison with related work (particularly \cite{ELS} and \cite{Miku}), with the size of covariance.

Firstly, for $n \geq 2$, it is clear by taking traces that $\beta$-semi-log-convexity is a stronger notion than $\beta$-semi-log-subharmonicity (and equivalent when $n=1$). Although less obvious, it is also the case that $\beta$-semi-log-convexity is also a stronger condition than ${\rm cov}\, (v) \ge \beta {\rm id}$. Analogous reverse statements also hold, and we summarise all of these facts as follows:
\begin{equation}\label{e:Relation27Apr1}
	\nabla^2 \log\, v \ge - \frac{1}{\beta} {\rm id}
	\Rightarrow
	\Delta \log\, v \ge - \frac{n}{\beta}
	\;\;\;
	{\rm and}
	\;\;\; 
	{\rm cov}\, (v) \ge \beta {\rm id},
\end{equation}
and 
\begin{equation}\label{e:Relation27Apr2}
	\nabla^2 \log\, v \le - \frac{1}{\beta} {\rm id}
	\Rightarrow
	\Delta \log\, v \le - \frac{n}{\beta}
	\;\;\;
	{\rm and}
	\;\;\; 
	{\rm cov}\, (v) \le \beta {\rm id}.
\end{equation}
Lemma \ref{l:LogPreserve} contains a more precise statement and we refer the reader forward to Section \ref{S2} for the proof as well.

We also have the following relation between the largeness of the covariance and semi-log-subharmonicity: 
\begin{equation}\label{e:Relation27Apr3}
	\Delta \log\, v \ge - \frac{n}{\beta}
	\Rightarrow
	{\rm tr}\, {\rm cov}\, (v) \ge n\beta. 
\end{equation}
Note that there is no such direct link between the size of $\Delta \log\, v$ and ${\rm cov}\, (v)$.
We also remark that whilst assumptions on covariance are usually weaker, if one restricts input functions to centered isotropic gaussians, then all of the above notions become equivalent: for all $a>0$, 
$$
	\nabla^2 \log\, \gamma_a \ge - \frac{1}{\beta} {\rm id}
	\Leftrightarrow
	\Delta \log\, \gamma_a \ge - \frac{n}{\beta}
	\Leftrightarrow 
	{\rm cov}\, (\gamma_a) \ge \beta {\rm id}
	\Leftrightarrow
	a\ge\beta.  
$$
In particular, these observations reveal that semi-log-subharmonicity and semi-log-convexity are natural alternatives to the largeness of the covariance when it comes to salvaging the failure of the deficit estimates for the LSI and Talagrand's inequality in \cite{ELS} and \cite{Miku}.

Looking more widely, notions of semi-log-convexity and semi-log-concavity have appeared in several different contexts. For instance, semi-log-convexity, which is closely related to the Li--Yau gradient estimate \cite{LiYau}, has recently featured in the work of Eldan--Lee \cite{EL} and Lehec \cite{Lehec}. These papers address a conjecture of Talagrand concerning a regularised version of the classical Markov inequality for gaussian measure $\gamma$ in which the inputs are restricted to those which arise from Ornstein--Uhlenbeck flow. Interestingly, the breakthrough work of Eldan--Lee \cite{EL} noticed that the natural framework for such a regularising phenomenon is wider than simply Ornstein--Uhlenbeck flows and they established a stronger result for semi-log-convex inputs. The refinement of Eldan and Lee's work by Lehec \cite{Lehec} also holds in the more general setting of semi-log-convex inputs, and this perspective has been further developed by Gozlan--Li--Madiman--Roberto--Samson \cite{GLMRS} as the theory begins to grow beyond the gaussian measure.

Regarding semi-log-concavity, for example, we mention the work on the stability of the entropy jump inequality due to Ball--Nguyen \cite{BaNg}, Bizeul \cite{Biz}, and the Pr\'{e}kopa--Leindler inequality due to Ball--B\"{o}r\"{o}czky \cite{BaBo10,BaBo11} and B\"{o}r\"{o}czky--De \cite{BoDe21}. We also mention work of the third author \cite{Tsu} and Indrei--Marcon \cite{InMa} where they also established an improved LSI under certain a log-convexity and log-concavity assumption (the shape of the stability estimates in these papers differs from \eqref{e:ELS} in the current paper and are independent results).
We also mention recent work due to Mikulincer--Shenfeld \cite{MikuShen} regarding preservation of log-concavity.

\subsection{Main results: hypercontractivity, LSI, and Talagrand's inequality}

We begin with our main result on hypercontractivity. 
\begin{theorem} \label{t:MainHyper}
Let $s>0$ and suppose $p,q \in \mathbb{R} \setminus \{1\}$ satisfy $\frac{q-1}{p-1} = e^{2s}$. Also, let $\beta > 0$ and assume that the twice differentiable function $v:\mathbb{R}^n \to (0,\infty)$ belongs to $L^2(\gamma_\beta^{-1})$.

(1) Suppose $1<p<q<\infty$. When $\beta > 1$ let $v$ be $\beta$-semi-log-subharmonic, and when $\beta < 1$ let $v$ be $\beta$-semi-log-concave. Then
\begin{equation}\label{e:RegHC}
\big\| P_s\big[ \big(\frac{v}{\gamma}\big)^\frac1p\big] \big\|_{L^q(\gamma)}
\le 
\big\| P_s\big[ \big(\frac{\gamma_\beta}{\gamma}\big)^\frac1p\big] \big\|_{L^q(\gamma)}
\bigg( \int_{\mathbb{R}^n} \frac{v}{\gamma} \, d\gamma\bigg)^\frac1p.
\end{equation}

(2) Suppose $-\infty <q < p< 1$ and $p\neq 0, 1-e^{-2s}$.
\begin{enumerate}
\item[(i)] 
If $p q >0$, $\beta > 1$ and $v$ is $\beta$-semi-log-subharmonic, then 
\begin{equation}\label{e:RegRevHC}
\big\| P_s \big[ \big(\frac{v}{\gamma}\big)^\frac1p \big] \big\|_{L^q(\gamma)} \ge \big\| P_s \big[ \big( \frac{\gamma_\beta}{\gamma} \big)^\frac1p \big] \big\|_{L^q(\gamma)}  \bigg( \int_{\mathbb{R}^n} \frac{v}{\gamma}\, d\gamma \bigg)^\frac1p. 
\end{equation}
\item[(ii)]
If $p q <0$, $0<\beta < 1$ and $v$ is $\beta$-semi-log-concave, then \eqref{e:RegRevHC} holds. 
\end{enumerate}
\end{theorem}
Before proceeding to our main results concerning the LSI and Talagrand's inequality, we make a few remarks about the above result. Firstly, it is an easy exercise to check that the constant in \eqref{e:RegHC} and \eqref{e:RegRevHC} can be explicitly computed as 
$$
\big\| P_s\big[ \big(\frac{\gamma_\beta}{\gamma}\big)^\frac1p\big] \big\|_{L^q(\gamma)}
= 
\beta^{\frac {n}{2p'}}\beta_s^{-\frac{n}{2q'}}\in(0,1),
$$
and $\beta_s$ is given in \eqref{e:beta2}. Moreover equality is established when $v =\gamma_\beta$. 

Secondly, we recommend that the reader does not dwell on the condition $v \in L^2(\gamma_\beta^{-1})$. This is imposed in order to justify various technicalities in the proof and we will exhibit a rich class of functions satisfying these conditions in Corollary \ref{c:FPversion} below by considering non-negative solutions to Fokker--Planck equations.

Also, naively, one may expect that the assumption on $v$ in the case $\beta < 1$ may be weakened to \emph{$\beta$-semi-log-superhamonicity}, that is
$$
\Delta \log\, v \le - \frac{n}{\beta}.
$$
We leave this problem open and refer the reader forward to the end of this section for further discussion on this. From a different viewpoint, the case $pq < 0$ appears to be special in the sense that the standard argument from which hypercontractivity is derived from the LSI seems to break down when applied to the corresponding deficit estimates.

As a consequence of Theorem \ref{t:MainHyper}, we improve the LSI in the same spirit.
\begin{theorem} \label{t:MainLSI}
Let $\beta > 0$ and assume that the twice differentiable function $v:\mathbb{R}^n \to (0,\infty)$ belongs to $L^2(\gamma_\beta^{-1})$.

(1) If $\beta > 1$ and $v$ is $\beta$-semi-log-subharmonic, then \eqref{e:ELS} holds.

(2) If $0<\beta < 1$ and $v$ is $\beta$-semi-log-concave, then \eqref{e:ELS} holds.
\end{theorem}
In the case $\beta >1$, thanks to \eqref{e:Relation27Apr3}, Theorem \ref{t:MainLSI}(1) states that the deficit estimate \eqref{e:ELS} holds in the case that the trace of the covariance of the input is large enough if one imposes additional information on $\Delta \log\, v$.  This can be viewed as a complementary result to the small-covariance result of Eldan--Lehec--Shenfeld in Theorem \ref{t:ELS}\footnote{After we prepared an earlier version of this paper, it was pointed out to us that under the stronger assumption that $v$ is $\beta$-semi-log-convex, Theorem \ref{t:MainLSI}(1) follows from arguments in \cite{ELS}. More precisely, one should combine \cite[Theorem 1]{ELS} with arguments similar to the one deriving \cite[Theorem 3]{ELS} by noticing that the Fisher information matrix can be realised as 
$$
-\int_{\mathbb{R}^n}  v \nabla^2 \log\, v\, dx. 
$$
On the other hand, under our weaker assumption of $\beta$-semi-log-subharmonicity, it does not seem to be obvious if one is able to run such an argument.} .

In the case of $0<\beta <1$, thanks to \eqref{e:Relation27Apr2}, Theorem \ref{t:MainLSI}(2) follows from Theorem \ref{t:ELS}. Nevertheless, we include this in our statement of Theorem \ref{t:MainLSI} to emphasise that our approach naturally yields results in both regimes $\beta < 1$ and $\beta > 1$. Also, we remark that it does not seem to be obvious if one could derive some improved form of hypercontractivity from Theorem \ref{t:ELS} (for $0<\beta < 1$). This appears to be because the assumption on the covariance in Theorem \ref{t:ELS} is somewhat rigid with regard to running the standard argument from which one derives classical hypercontractivity from the LSI{\footnote{In fact, if one attempts to run the standard argument, then eventually one would need to apply Theorem \ref{t:ELS} with an input form of $v= \gamma P_s\big[ f^\frac1p \big]^q$. For that purpose one has to ensure that the covariance of $\gamma P_s\big[ f^\frac1p \big]^q $ is small enough. However, it is not clear if such a non-linear flow preserves the covariance or not.}}. On the other hand, our semi-log-subharmonicity and semi-log-convexity/semi-log-concavity assumption appears to be more flexible and this can be considered an advantage of our framework. 
In fact, it is robust enough to improve other functional inequalities such as the hypercontractivity inequality for the Hamilton--Jacobi equation, the dual form of Talagrand's inequality, as well as the Poincar\'{e} and Beckner inequalities; see Section \ref{S4}.

In our next result, we improve Miklincer's deficit estimate of Talagrand's inequality \eqref{e:MikuWeak} under somewhat stronger assumptions on the inputs.
\begin{theorem}\label{t:MainTalagrand}
	Let $\beta>0$. Suppose the twice differentiable $v:\mathbb{R}^n\to (0,\infty)$ satisfies $\int_{\mathbb{R}^n} v\, dx = 1$ and $\int_{\mathbb{R}^n} |x|^2 v\, dx < \infty$. 

(1) If $\beta >1$ and $v$ satisfies 
		$$
		0\ge \nabla^2 \log\, v \ge - \frac1{\beta}{\rm id},
		$$
then we have 
\begin{equation}\label{e:TIMWeak}
			\frac12 W_2( \gamma, v )^2 - {\rm Ent}_\gamma\big( \frac{v}{\gamma} \big) 
\le
\frac12 W_2( \gamma, \gamma_\beta )^2 - {\rm Ent}_\gamma\big( \frac{\gamma_\beta}{\gamma} \big). 
\end{equation}
		
(2) If $0<\beta < 1$ and $v$ is $\beta$-semi-log-concave, then 
we have \eqref{e:TIMWeak}.
\end{theorem}
A direct calculation shows that the constant in \eqref{e:TIMWeak} is explictly given by
\begin{equation*}\label{e:ConstTal}
		\frac12 W_2( \gamma, \gamma_\beta )^2 - {\rm Ent}_\gamma\big( \frac{\gamma_\beta}{\gamma} \big) 
		= 
		n\big( 1 + \frac12 \log\, \beta - \sqrt{\beta} \big).
\end{equation*}
For any $\beta > 0$ we have
$$
1 + \frac12 \log\, \beta - \sqrt{\beta}
<
- \frac{2(1-\beta) + (\beta+1)\log\, \beta}{2(\beta-1)}
$$
and so it is clear that \eqref{e:TIMWeak} improves upon \eqref{e:MikuWeak} for $\beta$-semi-log-concave inputs and $0< \beta < 1$. As a further appealing feature, inequality \eqref{e:TIMWeak} is clearly sharp since equality holds when $v = \gamma_\beta$. 

We have already remarked that the standard argument for deriving Talagrand's inequality from the LSI is not applicable in the setting of deficit estimates under a small-covariance assumption, and since the same appears to be true in our framework, we must prove Theorem \ref{t:MainTalagrand} directly without appealing to Theorem \ref{t:MainLSI} (or Theorem \ref{t:MainHyper}). Interestingly, our proofs of Theorems \ref{t:MainHyper} and \ref{t:MainTalagrand} are completely different; in the former case we use a flow monotonicity approach using Fokker--Planck flows, and in the latter case we use mass transportation. We refer the reader forward to Section \ref{section:overview} for further details of our methodology, as well as the preceding part of Section \ref{section:consequences} in which we give a number of further extended remarks on our main results Theorems \ref{t:MainHyper}--\ref{t:MainTalagrand}. Before that, we mention some related open problems which we consider to be interesting and were unable to address in the present work.

\subsection{A few open problems} 
It seems reasonable to expect that \eqref{e:TIMWeak} holds when $0<\beta <1$ and ${\rm cov}\, (v) \le \beta {\rm id}$. If true, this would simultaneously extend both Mikulincer's result in Theorem \ref{t:MikuWeak} and our result in Theorem \ref{t:MainTalagrand}(2) in what seems to be a natural way\footnote{This problem was raised by Yair Shenfeld in discussions which evolved after the first version of this paper was released.}.

In the context of hypercontractivity and/or the LSI, if the inequality is valid under $\beta$-semi-log-concavity then it seems reasonable to hope that it is also true under the weaker assumption of $\beta$-semi-log-superhamonicity\footnote{Similarly, one could hope to extend results which are valid when ${\rm cov}\, (v) \le \beta {\rm id}$ to the wider class of inputs for which
$$
{\rm tr}\, {\rm cov}\, (v) \le n\beta.
$$}.
Unfortunately, it appears that our methods do not allow us to address this. In particular, our proof of Theorem \ref{t:MainHyper} critically used the fact that semi-log-subhamonicity is preserved under the Ornstein--Uhlenbeck flow. On the other hand, as a simple example shows, semi-log-superhamonicity is not always preserved under the flow; see Lemma \ref{l:LogPreserve} and the subsequent remark. 

\section{Consequences and remarks on our main results} \label{section:consequences}

\subsection{Regularisation via Fokker--Planck flow}
Here we provide a concrete and rich class of functions satisfying the assumptions in Theorems \ref{t:MainHyper}--\ref{t:MainTalagrand}. The class we have in mind consists of non-negative functions which have been regularised by the Fokker--Planck equation. 
Instances where functional inequalities have been improved by restricting attention to inputs which arise as the evolution of a certain diffusion equation can be seen in several lines of research. 
For instance,  Bennett, Carbery, Christ and Tao \cite{BCCT} (see also \cite{BN}) addressed  regularisation of the Brascamp--Lieb inequality\footnote{Here we mean the Brascamp--Lieb inequality which generalises the H\"older, Loomis--Whitney and Young convolution inequalities arising in \cite{BraLi_Adv}, rather than the concentration inequality for log-concave probability distributions from \cite{BraLi}. On the other hand, certain ideas from the latter paper concerning the preservation of log-concavity under certain diffusion processes will in fact play an important role in the present paper as well -- see the forthcoming Lemma \ref{l:LogPreserve}.} via non-isotropic heat equations.  Also, the aforementioned resolution by Eldan and Lee \cite{EL} of the gaussian version of Talagrand's conjecture may be viewed as an improvement of Markov's inequality due to the regularising nature of the Ornstein--Uhlenbeck flow.  

We introduce the Fokker--Planck equation. 
For $\beta >0 $,  the Fokker--Planck equation with a diffusion speed $\beta$, or $\beta$-Fokker--Planck equation for short, is 
\begin{equation}\label{e:FP}
\partial_t v_t = \mathcal{L}_{\beta}^* v_t := \beta \Delta v_t + x\cdot  \nabla v_t  + n v_t,\;\;\; (t,x) \in (0,\infty)\times \mathbb{R}^n, 
\end{equation}
where $\mathcal{L}_{\beta}^*$ stands for the dual of $\mathcal{L}_\beta$ with respect to $L^2({d}x)$. 
As is well-known, the Fokker--Planck equation is closely related to the gaussian heat equation. In fact,  $v_t$ solves \eqref{e:FP} if and only if $u_t$ solves $\partial_t u_t = \mathcal{L}_\beta u_t$,  under the transformation $u_t = \frac{v_t}{\gamma_\beta}$. 
The relevancy of the Fokker--Planck equation to hypercontractivity and the LSI can be seen by focusing on its stable solution or equilibrium state given by $v_t = \gamma_\beta$.  
In fact,  the extremiser of hypercontractivity and the LSI given by $f = \frac{\gamma}{\gamma}$ can be regarded as the equilibrium state for the $1$-Fokker--Planck equation.
This heuristic explains why the \textit{heat-flow monotonicity} scheme is well-fitted to the theory of hypercontractivity and the LSI, see \cite{BGL,LedouxICM} for a comprehensive treatment on heat flow  and \cite{OttoVil} for the Fokker--Planck flow. 

We now introduce a class of functions \textit{regularised} by the Fokker--Planck equation. 
For $\beta>0$ and initial data any non-negative finite measure $d\mu$, let $v_{*}$ be the corresponding $2\beta$-Fokker--Planck solution: 
\begin{equation}\label{e:v*}
\partial_t v_{*} = \mathcal{L}_{2\beta}^* v_{*},\;\;\; (t,x) \in (0,\infty)\times \mathbb{R}^n,\;\;\; v_{*}(0,x) = d\mu(x). 
\end{equation}
Then for a fixed time $t_* := \frac12 \log\, 2 >0$ we introduce our regularised class 
$$
{\rm FP}(\beta) := \{  v = v_{*}(t_{*},\cdot) : \,  \text{$v_*$ is non-negative solution to \eqref{e:v*}} \}. 
$$
Our somewhat artificial choices of $2\beta$ and $t_*$ may raise eyebrows. They are natural in the sense that $\gamma_\beta$, the stationary solution of the $\beta$-Fokker--Planck equation, can be represented by $ \gamma_\beta =  v_{*}(t_{*},\cdot) $ with initial data taken to be the Dirac delta measure supported at the origin.  In particular, $ \gamma_\beta $,  the extremiser of the inequality \eqref{e:ELS}, belongs to the class ${\rm FP}(\beta)$.  
As other simple properties of ${\rm FP}(\beta)$,  we have the nesting property ${\rm FP}(\beta_1) \supseteq {\rm FP}(\beta_2)$ for $\beta_1\le \beta_2$. 
Also for $a>0$, $\gamma_a$ belongs to ${\rm FP}(\beta)$ if and only if $a\ge \beta$.  In particular, $\gamma$, which is the extremiser of \eqref{e:HCClassic} and \eqref{e:LogSob}, is excluded from the class ${\rm FP}(\beta)$ in the case $\beta>1$.  
Finally, we observe that if $v$ is the class of ${\rm FP}(\beta)$, then $v$ earns semi-log-convexity\footnote{For a precise statement and proof, see Lemma \ref{l:LogPreserve}.}: 
$$
v \in {\rm FP}(\beta)
\Rightarrow 
\nabla^2 \log\, v \ge - \frac{1}{\beta}{\rm id}.
$$
Hence, as a corollary of Theorems \ref{t:MainHyper} and \ref{t:MainLSI} (with a certain amount of technical justification), we obtain the following. 
\begin{corollary}\label{c:FPversion}
Let $\beta > 1$ and $v\in {\rm FP}(\beta)$. Then the LSI deficit estimate  \eqref{e:ELS} holds. Also, if
$s>0$ and $1<p<q<\infty$ satisfy $\frac{q-1}{p-1}= e^{2s}$, then the hypercontractivity deficit estimate \eqref{e:RegHC} holds. 
\end{corollary}

\subsection{Reducing dimensional dependence}
Here we state versions of our main results in which the dependence on the dimension of the ambient space is weakened. For this we first recall \cite[Theorem 3]{ELS} and \cite[Theorem 2]{Miku} in full generality (Theorems \ref{t:ELS} and \ref{t:MikuWeak} presented earlier being special cases). 
\begin{theorem}\label{t:ELS2}
Suppose the non-negative function $v$ satisfies $\int_{\mathbb{R}^n} v\, dx =1$, and let $\beta_1,\ldots, \beta_n$ be the eigenvalues of ${\rm cov}\, (v)$. 

(1) \cite[Theorem 3]{ELS} Then 
\begin{equation}\label{e:ELSgene}
{\rm Ent}_\gamma \big( \frac{v}{\gamma} \big)
\le 
\frac12 {\rm I}_\gamma\big( \frac{v}{\gamma} \big)
- 
\frac12 
\sum_{ \substack{ i=1,\ldots, n: \\ \beta_i \le 1 } } \big(\log\, \beta_i - 1 + \frac{1}{\beta_i}\big).
\end{equation}

(2) \cite[Theorem 2]{Miku} If, in addition, $\int_{\mathbb{R}^n} |x|^2 v\, dx<\infty$ and $\int_{\mathbb{R}^n} xv\, dx =0$,  then 
\begin{equation}\label{e:Miku}
\frac12 W_2( \gamma, v )^2 - {\rm Ent}_\gamma\big( \frac{v}{\gamma} \big) 
\le
-
\sum_{ \substack{  i=1,\ldots, n : \\ \beta_i <1 } } \frac{ 2(1-\beta_i) + (\beta_i +1) \log\, \beta_i }{2( \beta_i - 1)}. 
\end{equation}
\end{theorem}
Theorems \ref{t:MainHyper}--\ref{t:MainTalagrand} may be improved in the same spirit as follows. For the statement, we introduce the notation
$$
\gamma_B(x):= \frac{1}{ ({\rm det}\, (2\pi B))^\frac12 }e^{ - \frac12 \langle x, B^{-1}x\rangle },
$$
where $B$ is a positive definite symmetric matrix $\mathbb{R}^n$.
\begin{corollary}\label{Cor:Matrix}
Let $\beta_1,\ldots, \beta_n >0$ be the eigenvalues of the positive definite symmetric matrix $B$, and suppose the twice differentiable function $v:\mathbb{R}^n \to (0,\infty)$ belongs to $L^2(\gamma_B^{-1})$.

(1) If 
\begin{equation}\label{e:Semilogconvex}
\nabla^2 \log\, v \ge - B^{-1}
\end{equation}
then we have 
\begin{equation}\label{e:HCMatrixConv}
\big\| P_s\big[ \big( \frac{v}{\gamma}\big)^\frac1p \big]  \big\|_{ L^q(\gamma ) }
\le 
\prod_{ \substack{  i=1,\ldots, n : \\ \beta_i \ge 1 }} \beta_i^{ \frac{1}{2p'} } ( \beta_i )_s^{-\frac{1}{2q'}} 
\bigg( \int_{\mathbb{R}^n} \frac{v}{\gamma}\, d\gamma \bigg)^\frac1p
\end{equation}
and 
\begin{equation}\label{e:LSIMatrixConv}
{\rm Ent}_{\gamma} \big( \frac{v}{\gamma} \big)
\le 
\frac12 
{\rm I}_\gamma \big( \frac{v}{\gamma} \big)
-
\frac{1}{2}
\sum_{ \substack{ i=1,\ldots,n :\\ \beta_i \ge1 } } \big(\log\, \beta_i -1 + \frac1{\beta_i}\big). 
\end{equation}
If, in addition, $\int_{\mathbb{R}^n} |x|^2 v\, dx<\infty$ and $\nabla^2 \log\, v \le 0$, then
\begin{equation}\label{e:TIMatrixConv}
\frac12 W_2(\gamma,v)^2 - {\rm Ent}_\gamma\big( \frac{v}{\gamma} \big)
\le 
\sum_{ \substack{ i=1,\ldots,n :\\ \beta_i \ge1 } } (  1 + \frac12 \log\, \beta_i - \sqrt{\beta_i} ). 
\end{equation}

(2)  If
\begin{equation}\label{e:Semilogconcave}
\nabla^2 \log\, v \le - B^{-1}
\end{equation}
then we have 
\begin{equation}\label{e:HCMatrixConc}
\big\| P_s\big[ \big( \frac{v}{\gamma}\big)^\frac1p \big]  \big\|_{ L^q(\gamma ) }
\le 
\prod_{ \substack{  i=1,\ldots, n : \\ \beta_i \le 1 }} \beta_i^{ \frac{1}{2p'} } ( \beta_i )_s^{-\frac{1}{2q'}} 
\bigg( \int_{\mathbb{R}^n} \frac{v}{\gamma}\, d\gamma \bigg)^\frac1p
\end{equation}
and 
\begin{equation}\label{e:LSIMatrixConc}
{\rm Ent}_{\gamma} \big( \frac{v}{\gamma} \big)
\le 
\frac12 
{\rm I}_\gamma \big( \frac{v}{\gamma} \big)
-
\frac{1}{2}
\sum_{ \substack{ i=1,\ldots,n :\\ \beta_i \le1 } } \big(\log\, \beta_i -1 + \frac1{\beta_i}\big). 
\end{equation}
If, in addition, $\int_{\mathbb{R}^n} |x|^2 v\, dx<\infty$, then 
\begin{equation}\label{e:TIMatrixConc}
\frac12 W_2(\gamma,v)^2 - {\rm Ent}_\gamma\big( \frac{v}{\gamma} \big)
\le 
\sum_{ \substack{ i=1,\ldots,n :\\ \beta_i \le1 } } (  1 + \frac12 \log\, \beta_i - \sqrt{\beta_i} ).
\end{equation}
\end{corollary}
We clarify that in the above statement, following \eqref{e:beta2}, we write $(\beta_i)_s = 1 + (\beta_i - 1)\frac{q}{p}e^{-2s}$ for each eigenvalue $\beta_i$.

Clearly, equality holds in each of \eqref{e:HCMatrixConv}--\eqref{e:TIMatrixConc} when $v= \gamma_B$. We also remark that although Corollary \ref{Cor:Matrix} is clearly a stronger result compared with Theorems \ref{t:MainHyper}--\ref{t:MainTalagrand}, interestingly we will establish Corollary \ref{Cor:Matrix} by combining the one-dimensional result in Theorems \ref{t:MainHyper}--\ref{t:MainTalagrand} with a standard tensorisation argument. In particular, it is an appealing feature of our framework that the semi-log-convexity/semi-log-concavity assumption is very well suited to such a tensorisation argument.

\subsection{Overview of our approach and structure of the paper} \label{section:overview}
To prove the statement regarding the hypercontractivity inequality in Theorem \ref{t:MainHyper}, we use a flow monotonicity argument based on the Fokker--Planck equation. To be precise, we shall show that
\begin{equation*}
t \mapsto \int  P_s \big[ \big( \frac{v_t}{\gamma}\big)^\frac1p \big]^q \, d\gamma
\end{equation*}
is non-decreasing on $(0,\infty)$, where $v_t$ is the evolution of the given input $v$ along the $\beta$-Fokker--Planck equation. Comparing the above time-dependent functional at $t=0$ and $t = \infty$ generates the desired inequality \eqref{e:RegHC}. As a crucial step in the proof of the monotonicity of the above functional, we shall use a quantitative version of the fact that semi-log-subharmonicity and semi-log-convexity/concavity are preserved under Fokker--Planck flows (see Lemma \ref{l:LogPreserve}). The improved LSI in Theorem \ref{t:MainLSI} will be derived from \eqref{e:RegHC} by following the well-known differentiation argument (due to Gross \cite{Gross}). 

Our proof of Theorem \ref{t:MainHyper} takes a great deal of inspiration from work by Bennett--Carbery--Christ--Tao \cite{BCCT} on regularised Brascamp--Lieb inequalities and  \cite{BB} on the Young convolution inequality (both of which use more classical heat flows), as well as the approach to \eqref{e:HCClassic} taken by Aoki \emph{et al.} \cite{ABBMMS} (using Ornstein--Uhlenbeck flow); we refer the reader forward to Sections \ref{S3.4} and \ref{section:furtherremarksflow} for further elaboration on this, including some discussion on our reasons for employing Fokker--Planck flow. We also remark that although the arguments of Eldan--Lehec--Shenfeld \cite{ELS} are less close to those in the current paper, they also employed a flow (the F\"ollmer process from stochastic analysis).

As mentioned already, our proof of Theorem \ref{t:MainTalagrand} proceeds differently and we employ a mass transportation argument. Our approach here was inspired by the mass-transport proof of the classical version of the LSI due to Cordero-Erausquin \cite{Cor}.

\begin{organisation}
In Section \ref{S2}, we state and prove some results on our framework of semi-log-subharmonicity and semi-log-convexity/concavity; in particular, relations between these notions and covariance, as well as preservation properties under Fokker--Planck flows. Finally, a technical result concerning regularity and decay of Fokker--Planck flows is given and will be used to justify various interchanges of limits in the proof of Theorem \ref{t:MainHyper}.

In Section \ref{S3}, we will prove Theorems \ref{t:MainHyper} and \ref{t:MainLSI}. We also establish Corollary \ref{c:FPversion} and the statements in Corollary \ref{Cor:Matrix} concerning hypercontractivity and the LSI. 

In Section \ref{S3.5}, we prove Theorem \ref{t:MainTalagrand} as well as the statements in Corollary \ref{Cor:Matrix} regarding Talagrand's inequality. 

In Section \ref{S4}, we will apply these results to improve the hypercontractivity inequality for the Hamilton--Jacobi equation, the dual form of Talagrand's inequality, and the Poincar\'{e} and Beckner inequalities. 

Finally, in Section \ref{S5} we will explore the possibility of obtaining deficit estimates for hypercontractivity and the LSI in the setting of more general measure spaces. 

\end{organisation}

\section{Preliminaries}\label{S2}
We begin with the following lemma which explains relations between covariance, semi-log-subharmonicity, semi-log-convexity/concavity and the class ${\rm FP}(\beta)$. 
Although these results may be considered well-known, we provide a proof. 
\begin{lemma}\label{l:LogConVari}
Let $\beta>0$. 
\begin{enumerate}
\item
If $ v \in {\rm FP}(\beta)$, then 
\begin{equation}\label{e:LogConvexv*}
\nabla^2\log\, v \ge -\frac1\beta {\rm id}. 
\end{equation}
\item
Let $v:\mathbb{R}^n\to (0,\infty)$ be twice differentiable, normalised by $\int_{\mathbb{R}^n} v\, dx =1$, and satisfy the decay condition 
\begin{equation}\label{e:DecayTech}
\lim_{|x|\to\infty}|x|v(x) = \lim_{|x|\to\infty}|\nabla v(x)| =0. 
\end{equation}
Then we have \eqref{e:Relation27Apr1}--\eqref{e:Relation27Apr3}. 
\end{enumerate}
\end{lemma}

\begin{proof}[Proof of Lemma \ref{l:LogConVari}]
The assertion (1) follows from the fact that 
\begin{equation}\label{e:NatuLogConv}
\nabla^2 \log\, v(t,x) \ge - \frac{ 1 }{ 1-e^{-2t} }\frac1{\beta} {\rm id},\;\;\; (t,x) \in (0,\infty)\times\mathbb{R}^n
\end{equation}
for any $\beta>0$ and any $\beta$-Fokker--Planck solution $\partial_t v_t = \mathcal{L}_\beta^* v_t$ with any non-negative initial data; see \cite[(1) on p.3]{GLMRS} or \cite[Lemma 1.3]{EL}. 
In fact, this shows that for $v_*$, a solution to \eqref{e:v*},  and $t_* = \frac12\log\, 2$, 
$$
\nabla^2 \log\, v_*(t_*,x) \ge - \frac{ 1 }{ 1-e^{-2t_*} }\frac1{2\beta} {\rm id} = - \frac1\beta {\rm id}. 
$$

Let us show \eqref{e:Relation27Apr1} next. 
By translation invariance, without loss of generality we may assume $\int_{\mathbb{R}^n} x_i\, dv = 0$, $1\le i\le n$, and hence 
\begin{equation}\label{e:Cov}
{\rm cov}\, (v) = \bigg(\int_{\mathbb{R}^n} x_ix_j \, dv  \bigg)_{1\le i,j\le n}  = \int_{\mathbb{R}^n} x\otimes x \, dv.
\end{equation}
With this in mind, we integrate by parts
\begin{align*}
\int_{\mathbb{R}^n} x_ix_j \, dv
=&
\int_{\mathbb{R}^n} x_ix_j \frac{v}{\gamma_\beta}\, d\gamma_\beta 
=
\beta \delta_{i,j} + \beta^2 \int_{\mathbb{R}^n} \partial_{ij} \big( \frac{v}{\gamma_\beta} \big) \, d\gamma_\beta. 
\end{align*}
Here we used the decay assumption \eqref{e:DecayTech} and the normalisation. 
We next notice that 
$$
\partial_{ij}f = f \partial_{ij} \log\, f +  f( \partial_i \log\, f)( \partial_j  \log\, f) 
$$ 
from which we see that 
\begin{align*}
{\rm cov}\, (v) 
=& 
\beta {\rm id} 
+ 
\beta^2
\int_{\mathbb{R}^n}   \nabla^2 \log\, \frac{v}{\gamma_\beta} \, dv 
+ 
\beta^2 
\int_{\mathbb{R}^n}  \big( \nabla \log\, \frac{v}{\gamma_\beta}\big) \otimes \big( \nabla \log\, \frac{v}{\gamma_\beta}\big)  \, dv. 
\end{align*}
For the second term, the $-\frac1{\beta}$-semi-log-convexity assumption of $v$ yields that 
$$
\nabla^2 \log\, \frac{v}{\gamma_\beta}\ge 0
$$
whilst we clearly have non-negativity of the third term since 
$$
\big\langle x,  \big( \nabla \log\, \frac{v}{\gamma_\beta}\big) \otimes \big( \nabla \log\, \frac{v}{\gamma_\beta}\big) x \big\rangle
= 
\big|x \cdot \nabla \log\, \frac{v}{\gamma_\beta}\big|^2 \ge 0
$$
for all $x \in \mathbb{R}^n$.  This completes the proof of \eqref{e:Relation27Apr1}. 
By the similar argument, one can prove \eqref{e:Relation27Apr3}. 

For the proof of \eqref{e:Relation27Apr2}, we regard the assumption $\nabla^2 \log\, v \le -\frac1\beta {\rm id}$ as the curvature-dimension condition $CD(\frac1\beta, \infty)$, in which case it follows from \cite[Corollary 4.8.2, p.212]{BGL} that we have the Poincar\'{e} inequality 
\begin{equation}\label{e:BetaPoi}
\int_{\mathbb{R}^n} |\varphi|^2\,  dv \le \beta \int_{\mathbb{R}^n} |\nabla \varphi|^2\, dv,\;\;\; {\rm for}\;\;\; \int_{\mathbb{R}^n} \varphi \, dv =0. 
\end{equation}
Again, without loss of generality, we suppose $\int_{\mathbb{R}^n} x_i\, dv =0$. From \eqref{e:Cov} we obtain
$$
\langle \theta , {\rm cov}\, (v)\theta\rangle 
=
\int_{\mathbb{R}^n} | \langle \theta , x\rangle|^2\, dv. 
$$
Since $\int_{\mathbb{R}^n} x_i\, dv =0$, we may apply \eqref{e:BetaPoi} with $\varphi(x) := \langle \theta, x \rangle$ to see that 
$$
\langle \theta , {\rm cov}\, (v)\theta\rangle 
\le \beta 
\int_{\mathbb{R}^n} | \nabla_x \big( \langle \theta , x\rangle\big) |^2\, dv
=
\beta|\theta|^2
$$
and hence ${\rm cov}\,(v) \le \beta {\rm id}$. 
\end{proof}

In the proof of Theorem \ref{t:MainHyper}, we will flow input function $v$ via the $\beta$-Fokker--Planck equation:
\begin{equation}\label{e:BetaFPv0}
\partial_t v_t =\mathcal{L}_\beta^* v_t,\;\;\; (t,x)\in (0,\infty)\times\mathbb{R}^n,\;\;\; v_0 = v. 
\end{equation}
A crucial property in our argument is the fact that the Fokker--Planck flow preserves semi-log-subharmonicity as well as semi-log-convexity and semi-log-concavity in the following sense. Again, we are not claiming originality here and we include a proof for the sake of completeness.
\begin{lemma}\label{l:LogPreserve}
Let $\beta>0$ and $v_t$ be a $\beta$-Fokker--Planck solution of \eqref{e:BetaFPv0} with twice differentiable initial data $v: \mathbb{R}^n\to(0,\infty)$.  
Then 
\begin{equation}\label{e:ConvexPresereve}
\nabla^2 \log\, v \ge -\frac1\beta {\rm id} \Rightarrow \nabla^2 \log\, v_t \ge -\frac1\beta {\rm id},\;\;\; t>0,
\end{equation}
and 
\begin{equation}\label{e:ConcavePresereve}
\nabla^2 \log\, v \le -\frac1\beta {\rm id} \Rightarrow \nabla^2 \log\, v_t \le -\frac1\beta {\rm id},\;\;\; t>0.
\end{equation}
Moreover, we have that 
\begin{equation}\label{e:PreSubHar}
\Delta \log\, v \ge - \frac{n}{\beta}
\Rightarrow 
\Delta \log\, v_t \ge - \frac{n}{\beta},
\;\;\;
t>0.
\end{equation}

\end{lemma}

\begin{proof}
In this proof we make use of the fact that the solution to \eqref{e:BetaFPv0} has the explicit form\footnote{The uniqueness of the solution is a consequence of Widder's theorem \cite{Widder}.}
\begin{equation} \label{e:FPsolution}
v_t(x) = \frac{1}{ ( 2\pi \beta( 1-e^{-2t} ) )^\frac{n}{2} } \int_{\mathbb{R}^n} e^{ - \frac{|  x -  e^{-t}y |^2}{ 2\beta(1-e^{-2t}) }} v_0(y)\, dy.
\end{equation}
This may be rewritten as
\begin{equation}\label{e:vtForm18Feb}
v_t(x) =  e^{ - \frac{|x|^2}{2\beta} } \int_{\mathbb{R}^n} \phi(x,w)\, d\mu(w),
\end{equation}
where
\begin{align*}
\phi(x,w) & := e^{\frac{|w+ e^{-t}x|^2}{2\beta}} v_0(w +e^{-t}x),\\ 
d\mu(w) & :=  \frac{1}{ ( 2\pi \beta( 1-e^{-2t} ) )^\frac{n}{2} } e^{ - \frac{1}{2\beta(1 - e^{-2t})}  |w|^2 }\, dw.
\end{align*}
From the assumption $\nabla^2 \log\, v_0 \ge - \frac1\beta {\rm id}$, we see that 
$$
\nabla^2_x  \log\, \phi(x,w) 
= 
\frac{e^{-2t}}{\beta} {\rm id} + e^{-2t} \nabla^2 \log\, v_0 (w + e^{-t}x)
\ge 
0.
$$
Using the fact that log-convexity is preserved under superpositions, we obtain \eqref{e:ConvexPresereve}.

The preservation of semi-log-concavity in \eqref{e:ConcavePresereve} is more subtle and can be found, for instance, in \cite{BraLi,IPT}. For the sake of completeness, we give a proof of \eqref{e:ConcavePresereve} here. 
We use the explicit form of the solution \eqref{e:FPsolution} to write 
\begin{align*}
v_t(x) 
& = 
e^{ - \frac{|x|^2}{2\beta }}\int_{\mathbb{R}^n} \gamma_{ \beta (1-e^{-2t}) }(w) e^{\frac{|w+ e^{-t}x|^2}{2\beta}} v_0(w +e^{-t}x)\, dw\\
& =: 
e^{ - \frac{|x|^2}{2\beta }}\int_{\mathbb{R}^n} \psi_x(w)\, dw,\;\;\; 
\psi_x(w)
:= 
\gamma_{ \beta (1-e^{-2t}) }(w) e^{\frac{|w+ e^{-t}x|^2}{2\beta}} v_0(w +e^{-t}x)\\
& =: 
e^{ - \frac{|x|^2}{2\beta }} \Psi(x),
\end{align*}
where we suppress the dependence on $t$ in $\psi_x$ and $\Psi$.

It is sufficient to show the log-concavity of $\Psi$. To this end, we focus on 
$$
\Psi(\lambda x_0 + (1-\lambda)x_1)
= 
\int_{\mathbb{R}^n} \psi_{\lambda x_0 + (1-\lambda)x_1}(y)\, dy. 
$$
We claim that 
\begin{equation}\label{e:PLAssump}
\psi_{\lambda x_0 + (1-\lambda)x_1}(\lambda y_0 + (1-\lambda) y_1)
\ge 
\psi_{x_0}(y_0)^\lambda \psi_{x_1}(y_1)^{1-\lambda}. 
\end{equation}
Once we could prove \eqref{e:PLAssump}, then we conclude from the Pr\'{e}kopa--Leindler inequality that 
$$
\Psi(\lambda x_0 + (1-\lambda)x_1)
\ge \bigg( \int_{\mathbb{R}^n} \psi_{x_0} \bigg)^\lambda \bigg( \int_{\mathbb{R}^n} \psi_{x_1} \bigg)^{1 - \lambda} 
= 
\Psi(x_0)^\lambda \Psi(x_1)^{1-\lambda }
$$
which shows the desired log-convexity of $\Psi$. 

Towards \eqref{e:PLAssump}, we regard $\psi(x,y):= \psi_x(y)$. In fact, from this viewpoint,  \eqref{e:PLAssump} is clearly equivalent to the log-concavity of $\psi$ in the $(x,y)$-variable: 
$$
\psi( \lambda ( x_0,y_0 ) + (1-\lambda) (x_1,y_1) )
\ge 
\psi( x_0,y_0  )^{\lambda}
\psi( x_1,y_1  )^{1-\lambda}. 
$$
This can be seen by writing  
$$
\psi(x,y) = \gamma_{\beta(1-e^{-2t})}(y) \vartheta(y+e^{-t}x),\;\;\; \vartheta(w):= e^{\frac{|w|^2}{2\beta}}v_0(w)
$$
and noting that our assumption \eqref{e:ConcavePresereve} implies $\vartheta$ is log-concave. It is also clear that $\gamma_{\beta(1-e^{-2t})}$ is log-concave too. Hence 
\begin{align*}
&\psi( \lambda ( x_0,y_0 ) + (1-\lambda) (x_1,y_1) )\\
& \ge
\gamma_{\beta(1-e^{-2t})}(y_0)^{\lambda}  \gamma_{\beta(1-e^{-2t})}(y_1)^{1-\lambda} \vartheta(y_0+e^{-t}x_0)^\lambda\vartheta( y_1+e^{-t}x_1)^{1-\lambda}\\
& = 
\psi( x_0,y_0  )^{\lambda}
\psi( x_1,y_1  )^{1-\lambda}
\end{align*}
which gives the desired log-concavity of $\psi$.

To show \eqref{e:PreSubHar}, we borrow ideas from the proof\footnote{For technical reasons, the result itself in \cite[Lemma 2.4]{GKL} is not quite applicable in our case to derive \eqref{e:PreSubHar}.} of \cite[Lemma 2.4]{GKL}.
We first note that for any twice differentiable function $\rho$ on $\mathbb{R}^n$ 
$$
\Delta \rho \ge 0
\Leftrightarrow 
\rho(x) \le 
\int_{O(n)} \rho(x+\theta y)\, d\sigma(\theta) 
$$
for all $x,y\in\mathbb{R}^n$, where $O(n)$ represents the group of orthonomal transforms and $d\sigma$ denotes the normalised Haar measure on $O(n)$, see \cite{GKL} for instance.  
Hence, recalling \eqref{e:vtForm18Feb}, it suffices to show that $x\mapsto \int \phi(x,w)\, d\mu(w)$ is log-subharmonic, namely 
\begin{equation}\label{e:Goal26Apr}
\log\, \int_{\mathbb{R}^n} \phi(x,w)\, d\mu(w)
\le 
\int_{O(n)} 
\log\, \int_{\mathbb{R}^n} \phi(x+\theta y,w)\, d\mu(w)
\, d\sigma(\theta)
\end{equation}
for any $x,y\in\mathbb{R}^n$. Also the assumption ensures that $x\mapsto \phi(x,w)$ is log-subharmonic for each fixed $w$.
From now we fix arbitrary $x,y$. 
We claim that 
\begin{equation}\label{e:Goal26Apr_2}
\log\, \int_{\mathbb{R}^n} \phi(x,w) \chi_R(w)\, d\mu(w)
\le 
\int_{O(n)} 
\log\, \int_{\mathbb{R}^n} \phi(x+\theta y,w)\chi_R(w)\, d\mu(w)
\, d\sigma(\theta)
 \end{equation}
 for all $R\in\mathbb{N}$ where $\chi_R$ is the characteristic function of $[-R,R]^n$. 
 Once we could see \eqref{e:Goal26Apr_2}, then Fatou's lemma and the monotone convergence theorem show that 
 \begin{align*}
\log\, \int_{\mathbb{R}^n} \phi(x,w)\, d\mu(w)
\le&
\liminf_{R\to\infty}
\log\, \int_{\mathbb{R}^n} \phi(x,w) \chi_R(w)\, d\mu(w)\\
\le& 
\liminf_{R\to\infty}
\int_{O(n)} 
\log\, \int_{\mathbb{R}^n} \phi(x+\theta y,w)\chi_R(w)\, d\mu(w)
\, d\sigma(\theta)\\
=& 
\int_{O(n)} 
\log\, \int_{\mathbb{R}^n} \phi(x+\theta y,w)\, d\mu(w)
\, d\sigma(\theta) 
 \end{align*}
 and hence \eqref{e:Goal26Apr}.
 
To show \eqref{e:Goal26Apr_2}, we denote the density of $d\mu(w)$ by $p(w)$ and approximate the integral $ \int \phi(x,w) \chi_R(w)\, d\mu(w)$ by its Riemann sum: 
\begin{align*}
\log\, \int_{\mathbb{R}^n} \phi(x,w) \chi_R(w)\, d\mu(w)
=& 
\log\, \int_{\mathbb{R}^n} \phi(x,w) \chi_R(w)p(w)\, dw\\
=&
\lim_{h\downarrow0}
\log\,  
{h^n} \sum_{i=1}^{I_{h,R}} \phi( x, w_i ) p(w_i), 
\end{align*} 
where $(w_i)_{i=1}^{I_{h,R}}$ denotes lattice points in $[-R,R]^n$ with width $h$ and $I_{h,R} \sim h^{-n} (2R)^n$. Since log-subharmonicity is preserved under the finite superposition (see, for example, Proposition 2.2 in \cite{GKL}), it follows that
\begin{align*}
\log\, \int_{\mathbb{R}^n} \phi(x,w) \chi_R(w)\, d\mu(w) \le 
\lim_{h\downarrow 0} 
\int_{O(n)}
\log\, 
\big( {h^n} \sum_{i=1}^{I_{h,R}} \phi( x + \theta y , w_i ) p(w_i)\big)\, d\sigma(\theta). 
\end{align*}
If we could justify the interchanging of $\lim_{h\downarrow 0}$ and $\int_{O(n)}$, then we would  conclude \eqref{e:Goal26Apr_2} and complete the proof. 
To this end let us show that 
\begin{equation}\label{e:UniformConv}
\lim_{h\downarrow 0 } 
\sup_{\theta \in O(n)}
\bigg| 
{h^n} \sum_{i=1}^{I_{h,R}} \phi( x + \theta y , w_i ) p(w_i)
- 
\int_{\mathbb{R}^n} \phi(x+\theta y, w) \chi_R(w) p(w)\, dw
\bigg|
=
0.
\end{equation}
Note that we do not need to pay attention to the dependence on $x,y,R$ as these are fixed.  
If we decompose $[-R,R]^n$ into cubes $(Q_i)_{i=1}^{I_{h,R}}$ where the sidelength of each cube is $h$ and $w_i \in Q_i$, then 
\begin{align*}
&\bigg| 
{h^n} \sum_{i=1}^{I_{h,R}} \phi( x + \theta y , w_i ) p(w_i)
- 
\int_{\mathbb{R}^n} \phi(x+\theta y, w) \chi_R(w) p(w)\, dw
\bigg|\\
& \le 
{h^n} \sum_{i=1}^{I_{h,R}} 
\frac{1}{|Q_i|}
\int_{Q_i} 
\big| 
\psi_\theta(w_i)
-
\psi_\theta(w)
\big|\, dw, 
\end{align*}
where $\psi_\theta(w) := \phi(x+\theta y, w) p(w)$.  From the mean value theorem, and by using the explicit formulae for $\psi_\theta, \phi, p$, we obtain
\begin{align*}
\bigg| 
{h^n} \sum_{i=1}^{I_{h,R}} \phi( x + \theta y , w_i ) p(w_i)
- 
\int_{\mathbb{R}^n} \phi(x+\theta y, w) \chi_R(w) p(w)\, dw
\bigg| \lesssim
h, 
\end{align*}
where the implicit constant depends on $x,y,R,\beta,t$. Hence \eqref{e:UniformConv} follows. 
\end{proof}
\begin{remark}
It is reasonable to seek an analogue of \eqref{e:PreSubHar} for semi-log-superharmonicity. This would follow if it were true that $\Delta \log\, P_tf \le 0$ for all $t > 0$ whenever 
$
\Delta \log\, f \le 0.
$ 
Interestingly, this is too optimistic and can easily be seen to fail. 
In fact, if we take 
$$
f(x) = e^{ x_1x_2 },\;\;\; x = (x_1,x_2) \in \mathbb{R}^2, 
$$
then $ \Delta \log\, f =0 $ whilst $\Delta \log\, P_tf >0$ for all $t>0$.  
\end{remark}

The following lemma, which is a consequence of Lemmas \ref{l:LogConVari} and \ref{l:LogPreserve}, will be used to justify certain technicalities in the flow monotonicity argument in the next section.
We denote by $\Gamma$ the carr\'e du champ operator
$$
\Gamma\phi(x):= |\nabla \phi(x)|^2
$$
for differentiable functions $\phi$ on $\mathbb{R}^n$. 
\begin{lemma}\label{l:Technical1}
Let  $\beta>0$. Assume that the twice differentiable function $v:\mathbb{R}^n\to(0,\infty)$ satisfies 
\begin{equation}\label{e:Tech25May}
\inf \frac{v}{\gamma_\beta} >0.
\end{equation}
Suppose also that $v \in L^2(\gamma_\beta^{-1})$.
Then, for a solution $v_t$ to \eqref{e:BetaFPv0} and a nonzero real number $p$,  we have 
$$
\big( \frac{v_t}{\gamma} \big)^\frac1p \Delta W_j,\; \big( \frac{v_t}{\gamma} \big)^\frac1p \nabla W_j,\;  \big( \frac{v_t}{\gamma} \big)^\frac1p \Gamma W_j,\; \big( \frac{v_t}{\gamma} \big)^\frac1p \nabla W_1 \cdot \nabla W_2 \in L^1(d\gamma),\;\;\; j=1,2, 
$$
for each $t>0$, where $W_1(x):= \frac{|x|^2}{2}$ and $W_2:= \log\, \frac{v_t}{\gamma}$. 
\end{lemma}

\begin{proof}
We prove $ \big( \frac{v_t}{\gamma} \big)^\frac1p \Delta W_2 \in L^1(\gamma) $ and remark that this case contains ideas which may be used for other terms. 
A direct computation shows that 
\begin{align*}
\big( \frac{v_t}{\gamma} \big)^\frac1p \Delta W_2
=& 
- \big( \frac{v_t}{\gamma_\beta} \big)^{\frac1p - 2} \big( \frac{\gamma_\beta}{\gamma} \big)^{\frac1p - 2} \big| \frac{\gamma_\beta}{\gamma} \nabla \frac{v_t}{\gamma_\beta} +  \frac{v_t}{\gamma_\beta} \nabla \frac{\gamma_\beta}{\gamma}  \big|^2\\
&+
\big( \frac{v_t}{\gamma_\beta} \big)^{\frac1p - 1} \big( \frac{\gamma_\beta}{\gamma} \big)^{\frac1p - 1}
\big(  \frac{\gamma_\beta}{\gamma} \Delta \frac{v_t}{\gamma_\beta} + 2 \nabla \frac{v_t}{\gamma_\beta} \cdot \nabla \frac{\gamma_\beta}{\gamma} + \frac{v_t}{\gamma_\beta} \Delta \frac{\gamma_\beta}{\gamma}  \big).
\end{align*}
Hence  $ |\big( \frac{v_t}{\gamma} \big)^\frac1p \Delta W_2| $ is bounded, up to a multiplicative constant, by 
\begin{align*}
& \big( \frac{v_t}{\gamma_\beta} \big)^{\frac1p - 2} |\nabla \frac{v_t}{\gamma_\beta}|^2  \big( \frac{\gamma_\beta}{\gamma} \big)^{\frac1p }
+ 
\big( \frac{v_t}{\gamma_\beta} \big)^{\frac1p }  \big( \frac{\gamma_\beta}{\gamma} \big)^{\frac1p  }|\nabla W_1|^2\\
+& 
\big( \frac{v_t}{\gamma_\beta} \big)^{\frac1p - 1} | \Delta \frac{v_t}{\gamma_\beta} |  \big( \frac{\gamma_\beta}{\gamma} \big)^{\frac1p }
+
\big( \frac{v_t}{\gamma_\beta} \big)^{\frac1p }  \big( \frac{\gamma_\beta}{\gamma} \big)^{\frac1p  } |\Delta W_1 | 
+
\big( \frac{v_t}{\gamma_\beta} \big)^{\frac1p - 1} |\nabla \frac{v_t}{\gamma_\beta} |  \big( \frac{\gamma_\beta}{\gamma} \big)^{\frac1p}|\nabla W_1|. 
\end{align*}
Note that $ u_t:= \frac{v_t}{\gamma_\beta} $ satisfies the gaussian heat equation 
$$
\partial_t u_t = \mathcal{L}_\beta u_t,\;\;\; u_0 \in L^1(\gamma_\beta ) \cap L^2(\gamma_\beta)
$$
since we assumed $ v_0 \in L^2(\gamma_\beta^{-1}) $. 
Hence we may employ \cite[(7.78)]{Grigolyan} to bound 
$$
\sup_{y \in \mathbb{R}^n} | \frac{v_t}{\gamma_\beta}(y)|  +   |\nabla \frac{v_t}{\gamma_\beta}(y)| + |\Delta \frac{v_t}{\gamma_\beta}(y)|
\le 
C(t) 
\| u_0 \|_{L^2(\gamma_\beta)},
$$
where $C(t)$ is locally bounded in $t$. 

Also we know from the assumption \eqref{e:Tech25May} and the formula \eqref{e:FPsolution} that, for some $\varepsilon > 0$, we have the lower bound $ \frac{v_t}{\gamma_\beta} \ge \varepsilon$ for all $t > 0$. This yields  
\begin{align*}
&|\big( \frac{v_t}{\gamma} \big)^\frac1p \Delta W_2 | \le 
C(t,\varepsilon) \bigg( 1+ |\nabla W_1| + |\nabla W_1|^2  \bigg) \big( \frac{\gamma_\beta}{\gamma} \big)^\frac1p.
\end{align*}
If we recall $W_1(x) = \frac{|x|^2}{2}$ and notice that  $ \big( \frac{\gamma_\beta}{\gamma} \big)^\frac1p \le C_{p,\beta} \frac{\gamma_\beta}{\gamma} $, then we conclude that 
\begin{align*}
\int_{\mathbb{R}^n} |\big( \frac{v_t}{\gamma} \big)^\frac1p(y) \Delta W_2(y) | \, d\gamma(y) \le
C(t,\varepsilon) \int_{\mathbb{R}^n} ( 1  + |y|^2)  \frac{\gamma_\beta}{\gamma} \, d\gamma(y)<\infty
\end{align*} 
as desired.
\end{proof}

\section{Hypercontractivity and the LSI : Proof of Theorems \ref{t:MainHyper} and \ref{t:MainLSI}, and Corollaries \ref{c:FPversion} and \ref{Cor:Matrix}}\label{S3}

\subsection{Flow monotonicity: the key result}
Recall that our strategy is to first prove Theorem \ref{t:MainHyper} and then deduce Theorem \ref{t:MainLSI} from it. For Theorem \ref{t:MainHyper}, we establish the monotonicity of a certain functional under Fokker--Planck flow.  
Fix any $\beta>0$ and $v$ satisfying the assumptions in Theorem \ref{t:MainHyper}, and let $v_t$ be the $\beta$-Fokker--Planck solution with initial data $v$ as in \eqref{e:BetaFPv0}. 
We then introduce a functional by 
\begin{equation}\label{e:vtilde}
Q(t):= \int_{\mathbb{R}^n} \widetilde{v}_t\, dx,\;\;\; \widetilde{v}_t:= \gamma P_s \big[ \big( \frac{v_t}{\gamma}\big)^\frac1p \big]^q ,\;\;\; t>0. 
\end{equation}
It will suffice to establish that $Q(t)$ is non-decreasing on $(0,\infty)$ and in order to achieve this we seek to obtain an expression for $\partial_t \widetilde{v}_t$ involving either non-negative terms or terms which vanish upon integrating with respect to Lebesgue measure on $\mathbb{R}^n$ (such as $\mathcal{L}_{\beta_s}^* \widetilde{v}_t$). Establishing such an expression is the main component of the proof of Theorem \ref{t:MainHyper} and follows from an extremely careful investigation of the interaction between two flows,  one is the Ornstein--Uhlenbeck flow $P_s$ and the other is the $\beta$-Fokker--Planck flow $v_t$. The key result underlying this phenomenon is the following.
\begin{proposition}\label{Prop:Uhmmm2}
Let $\beta>0$, $s>0$, $1<p<q<\infty$ satisfy $\frac{q-1}{p-1} = e^{2s}$.  
Suppose the twice differentiable $v:\mathbb{R}^n\to(0,\infty)$ belongs to $L^2(\gamma_\beta^{-1})$ and satisfies \eqref{e:Tech25May}. Also, when $\beta> 1$ let $v$ be $\beta$-semi-log-subharmonic, and when $\beta<1$ let $v$ be $\beta$-semi-log-concave.
If 
$$
\partial_{t} v_t = \mathcal{L}_\beta^* v_t,\;\;\; v_0 = v,
$$
then we have that 
\begin{align*}
&\big( \frac{\beta_s}{\beta} \big)^2	\frac1{\beta_s} \frac{pp'}{q} \frac{ \partial_t \widetilde{v}_t - \mathcal{L}_{\beta_s}^* \widetilde{v}_t }{ \widetilde{v}_t^{1-\frac{2}{q}}\gamma^\frac{2}q } \\
& \ge
P_sh_t P_s\big[ h_t \Gamma( \log\, \frac{v_t}{\gamma_\beta} ) \big] - \big| P_s\big[ h_t \nabla ( \log\, \frac{v_t}{\gamma_\beta} ) \big]\big|^2 \\
&\quad  + 
\big( \frac{\beta_s}{\beta} \big)^2	\frac1{\beta_s} \frac{pp'}{q} (I+N) 
+\big(1 + \frac{\beta_s}{\beta} \big)(1-\frac{1}{\beta})
	R( h_t, h_t  \nabla \big(\log\, \frac{v_t}{\gamma}\big) )\\
	&\quad  -
	(1-\frac1{\beta})^2 R( h_t, h_t x   ),
\end{align*}
where $h_t:= \big( \frac{v_t}{\gamma}\big)^\frac1p$, and $\beta_s>0$ and $\widetilde{v}_t$ are defined by \eqref{e:beta2} and \eqref{e:vtilde}, respectively. Here, 
\begin{align*}
I & :=
q(\beta-\beta_s) P_sh_t P_s \big[ h_t \Delta \big( \log\,  \frac{v_t}{\gamma} \big) \big] 
+\frac qp (\beta-\beta_s) P_sh_t P_s \big[ h_t \Gamma( \log\,  \frac{v_t}{\gamma} ) \big]\\
& \quad  + 
\frac{q}{p}( \beta_s - 2\beta +1 )  P_sh_t P_s \big[ h_t x \cdot \nabla \big( \log\,  \frac{v_t}{\gamma} \big) \big]\\
&\quad  +
\frac qp (1-\beta) P_sh_t  P_s \big[  h_t( n - |x|^2 ) \big],\\
N & := 
q( \beta_s - 1)  P_sh_t x \cdot \nabla P_sh_t + (\beta_s - 1) (n - |x|^2)  \big(P_sh_t\big)^2,
\end{align*}
and for $f$ and $\mathbf{g} = (g_j)_{j=1}^n$, 
$$
R(f,\mathbf{g})
:= 
P_s f P_s\big[ x\cdot \mathbf{g} \big] - P_s\big[ f x \big] \cdot P_s\mathbf{g}
		-
		(1-e^{-2s})
		\big(
		P_s f P_s\big[ {\rm div}\, \mathbf{g} \big] - P_s\big[ \nabla f\big] \cdot P_s\mathbf{g}
		\big).
$$
\end{proposition}
Two remarks are in order.  Firstly the additional assumption \eqref{e:Tech25May} we imposed in the statement of Proposition \ref{Prop:Uhmmm2} is a technical one and it is harmless for proving Theorem \ref{t:MainHyper}.  Secondly, we recommend the reader does not dwell on the complicated terms $ I,N,R( h_t, h_t  \nabla \big(\log\, \frac{v_t}{\gamma}\big) ), R( h_t, h_t x   )$. In fact, one can see that all of them become identically zero; see the forthcoming discussion in Section \ref{S3.4}.  Nevertheless, we keep them as they are in order facilitate our forthcoming investigation of hypercontractivity on more general measure spaces rather than gaussian space; see Section \ref{subsection:hypergeneral}. 

To show Proposition \ref{Prop:Uhmmm2}, we use the following two lemmas. 
\begin{lemma}\label{l:Step1}
Under the assumptions in Proposition \ref{Prop:Uhmmm2}, we have 
\begin{align*}
&\frac{ \partial_t \widetilde{v}_t - \mathcal{L}_{\beta_s}^* \widetilde{v}_t }{ \widetilde{v}_t^{1- \frac{2}{q}}\gamma^\frac2q }
=
\frac{q}{pp'} p( \beta_s - \beta ) P_sh_t P_s\big[ h_t \Delta \log\, \frac{v_t}{\gamma} \big]
+ 
\frac{q}{pp'} \beta_s I_1 + I + N,
\end{align*}
where $h_t: = \big(\frac{v_t}{\gamma}\big)^\frac1p$ and 
\begin{equation}\label{e:I1}
I_1:= P_sh_t P_s\big[ h_t \Gamma\log\, \frac{v_t}{\gamma} \big] - \big| P_s\big[ h_t \nabla \log \frac{v_t}{\gamma}\big]\big|^2. 
\end{equation}
\end{lemma}

\begin{lemma}\label{l:Step2}
Under the assumptions in Proposition \ref{Prop:Uhmmm2}, we have 
\begin{align*}
&
P_sh_t P_s\big[ h_t \Gamma\log\, \frac{v_t}{\gamma_\beta} \big] - \big| P_s\big[ h_t \nabla \log \frac{v_t}{\gamma_\beta}\big]\big|^2\\
& = 
p( 1-\frac1\beta )( 1-\frac{\beta_s}{\beta} )P_s h_t P_s\big[ n h_t  \big]
-
p( 1 - \frac{\beta_s}{\beta} )( 1 + \frac{\beta_s}{\beta} ) P_sh_t P_s\big[ h_t \Delta \log\, \frac{v_t}{\gamma} \big] \\
& 
\quad  +\big( \frac{\beta_s}{\beta} \big)^2 I_1
    -
	\big(1 + \frac{\beta_s}{\beta} \big)(1-\frac{1}{\beta})
	R( h_t, h_t \nabla \big(\log\, \frac{v_t}{\gamma}\big) )
	+
	(1-\frac1{\beta})^2 R( h_t, h_t x ), 
\end{align*}
where $h_t: = \big(\frac{v_t}{\gamma}\big)^\frac1p$ and $I_1$ is given by \eqref{e:I1}. 
\end{lemma}

Let us complete the proof of Proposition \ref{Prop:Uhmmm2} by assuming the validity of Lemmas \ref{l:Step1} and \ref{l:Step2} for a moment. 

\begin{proof}[Proof of Proposition \ref{Prop:Uhmmm2}]
With Lemma \ref{l:Step2} in mind,  we  focus on the term
$$
p( 1-\frac1\beta )( 1-\frac{\beta_s}{\beta} )P_s h_t P_s \big[ nh_t\big] 
= 
pP_s h_t \frac1\beta P_s\big[ h_t ( \beta- \beta_s ) ( 1-\frac1\beta )n \big]. 
$$
We appeal to Lemma \ref{l:LogPreserve} and our assumptions on $v$ to see 
$$
\begin{cases}
(1 - \frac1\beta ) n  \ge \Delta \log\, \frac{v_t}{\gamma}\;\;\; &{\rm if}\;\;\; 0<\beta<1\\
(1 - \frac1\beta ) n  \le \Delta \log\, \frac{v_t}{\gamma}\;\;\; &{\rm if}\;\;\; 1<\beta,
\end{cases}
$$
for all $t>0$. 
For the sign of the coefficient, from the definition of $\beta_s$ we know that 
$$
\beta - \beta_s = \beta - 1 - (\beta - 1) \frac{q}{p}e^{-2s} = \frac1p( \beta - 1 )( 1 - e^{-2s}  )
$$
and hence
$$
\begin{cases}
\beta - \beta_s\le 0\;\;\; &{\rm if}\;\;\; 0< \beta <1\\
\beta - \beta_s\ge0\;\;\; &{\rm if}\;\;\; 1< \beta.  
\end{cases}
$$
Overall, we see that for all $t>0$
$$
( \beta- \beta_s ) ( 1-\frac1\beta ) n \le (\beta - \beta_s) \Delta \log\, \frac{v_t}{\gamma},
$$
regardless whether $\beta$ is less than or greater than one. 
Therefore from Lemma \ref{l:Step2} we obtain that for all $t,\beta>0$, 
\begin{align*}
&
P_sh_t P_s\big[ h_t \Gamma\log\, \frac{v_t}{\gamma_\beta} \big] - \big| P_s\big[ h_t \nabla \log \frac{v_t}{\gamma_\beta}\big] \big|^2\\
& \le 
p( 1-\frac{\beta_s}{\beta} )P_s h_t P_s\big[ h_t \Delta \log\, \frac{v_t}{\gamma} \big]
-
p( 1 - \frac{\beta_s}{\beta} )( 1 + \frac{\beta_s}{\beta} ) P_sh P_s\big[ h_t \Delta \log\, \frac{v_t}{\gamma} \big] 
+\big( \frac{\beta_s}{\beta} \big)^2 I_1\\
& \quad 
-
	\big(1 + \frac{\beta_s}{\beta} \big)(1-\frac{1}{\beta})
	R( h_t, h_t \nabla \big(\log\, \frac{v_t}{\gamma}\big) )
	+
	(1-\frac1{\beta})^2 R( h_t, h_t x )\\
& = 
\frac{\beta_s}{\beta^2}\frac{pp'}{q} 
\bigg(
\frac{q}{pp'}p( \beta_s - \beta )  P_s h_t P_s\big[ h_t \Delta \log\, \frac{v_t}{\gamma} \big]
+\frac{q}{pp'} \beta_s I_1
\bigg)\\
& \quad  -
	\big(1 + \frac{\beta_s}{\beta} \big)(1-\frac{1}{\beta})
	R( h_t, h_t \nabla \big(\log\, \frac{v_t}{\gamma}\big) )
	+
	(1-\frac1{\beta})^2 R( h_t, h_t x ). 
\end{align*}
By applying Lemma \ref{l:Step1}, we conclude the proof. 
\end{proof}

It remains to show Lemmas \ref{l:Step1} and \ref{l:Step2}. 
\begin{proof}[Proof of Lemma \ref{l:Step1}]
We remark that each of the terms appearing on the right-hand side (e.g. $P_s \big[ \big( \frac{v_t}{\gamma} \big)^\frac1p \Delta( \log\, \frac{v_t}{\gamma} ) \big]$) are well-defined for each $t>0$ since $P_s$ is bounded on $L^1(\gamma)$ and Lemma \ref{l:Technical1}.

In view of the commutation relation $\nabla P_s = e^{-s}P_s \nabla$, our goal is to show 
\begin{align}
\frac{\partial_t \widetilde{v}_t - \mathcal{L}_{\beta_s}^* \widetilde{v}_t}{\widetilde{v}_t^{1-\frac 2q}\gamma^\frac{2}{q}}
& = \frac q{pp'}p (\beta_s-\beta) P_sh_t P_s \big[ h_t \Delta \big( \log\,  \frac{v_t}{\gamma} \big) \big] \label{e:4OctStep1}\\
& \quad  + 
\frac{q}{pp'} \beta_s \bigg( P_sh_t P_s \big[ h_t \Gamma( \log\,  \frac{v_t}{\gamma} ) \big] - ( pe^{ s } )^2 \Gamma( P_s h_t ) \bigg) + I +N. \nonumber 
\end{align}

Thanks to the definition of $\widetilde{v}_t$ and the fact that $v_t$ is a solution to \eqref{e:BetaFPv0}, we see that 
\begin{equation}\label{e:7Oct1}
\partial_t \widetilde{v}_t = q \widetilde{v}_t^{ 1-\frac1q }\partial_t \big( \widetilde{v}_t^\frac1q \big)
= 
\frac{q}{p} \widetilde{v}_t\big( \frac{\widetilde{v}_t}{\gamma} \big)^{- \frac1q} P_s\big[ \big( \frac{v_t}{\gamma} \big)^{\frac1p-1} \frac{ \mathcal{L}_{\beta}^* v_t}{\gamma}  \big].
\end{equation}
Here we implicitly commuted the operators $\partial_t$ and $P_s$; that is
\begin{equation}\label{e:Change1}
\partial_t P_s \big[ \big( \frac{v_t}{\gamma} \big)^\frac1p \big]
=
P_s \big[ \partial_t \big( \frac{v_t}{\gamma} \big)^\frac1p \big].
\end{equation}
This can be checked as follows. 
By denoting the integral kernel of $P_s$ with respect to $ d\gamma$ by $p_s(x,y)$,  we write 
\begin{align*}
P_s \big[ \big( \frac{v_t}{\gamma} \big)^\frac1p\big](x) 
&=
\int  \big( \frac{v_t}{\gamma_\beta} \big)^\frac1p (y) \big( \frac{\gamma_\beta}{\gamma} \big)^\frac1p(y) p_s(x,y)\, d\gamma(y)\\
&=: 
\int \big( \frac{v_t}{\gamma_\beta} \big)^\frac1p (y) \big( \frac{\gamma_\beta}{\gamma} \big)^\frac1p(y) \, dP_{s,x}(y)
\end{align*}
and investigate the bound of $\partial_t \big( \frac{v_t}{\gamma_\beta} \big)^\frac1p = \frac1p\big( \frac{v_t}{\gamma_\beta} \big)^{\frac1p-1} \partial_t \big( \frac{v_t}{\gamma_\beta} \big) $. 
As we did in the proof of Lemma \ref{l:Technical1}, we may appeal to \cite[(7.78)]{Grigolyan} to obtain that 
$$
\sup_y |\partial_t \frac{v_t}{\gamma_\beta}(y)| \le C(t) \big\| \frac{v_0}{\gamma_\beta} \big\|_{L^2(\gamma_\beta)}
=
C(t) \| v \|_{L^2(\gamma_\beta^{-1})}
$$
and hence 
$$
\partial_t \big(  \big( \frac{v_t}{\gamma_\beta} \big)^\frac1p (y) \big( \frac{\gamma_\beta}{\gamma} \big)^\frac1p(y) \big)
\le 
C(t,\varepsilon) \| u_0\|_{L^2(\gamma_\beta)} G_x(y),
$$
where $\varepsilon = \inf \frac{v}{\gamma_\beta}$, the constant $C(t,\varepsilon)$ is locally bounded in $t$, and 
$$
G_x(y) 
:= 
\big( \frac{\gamma_\beta}{\gamma} \big)^\frac1p(y).
$$
The same argument in the proof of Lemma \ref{l:Technical1} shows that $G_x\in L^1(\gamma)$ from which we obtain that 
\begin{equation}\label{e:7Oct_2}
\int G_x(y) \, dP_{s,x}(y) = P_s \big[ G_x \big](x) <\infty, 
\end{equation}
for almost every $x \in \mathbb{R}^n$ since $P_s$ is bounded on $L^1(\gamma)$.  
This suffices to ensure \eqref{e:Change1}. 
We remark that this further shows the bound
\begin{equation}\label{e:Deri_t}
|\partial_t \widetilde{v}(t,x)|
\le 
C(t,\varepsilon)\gamma(x) P_s\big[e^{\frac1p(1-\frac1\beta)|\cdot|^2+c(t)|\cdot|}\big](x)^{q} ,
\end{equation}
where $C(t,\varepsilon),c(t)>0$ are locally bounded in $t$ and depend on $\|v\|_{L^2(\gamma_\beta^{-1})} <\infty$. 


We then compute $\mathcal{L}_\beta^* {v}_t$.  Since $\gamma$ satisfies $ \mathcal{L}_1^* \gamma = 0 $, in general we have that 
\begin{equation}\label{e:Formula1}
\mathcal{ L }^*_{\beta} \big( \gamma^\frac1q f\big) 
=
\gamma^\frac1q \mathcal{L}_{\beta} f
+ 
(1 - \frac{\beta}{q}) \gamma^\frac1q T_q f
\end{equation}
for $q \in \mathbb{R}\setminus\{0\}$,
where  
$$
T_qf:= 2 x \cdot \nabla f + ( n - \frac1q | x |^2 ) f. 
$$
Applying this with $q = 1$ and $ f= \frac{v_t}{\gamma} $, it follows that 
$$
\mathcal{L}^*_{\beta} v_t 
= 
\mathcal{ L }^*_{\beta} \big( \gamma \frac{v_t}{\gamma} \big) 
=
\gamma \mathcal{L}_{\beta}\big( \frac{v_t}{\gamma} \big) 
+ 
(1 - \beta) \gamma T_1\big( \frac{v_t}{\gamma} \big),
$$
which further yields that 
\begin{equation}\label{e:5OctId1}
\partial_t \widetilde{v}_t
=
\frac{q}{p} \widetilde{v}_t\big( \frac{\widetilde{v}_t}{\gamma} \big)^{- \frac1q} P_s\big[ \big( \frac{v_t}{\gamma} \big)^{\frac1p-1} \mathcal{L}_{\beta}\big(  \frac{  v_t}{\gamma} \big)  \big]
+ 
\frac{q}{p} (1-\beta) \widetilde{v}_t\big( \frac{\widetilde{v}_t}{\gamma} \big)^{- \frac1q} P_s\big[ \big( \frac{v_t}{\gamma} \big)^{\frac1p-1} T_1\big( \frac{ v_t}{\gamma} \big)  \big].
\end{equation}

On the other hand,  in order to compute $\mathcal{L}_{\beta_s}^* \widetilde{v}_t$, we appeal to the diffusion property of $\mathcal{L}_{\beta}^*$ which states that for $\mu \in \mathbb{R} \setminus\{0\}$, 
\begin{equation}\label{e:DiffL*}
\mathcal{L}_{\beta}^* (f^\mu)
=
\mu f^{ \mu - 1 } \mathcal{L}_{\beta}^* f - (\mu - 1) f^\mu  n+ \beta \mu (\mu - 1) f^{ \mu -2 } \Gamma(f),
\end{equation}
and the commutation property 
\begin{equation}\label{e:Comm}
\mathcal{L}_{\beta} P_sf = \beta P_s \mathcal{L} f + (\beta-1) x\cdot \nabla \big( P_sf \big).
\end{equation}
A sequence of applications of \eqref{e:DiffL*}, \eqref{e:Formula1}, and \eqref{e:Comm} reveals that 
\begin{align*}
\mathcal{L}_{\beta_s}^* \widetilde{v}_t
& = 
\beta_s q \widetilde{v}_t^{ 1 - \frac1q } \gamma^\frac1q P_s \mathcal{L}\big[ \big( \frac{v_t}{\gamma} \big)^\frac1p \big]
+ 
(\beta_s -1) q \widetilde{v}_t^{1-\frac1q} \gamma^\frac1q x \cdot \nabla \big( P_s \big[ \big( \frac{v_t}{\gamma} \big)^\frac1p \big] \big)\\
& \quad + 
(q-\beta_s) \widetilde{v}_t^{1-\frac1q} \gamma^\frac1q T_qP_s \big[ \big( \frac{v_t}{\gamma} \big)^\frac1p \big]
-
(q-1) n\widetilde{v}_t
+ 
\beta_s q(q-1) \widetilde{v}_t^{1-\frac2q} \Gamma( \widetilde{v}_t^\frac1q ).
\end{align*}
We further apply the diffusion property of $\mathcal{L}$ to see that 
\begin{align}
\mathcal{L}_{\beta_s}^* \widetilde{v}_t
& =
\beta_s \frac{q}{p}  \widetilde{v}_t^{ 1 - \frac1q } \gamma^\frac1q P_s \big[ \big( \frac{v_t}{\gamma} \big)^{\frac1p-1} \mathcal{L} \big( \frac{v_t}{\gamma} \big) \big]
-
\beta_s \frac{q}{pp'} \widetilde{v}_t^{ 1 - \frac1q } \gamma^\frac1q P_s \big[ \big( \frac{v_t}{\gamma} \big)^{\frac1p-2} \Gamma \big( \frac{v_t}{\gamma} \big) \big]\label{e:5OctId2}\\
&\quad + 
(\beta_s -1) q \widetilde{v}_t^{1-\frac1q} \gamma^\frac1q x \cdot \nabla \big( P_s \big[ \big( \frac{v_t}{\gamma} \big)^\frac1p \big] \big)\nonumber\\
&\quad + 
(q-\beta_s) \widetilde{v}_t^{1-\frac1q} \gamma^\frac1q T_qP_s \big[ \big( \frac{v_t}{\gamma} \big)^\frac1p \big]
-
(q-1) n\widetilde{v}_t
+ 
\beta_s q(q-1) \widetilde{v}_t^{1-\frac2q} \Gamma( \widetilde{v}_t^\frac1q ).\nonumber
\end{align}
We combine \eqref{e:5OctId1} and \eqref{e:5OctId2} and then rearrange terms  to derive that 
\begin{align*}
&\frac{\partial_t \widetilde{v}_{  t} - \mathcal{L}_{\beta_s}^* \widetilde{v}_{  t}}{\widetilde{v}_{  t}^{1-\frac 2q}\gamma^\frac{2}{q}}\\
& =
\frac{q}{p} ( \beta - \beta_s ) P_sh_t P_s \big[ \big( \frac{v_t}{\gamma}  \big)^{\frac1p-1} \Delta \frac{v_t}{\gamma} \big]
+ 
\frac{q}{p}(1-\beta)  P_sh_t P_s \big[  h_t n \big] \\
& \quad +
(\beta_s -1)  \big(P_sh_t\big)^2 n
+
\bigg( \frac{q}{p} (\beta_s - 1) -2\frac{q}{p} (\beta-1) \bigg) P_sh_t P_s \big[ \big( \frac{v_t}{\gamma}  \big)^{\frac1p-1} x  \cdot \nabla \frac{v_t}{\gamma}\big]\\
& \quad  +
q(\beta_s -1) P_sh_t  x \cdot \nabla  P_s h_t 
+ 
\frac{q}{p}(\beta-1) P_sh_t P_s\big[ h_t |x|^2 \big] \\
&\quad  - 
(\beta_s-1) \big( P_sh_t \big)^2  |x|^2 
+
\frac{q}{pp'} \beta_s P_sh_t P_s\big[ \big( \frac{v_t}{\gamma} \big)^{\frac1p - 2} \Gamma( \frac{v_t}{\gamma} ) \big]
- 
\frac{q}{pp'} \beta_s ( pe^s )^2 \Gamma\big( P_s h_t \big).
\end{align*}
It now suffices to apply 
$$
\nabla \frac{v_t}{\gamma} = \frac{v_t}{\gamma} \nabla\big( \log\, \frac{v_t}{\gamma} \big),\;\;\;
\Delta \frac{v_t}{\gamma} = \frac{v_t}{\gamma} \Delta\big( \log\, \frac{v_t}{\gamma} \big) + \frac{v_t}{\gamma }\Gamma( \log\, \frac{v_t}{\gamma} )
$$
to obtain \eqref{e:4OctStep1}. 
\end{proof}

\begin{proof}[Proof of Lemma \ref{l:Step2}]
Using the fact that 
$$
\nabla \big( \log\, \frac{v_t}{\gamma_\beta} \big)
=
\nabla \big( \log\, \frac{v_t}{\gamma} \big)
+
( 1-\frac1\beta ) x, 
$$
it is easy to see that 
$$
P_s h_t P_s\big[ h_t \Gamma( \log\, \frac{v_t}{\gamma_\beta} ) \big] - P_s\big[ h_t \nabla ( \log\, \frac{v_t}{\gamma_\beta} )\big]^2
=
I_1+I_2+I_3,
$$
where $I_1$ is defined in \eqref{e:I1} and 
\begin{align*}
I_2
:=& 
( 1- \frac1\beta )^2 
\bigg(
P_sh_t P_s\big[ h_t |x|^2 \big] - P_s\big[ h_t x \big] \cdot P_s\big[ h_t x \big]
\bigg), \\
I_3
:=& 
-2( 1-\frac1{\beta} )
\bigg(
P_s h_t P_s\big[ h_t x \cdot \nabla \big( \log\, \frac{v_t}{\gamma} \big) \big] - P_s\big[ h_t x \big] \cdot P_s\big[ h_t \nabla \big( \log\, \frac{v_t}{\gamma}\big) \big]
\bigg). 
\end{align*}
For $I_2$,  from the definition of $R(f, \mathbf{g})$ with $f = h_t$ and $\mathbf{g} = h_t x $,  we have that 
\begin{align*}
	I_2=
	(1-\frac1{\beta})^2
	\bigg(& (1-e^{-2s})P_s h_t P_s\big[ h_t n \big] \\
	&+\frac{1-e^{-2Ks}}p( - \frac12(1-\frac{1}{\beta})^{-1}) I_3 
	+
	R( h_t, h_t x)
	\bigg). 
\end{align*}
Similarly for $I_3$, 
\begin{align*}
	I_3
	=& 
	-2(1-\frac1{\beta})\bigg(
	(1-e^{-2s})P_s h_t P_s\big[ h_t \Delta \big(\log\, \frac{v_t}{\gamma}\big)\big] \\
	&+\frac{1-e^{-2s}}p \big\{ 
	P_sh_t P_s\big[ h_t  \Gamma\big(\log\, \frac{v_t}{\gamma}\big) \big] 
	-
	\big| P_s\big[ h_t \nabla  \big(\log\, \frac{v_t}{\gamma}\big)
	\big]\big| ^2 
	\big\}\nonumber\\
	&+
	R( h_t, h_t \nabla  \big(\log\, \frac{v_t}{\gamma}\big) )
	\bigg).
\end{align*}
Hence, after a rearrangement, we arrive at 
\begin{align*}
&I_1+I_2+I_3\\
& =
(1-\frac1{\beta})^2 (1-e^{-2s})P_sh_t P_s\big[ h_t n \big] \\
	&\quad +\big(-2 + (1-\frac1{\beta})\frac{1-e^{-2s}}{p}\big) 
	(1-\frac{1}{\beta})(1-e^{-2s})P_s h_t P_s\big[ h_t \Delta \big(\log\, \frac{v_t}{\gamma}\big)\big] \\
	&\quad +\big( (1-\frac1{\beta})\frac{1-e^{-2s}}{p} - 1\big)^2 
	 I_1 \nonumber\\
	&\quad +
	\big(-2 + (1-\frac1{\beta})\frac{1-e^{-2s}}{p}\big)(1-\frac{1}{\beta})
	R( h_t, h_t  \nabla \big(\log\, \frac{v_t}{\gamma}\big) )
	+
	(1-\frac1{\beta})^2 R( h_t, h_t {  x} ).
\end{align*}
In view of the relation 
\begin{equation}\label{e:Exp5Oct}
(1-\frac1{\beta}) \frac{ 1 - e^{-2s} }{p}
=
-\frac1{\beta} (\beta_s - \beta),
\end{equation}
we can simplify to 
\begin{align*}
&I_1+I_2+I_3\\
& = 
p(1-\frac1{\beta}) (1- \frac{\beta_s}{\beta})P_s h_t  P_s\big[ h_t  n \big] 
	-p(1-\frac{\beta_s}{\beta})\big(1 + \frac{\beta_s}{\beta}\big) 
	P_s h_t P_s\big[ h_t \Delta \big(\log\, \frac{v_t}{\gamma}\big)\big] \\
	&\quad +\big( \frac{\beta_s}{\beta}\big)^2 
	 I_1 
	-
	\big(1 + \frac{\beta_s}{\beta} \big)(1-\frac{1}{\beta})
	R( h_t, h_t \nabla \big(\log\, \frac{v_t}{\gamma}\big) )
	+
	(1-\frac1{\beta})^2 R( h_t, h_t {  x} )
\end{align*}
as desired.
\end{proof}

\subsection{Proof of Theorem \ref{t:MainHyper}: hypercontractivity}
\begin{proof}[Proof of Theorem \ref{t:MainHyper}]
First we consider the forward version inequality \eqref{e:RegHC}, so we suppose $1 < p < q < \infty$. Also, we first handle the case $\beta > 1$ and thus we assume, in particular, that $v$ is $\beta$-semi-log-subharmonic.  To show \eqref{e:RegHC}, we may also suppose that $v$ satisfies \eqref{e:Tech25May}\footnote{To see this, consider $v^{\varepsilon}:= v + \varepsilon \gamma_\beta$ for arbitrary small $\varepsilon>0$. Then we have that 
$$
\Delta \log\, \frac{v^{\varepsilon}}{\gamma_\beta} \ge0,\;\;\; \frac{v^{\varepsilon}}{\gamma_\beta} \ge \varepsilon
$$
since log-sub-harmonicity is preserved under the superposition; see \cite[Proposition 2.2]{GKL}. Obtaining \eqref{e:RegHC} for each $v^\varepsilon$ would then allow us to deduce \eqref{e:RegHC} for $v$ by a limiting argument.}.

Recall the quantity $Q(t)$ defined in \eqref{e:vtilde} and that our goal is to show that $Q(t)$ is monotone non-decreasing. 
To this end, we first note that
\begin{equation}\label{e:TimeChange}
\frac{d}{dt}Q(t)
=
\int_{\mathbb{R}^n} \partial_t \widetilde{v}(t,x)\, dx,
\end{equation}
where the interchange of derivative and integral can be justified by observing \eqref{e:Deri_t} from the proof of Lemma \ref{l:Step1}.
We further see from integration by parts that 
$$
\frac{d}{dt}Q(t)
=
\int_{\mathbb{R}^n} \big( \partial_t \widetilde{v}(t,x) - \mathcal{L}_{\beta_s}^* \widetilde{v}(t,x) \big)\, dx. 
$$
To be precise, we implicitly used the fact that $\widetilde{v}(t,\cdot)$ satisfies \eqref{e:DecayTech} for each $t>0$ in the above step. 
This can be justified by 
$$
|x|\widetilde{v}(t,x)\le C(t,\varepsilon)|x|\gamma(x)P_s\big[\big(\frac{\gamma_\beta}{\gamma}\big)^\frac1p\big](x)^q\to 0,\;\;\; |x|\to\infty
$$
and 
$$
|\nabla \widetilde{v}(t,x)|\le 
C(t,\varepsilon)(1+|x|)\gamma(x) P_s\big[e^{\frac1p(1-\frac1\beta)|\cdot|^2+c(t)|\cdot|}\big](x)^{q} \to0,\;\;\; |x|\to\infty, 
$$
both of which follow from the assumption \eqref{e:Tech25May} and the argument in Lemma \ref{l:Technical1}. 

Thanks to the assumption on $v$, we may apply Proposition \ref{Prop:Uhmmm2} to see that 
\begin{align*}
& \big( \frac{\beta_s}{\beta} \big)^2	\frac1{\beta_s} \frac{pp'}{q}\frac{d}{dt}Q(t)\\
& \ge \int_{\mathbb{R}^n} \widetilde{v}_t^{1-\frac{2}{q}}\gamma^\frac{2}q 
\bigg(
\big( \frac{\beta_s}{\beta} \big)^2	\frac1{\beta_s} \frac{pp'}{q} (I+N)
	+\big(1 + \frac{\beta_s}{\beta} \big)(1-\frac{1}{\beta})
	R( h_t,  h_t \nabla  \big(\log\, \frac{v_{  t}}{\gamma}\big) )\\
& \quad -
	(1-\frac1{\beta})^2 R( h_t,  h_t  {  x} )
\bigg)\, dx, 
\end{align*}
where we also used the Cauchy--Schwarz inequality to see that 
$$
P_s h P_s\big[ h \Gamma( \log\, \frac{v_t}{\gamma_\beta}) \big] - \big| P_s\big[ h \nabla \big( \log\, \frac{v_t}{\gamma_\beta} \big) \big] \big|^2\ge 0. 
$$
In the above expression, the term involving $N$ vanishes. 
In fact,  we see that 
\begin{align*}
\int_{\mathbb{R}^n}  P_s\big[ \big(\frac{v_t}{\gamma}\big)^\frac1p\big]^q |x|^2 \, d\gamma
& =
- \int_{\mathbb{R}^n} P_s\big[\big( \frac{v_t}{\gamma} \big)^\frac1p \big]^q x\cdot \nabla \gamma  \, dx\\
& = 
q \int_{\mathbb{R}^n} \big(\frac{\widetilde{v}_t}{\gamma} \big)^{\frac{q-1}{q}}  \nabla \big( \big( \frac{ \widetilde{v}_t }{ \gamma }  \big)^\frac1q \big) \cdot x\, d\gamma 
+
\int_{\mathbb{R}^n} \frac{\widetilde{v}_t}{\gamma} n\, d\gamma\\
& =
q \int_{\mathbb{R}^n} \big(\frac{\widetilde{v}_t}{\gamma} \big)^{ -\frac{1}{q}} \widetilde{v}_t  \nabla \big( \big( \frac{ \widetilde{v}_t }{ \gamma }  \big)^\frac1q \big) \cdot x\, dx
+
\int_{\mathbb{R}^n} \widetilde{v}_t n\, dx
\end{align*}
from which $\int_{\mathbb{R}^n} \widetilde{v}_t^{1-\frac{2}{q}}\gamma^\frac{2}q  N\, dx = 0$ follows.  
Hence we see that 
\begin{align}
& \big( \frac{\beta_s}{\beta} \big)^2	\frac1{\beta_s} \frac{pp'}{q} \frac{d}{dt}Q(t) \label{e:10Oct_1} \\
& \ge \int_{\mathbb{R}^n} \widetilde{v}_t^{1-\frac{2}{q}}\gamma^\frac{2}q 
\bigg(
\big( \frac{\beta_s}{\beta} \big)^2	\frac1{\beta_s} \frac{pp'}{q} I 
	+\big(1 + \frac{\beta_s}{\beta} \big)(1-\frac{1}{\beta})
	R( h_t, h_t \nabla  \big(\log\, \frac{v_{  t}}{\gamma}\big) ) \nonumber \\
& \quad -
	(1-\frac1{\beta})^2 R( h_t, h_t  {  x} )
\bigg)\, dx. \nonumber 
\end{align}

Next we appeal to a feature well-suited to the gaussian case, namely two identities 
\begin{align}
&\int_{\mathbb{R}^n}  \nabla \big(P_{s}\big[ (P_sh)^{q-2} P_s f\big] \big) \cdot \mathbf{g} \, d\gamma\label{e:AssumpG=}\\
& = 
e^{-2s} \int_{\mathbb{R}^n}    P_{s} \big[ (P_s h)^{q-2} P_s\big[ \nabla f \big] \big] \cdot \mathbf{g} \, d\gamma \nonumber \\
& \quad +
e^{-2s} (q-2) 
\int_{\mathbb{R}^n}    P_{s} \big[ (P_s h)^{q-3} P_s f P_s\big[ \nabla h \big] \big] \cdot  \mathbf{g} \, d\gamma,\nonumber 
\end{align}
for $f, \mathbf{g} = (g_j)_{j=1}^n, {  h}$ and 
\begin{align}
&\int_{\mathbb{R}^n}   {\rm div}\,  \big(P_{s}\big[ (P_sh)^{q-2} P_s \mathbf{f} \big] \big){g} \, d\gamma\label{e:AssumpF=}\\
& = 
e^{-2s} \int_{\mathbb{R}^n}    P_{s} \big[ (P_s h)^{q-2} P_s\big[ {\rm div}\, \mathbf{f} \big] \big]{g} \, d\gamma \nonumber \\
& \quad +
e^{-2s} (q-2) 
\int_{\mathbb{R}^n} P_{s} \big[ (P_s h)^{q-3} P_s \mathbf{f} \cdot P_s\big[ \nabla h \big]  \big] \mathbf{g} \, d\gamma,\nonumber 
\end{align}
for $\mathbf{f}= (f_j)_{j=1}^n, g,{  h}$. 
These two identities follow from the commutation property $\nabla P_s = e^{-s} P_s \nabla$. 
As an immediate consequence from integration by parts, we obtain 
\begin{align}\label{e:ConsSep26G=}
&
\int_{\mathbb{R}^n} 
P_s f P_s \big[  \mathbf{g} \cdot {  x} \big] (P_sh)^{q-2} \, d\gamma  \\
& =
\int_{\mathbb{R}^n} 
P_s f P_s \big[ {\rm div}\, \mathbf{g} \big] (P_sh)^{q-2} \, d\gamma 
+ 
e^{-2s} \int_{\mathbb{R}^n} 
P_s \big[\nabla f\big] \cdot P_s \mathbf{g} (P_sh)^{q-2} \, d\gamma \nonumber \\
& \quad + e^{-2s} (q-2) 
\int_{\mathbb{R}^n} 
P_s f P_s \mathbf{g} \cdot  P_s\big[ \nabla h \big] (P_sh)^{q-3} \, d\gamma \nonumber
\end{align}
and
\begin{align}\label{e:ConsSep26F=}
&
\int_{\mathbb{R}^n} 
P_s \mathbf{f} \cdot  P_s \big[ g { x} \big] (P_sh)^{q-2} \, d\gamma  \\
& =
\int_{\mathbb{R}^n} 
P_s \mathbf{f} \cdot P_s \big[ \nabla g \big] (P_sh)^{q-2} \, d\gamma 
+ 
e^{-2s} \int_{\mathbb{R}^n} 
P_s \big[{\rm div}\, \mathbf{f} \big] P_s g (P_sh)^{q-2} \, d\gamma \nonumber \\
& \quad + e^{-2s} (q-2) 
\int_{\mathbb{R}^n} 
P_s \mathbf{f} \cdot   P_s\big[ \nabla h \big] P_s g  (P_sh)^{q-3}  \, d\gamma. \nonumber
\end{align}
These identities will be applied for appropriate inputs. 

We focus on the term involving $ P_s\big[ h_{  t} | {  x} |^2 \big] $ which comes from $I$ and $ R( h_{  t}, h_{  t} { x} ) $.  More precisely the contribution from such a term in \eqref{e:10Oct_1} is 
$$
\Upsilon_1(p,s,\beta)\int_{\mathbb{R}^n} 
P_s\big[ h_{  t} |{ x}|^2 \big]
P_s h_{  t}^{q-1}d\gamma,
$$
where 
\begin{equation}\label{e:Upsilon1}
\Upsilon_1(p,s,\beta):= 
\big( \frac{\beta_s}{\beta} \big)^2	\frac1{\beta_s} \frac{pp'}{q} \frac{q}{p} (\beta-1)
-
(1-\frac1{\beta})^2. 
\end{equation}
We apply \eqref{e:ConsSep26G=} with $(f,\mathbf{g}) = (h_{  t},h_{  t} { x})$ to this term to see that 
\begin{align*}
& \big( \frac{\beta_s}{\beta} \big)^2	\frac1{\beta_s} \frac{pp'}{q} \frac{d}{dt} Q(t) \\
& \ge
\big( \frac{\beta_s}{\beta} \big)^2	\frac1{\beta_s} \frac{pp'}{q} \widetilde{I} 
+\big(1 + \frac{\beta_s}{\beta} \big)(1-\frac{1}{\beta})
	\int_{\mathbb{R}^n} R( h_t,  h_t \nabla  \big(\log\, \frac{v_{  t}}{\gamma}\big) ) 
	P_sh_t^{q-2}\, d\gamma \\
	& \quad - (1-\frac1{\beta})^2 \int_{\mathbb{R}^n} \widetilde{R}(  h_t,  h_t { x} )
	P_sh_t^{q-2}\, d\gamma,
\end{align*}
where 
\begin{align}\label{e:ItildeNew}
& \widetilde{I} := 
q(\beta-\beta_s) \int_{\mathbb{R}^n}  P_s \big[ h_t \Delta\big( \log\,  \frac{v}{\gamma} \big) \big] P_sh_t^{q-1}\, d\gamma \\
& \quad +\frac qp (\beta-\beta_s) \int_{\mathbb{R}^n} P_s \big[ h_t \Gamma( \log\,  \frac{v}{\gamma} ) \big] P_sh_t^{q-1}\, d\gamma \nonumber \\
& \quad + 
\frac{q}{p}( \beta_s - 2\beta +1 + \frac{\beta-1}{p} ) \int_{\mathbb{R}^n}  P_s \big[ h_t { x}\cdot \nabla\big( \log\,  \frac{v_{  t}}{\gamma} \big) \big] P_sh_t^{q-1}\, d\gamma \nonumber \\
& \quad + 
\frac{q}{p} (\beta-1) \frac{e^{-2s}}{p} (q-1)
\int_{\mathbb{R}^n} P_s \big[ h_t \nabla \big( \log\, \frac{v_{  t}}{\gamma} \big) \big] \cdot P_s \big[  h_t { x} \big] P_sh_t^{q-2}\, d\gamma\nonumber
\end{align}
and 
\begin{align*}
&
\int_{\mathbb{R}^n} 
\widetilde{R}( h_t, h_t { x} ) 
P_sh_t^{q-2}\, d\gamma\\
& := 
e^{-2s}\int_{\mathbb{R}^n}  P_s\big[ {\rm div}\, \big(  h_t { x} \big) \big]  P_sh_t^{q-1}\, d\gamma
+ 
\int_{\mathbb{R}^n} P_s\big[ \nabla h_t  \big] \cdot P_s\big[ h_t { x}   \big]  P_sh_t^{q-2}\, d\gamma\\
& \quad + e^{-2s}(q-2) 
\int_{\mathbb{R}^n} P_s\big[  h_t { x} \big] \cdot P_s\big[\nabla h_t\big]  P_sh_t^{q-2} \, d\gamma
-
\int_{\mathbb{R}^n} \big| P_s \big[ h_t { x}  \big] \big|^2 P_sh_t^{q-2}\, d\gamma .
\end{align*}
We then apply \eqref{e:ConsSep26F=} with $( \mathbf{f}, g ) = ( h_{  t} { x}, h_{  t} )$ to the term involving $  \big| P_s \big[ h_{  t} { x}  \big] \big|^2$ to see that  
$$
\int_{\mathbb{R}^n} 
\widetilde{R}( h_t, h_t { x} ) 
P_sh_t^{q-2}\, d\gamma
= 0
$$
which shows that 
\begin{align}
\big( \frac{\beta_s}{\beta} \big)^2	\frac1{\beta_s} \frac{pp'}{q} \frac{d}{dt} Q(t) &\ge
\big( \frac{\beta_s}{\beta} \big)^2	\frac1{\beta_s} \frac{pp'}{q} \widetilde{I} \label{e:Q'(t)2} \\
& \quad
+\big(1 + \frac{\beta_s}{\beta} \big)(1-\frac{1}{\beta})
	\int_{\mathbb{R}^n} R( h_t,  h_t \nabla  \big(\log\, \frac{v_{  t}}{\gamma}\big) ) 
	P_sh_t^{q-2}\, d\gamma. \nonumber
\end{align}
Next we focus on the term involving $ P_s\big[ h_t { x} \big] \cdot P_s\big[ h_t \nabla \log\, \frac{v_{  t}}{\gamma} \big] $. Such term on the right-hand side of \eqref{e:Q'(t)2} is 
\begin{align*}
\Upsilon_2(p,s,\beta) \int_{\mathbb{R}^n} P_s \big[  h_t \nabla \big( \log\, \frac{v_{  t}}{\gamma} \big) \big] \cdot P_s \big[  h_t { x} \big] P_sh_t^{q-2}\, d\gamma, 
\end{align*}
where
\begin{equation}\label{e:Upsilon2}
\Upsilon_2(p,s,\beta):= \big( \frac{\beta_s}{\beta} \big)^2 \frac{1}{\beta_s} \frac{pp'}{q} \frac{q}{p} (\beta-1) \frac{e^{-2s}}{p}(q-1) - ( 1+ \frac{\beta_s}{\beta} )(1-\frac1{\beta})=-(1-\frac1\beta) .
\end{equation}
Hence we may apply \eqref{e:ConsSep26F=} with $ (\mathbf{f},g) = (h_t \nabla \log\, \frac{v_{  t}}{\gamma}, h_t) $ to see that 
\begin{align}
\big( \frac{\beta_s}{\beta} \big)^2	\frac1{\beta_s} \frac{pp'}{q}\frac{d}{dt}Q(t) 
& \ge 
\big( \frac{\beta_s}{\beta} \big)^2	\frac1{\beta_s} \frac{pp'}{q} \overline{I} \label{e:Sep27_2} \\
& \quad +\big(1 + \frac{\beta_s}{\beta} \big)(1-\frac{1}{\beta})
	\int_{\mathbb{R}^n} \overline{R}(h_t, h_t \nabla \big(\log\, \frac{v_{  t}}{\gamma}\big) )
	P_sh_t^{q-2}\, d\gamma, \nonumber
\end{align}
where 
\begin{align*}
\overline{I}
& =:
q(\beta-\beta_s) \int_{\mathbb{R}^n}  P_s \big[ h_t \Delta \big( \log\,  \frac{v_{  t}}{\gamma} \big) \big] P_sh_t^{q-1}\, d\gamma \nonumber \\
& \quad +\frac qp (\beta-\beta_s) \int_{\mathbb{R}^n}   P_s \big[ h_t \Gamma( \log\,  \frac{v_{  t}}{\gamma} ) \big] P_sh_t^{q-1}\, d\gamma \nonumber \\
& \quad + 
\frac{q}{p}( \beta_s - 2\beta +1 + \frac{\beta-1}{p} ) \int_{\mathbb{R}^n} P_s \big[ h_t { x}\cdot \nabla \big( \log\,  \frac{v_{  t}}{\gamma} \big) \big] P_sh_t^{q-1}\, d\gamma \nonumber \\
& \quad + 
\frac{q}{p} (\beta-1) \frac{e^{-2s}}{p} (q-1)
{ ( 1 + e^{-2s}(q-2) ) \int_{\mathbb{R}^n} P_s \big[  \nabla  \big( \log\, \frac{v_{  t}}{\gamma} \big) \big] \cdot P_s \big[ \nabla h_t \big] P_sh_t^{q-2}\, d\gamma}
\nonumber\\
& \quad + 
\frac{q}{p} (\beta-1) \frac{e^{-2s}}{p} (q-1)
{ e^{-2s} \int_{\mathbb{R}^n} P_s \big[ {\rm div}\,  \big( h_t \nabla \big( \log\, \frac{v_{  t}}{\gamma} \big) \big) \big]  P_sh_t^{q-1}\, d\gamma}.
\nonumber
\end{align*}
	and 
\begin{align*}
&\int_{\mathbb{R}^n} \overline{R}( h_t, h_t  \nabla \big(\log\, \frac{v_{  t}}{\gamma}\big) )
	P_sh_t^{q-2}\, d\gamma\\
& := 
\int_{\mathbb{R}^n}   P_s\big[ h_t  { x} \cdot\nabla  \big(\log\, \frac{v_{  t}}{\gamma}\big) \big]
	P_sh_t^{q-1}\, d\gamma \\
& \quad -
{
\int_{\mathbb{R}^n}  P_s\big[ \nabla h_t  \big] \cdot P_s\big[h_t \nabla \big(\log\, \frac{v_{  t}}{\gamma}\big) \big]
	P_sh_t^{q-2}\, d\gamma
	}\\
& \quad -
{ e^{-2s} 
\int_{\mathbb{R}^n}   P_s\big[ {\rm div}\,  \big( h_t \nabla \big(\log\, \frac{v_{  t}}{\gamma}\big) \big) \big]
	P_sh_t^{q-1}\, d\gamma
}	\\
& \quad -
{ e^{-2s} (q-2) 
\int_{\mathbb{R}^n}   P_s\big[  h_t \nabla \big(\log\, \frac{v_{  t}}{\gamma}\big)  \big] \cdot P_s\big[ \nabla h_t \big]
	P_sh_t^{q-2} \, d\gamma
}	\\
& \quad - 
(1-e^{-2s}) 
\int_{\mathbb{R}^n}  \big( P_s h_t P_s\big[{\rm div}\,  \big(  h_t \nabla \big(\log\, \frac{v_{  t}}{\gamma}\big) \big) \big] \\
& \qquad \qquad \qquad \qquad  -
P_s\big[ \nabla h_t \big] \cdot P_s\big[   h_t \nabla \big(\log\, \frac{v_{  t}}{\gamma}\big) \big]
\big)
	P_sh_t^{q-2}\, d\gamma. 
\end{align*}

Finally we focus on the term involving $ P_s\big[ h_t { x} \cdot \nabla \log\, \frac{v_{  t}}{\gamma} \big] $. 
The contribution of such term on the right-hand side of \eqref{e:Sep27_2} is
$$
\Upsilon_3(p,s,\beta) 
\int_{\mathbb{R}^n}  P_s\big[ h_t{ x}\cdot \nabla \log\, \frac{v_{  t}}{\gamma} \big]P_sh_{  t}^{q-1}\, d\gamma,
$$
where
\begin{equation}\label{e:Upsilon3}
\Upsilon_3(p,s,\beta)
=
\big( p-2 +e^{-2s} + \frac{(1-e^{-2s})^2}{p} (1-\frac1{\beta}) \big) \frac1{p-1} (1-\frac1{\beta}). 
\end{equation}
We may apply \eqref{e:ConsSep26G=} with $(f,\mathbf{g}) = ( h_t, h_t \nabla \log\, \frac{v_{  t}}{\gamma} )$ to this term. This, along with the fact that $ (\beta - 1) \frac{1-e^{-2s}}{p} = \beta - \beta_s $, allows us to conclude that 
$$
\frac{d}{dt} Q(t) \ge 0. 
$$
Hence we see that $Q(t)$ is non-decreasing and, comparing $t \to 0$ and $t  \to \infty$, we obtain
$$
\big\| P_s\big[ \big(\frac{v}{\gamma}\big)^\frac1p\big] \big\|_{L^q(d\gamma)}^q
\le 
\big\| P_s\big[ \big(\frac{\gamma_\beta}{\gamma}\big)^\frac1p\big] \big\|_{L^q(d\gamma)}^q
\bigg( \int_{\mathbb{R}^n} \frac{v}{\gamma}\, d\gamma\bigg)^\frac{q}{p},
$$
since 
$$
\lim_{t\to\infty} v_t = \bigg(\int_{\mathbb{R}^n} \frac{v}{\gamma}\, d\gamma\bigg)\gamma_\beta. 
$$

By a careful examination of the above proof of \eqref{e:RegHC} (in the case where $\beta > 1$), one can see that the argument works just as well if $\beta < 1$ and $v$ is $\beta$-semi-log-concave. Similarly, for the reverse inequality \eqref{e:RegRevHC}, one can also follow the above proof and obtain the claimed results in Theorem \ref{t:MainHyper}(2)\footnote{See also the forthcoming discussion in Section \ref{S3.4} where the existence of certain closure properties (which underpin the monotonicity of $Q$) is elucidated.}. 
\end{proof}

\subsection{Proof of Theorem \ref{t:MainLSI}: log-Sobolev inequality } 
\label{S3.3}
The following lemma bridges our regularised hypercontractivity inequality to the LSI.  
\begin{lemma}\label{l:Hyper->LogSob}
Let $q(s) = 1 + e^{2s} \in (2,\infty)$ for $s\ge 0$. 
Assume that $f\ge0$ satisfy 
\begin{equation}\label{e:30Apr-1}
\big\| P_s\big[ f^\frac12 \big]  \big\|_{L^{q(s)}(\gamma)} \le \psi(s) \| f\|_{L^1(\gamma)}^\frac12,\;\;\; \forall s>0,
\end{equation}
for some positive function $\psi(s)\in C^1((0,\infty))$ such that $\psi(0)=1$. Then for such $f$, we have that 
\begin{equation}\label{e:30Apr-2}
{\rm Ent}_{\gamma} (f) \le \frac1{2} {\rm I}_{\gamma}(f) + 2 \psi'(0) \|f\|_{L^1(\gamma)}.
\end{equation}
\end{lemma}

\begin{proof}
We modify the argument of \cite[Theorem 5.2.3]{BGL} so that the dependence on $\psi(s)$ can be tracked. 
Define 
$$
\Lambda(s):=  \frac{\big\| P_s\big[ f^\frac12 \big]  \big\|_{L^{q(s)}(\gamma)} }{\psi(s)},\;\;\; s\ge0. 
$$
Note that 
$
\Lambda(0) =  \|f\|_{L^1(\gamma)}^\frac12 
$
since $\psi(0)=1$ and $q(0) = 2$,  and hence the assumption \eqref{e:30Apr-1} implies that 
$
\Lambda'(0) \le 0.
$
Now notice from direct calculations that 
$$
\Lambda'(0)
= 
\frac{d}{ds} \big( \big\| P_s\big[ f^\frac12 \big] \big\|_{q(s)} \big)|_{s=0} - \big\| f^\frac12 \big\|_{2} \psi'(0).
$$
In order to compute the first term, we invoke two formulae (see \cite{BGL}):  
\begin{align*}
&\frac{d}{dq} \| h \|_{L^q(\gamma)}^q = \frac1q \bigg( {\rm Ent}_{\gamma} (h^q) + \bigg(\int_{\mathbb{R}^n} h^q \, d\gamma\bigg) \log\, \bigg( \int_{\mathbb{R}^n} h^q \, d\gamma\bigg)  \bigg) ,\\
& 
\frac{d}{ds} \int_{\mathbb{R}^n} \big( P_sf  \big)^q\, d\gamma 
= -q(q-1) \int_{\mathbb{R}^n} \big( P_sf \big)^{q-2}  |\nabla P_sf|^2 \, d\gamma. 
\end{align*}
From the first formula, we see that
\begin{equation}\label{e:30Apr-3}
\frac{d}{dq} \| h \|_{L^q(\gamma)}= \frac1{q^2} \| h\|_q^{1-q} {\rm Ent}_{\gamma} (h^q) .
\end{equation}
On the other hand, we see from $\frac{d}{ds}q(s) = 2e^{2s}$,  \eqref{e:30Apr-3} and the second formula that 
\begin{align*}
\frac{d}{ds} \big( \big\| P_s\big[ f^\frac12 \big] \big\|_{q(s)} \big)
=& 
-(q(s)-1) \big\| P_s\big[ f^\frac12 \big]  \big\|_{q(s)}^{1- q(s)} \bigg( \int_{\mathbb{R}^n} \big( P_s\big[f^\frac12 \big] \big)^{q(s)-2}  |\nabla P_s\big[f^\frac12 \big]|^2 \, d\gamma   \bigg) \\
&+
2e^{2s} \frac1{q(s)^2} \big\| P_s\big[ f^\frac12 \big] \big\|_{q(s)}^{1-q(s)} {\rm Ent}_{\gamma} ( P_s\big[ f^\frac12 \big]^{q(s)}).
\end{align*}
This shows that 
\begin{equation}\label{e:30Apr-4}
\frac{d}{ds} \big( \big\| P_s\big[ f^\frac12 \big] \big\|_{q(s)} \big)|_{s=0}
= 
\frac{1}2 \big\|  f \big\|_{1}^{-\frac12}
\bigg( -\frac1{2} {\rm I}_{\gamma}(f) + {\rm Ent}_{\gamma} ( f ) \bigg)
\end{equation}
and hence 
\begin{align*}
\Lambda'(0)
& =
- \big\|  f^\frac12  \big\|_{2}^{-1} \bigg( \int_{\mathbb{R}^n}  |\nabla \big[f^\frac12 \big]|^2 \, d\gamma   \bigg) 
+
2 \frac1{4} \big\|  f^\frac12 \big\|_{2}^{-1} {\rm Ent}_{\gamma} ( \big[ f^\frac12 \big]^{2})
-
\| f \|_{1}^\frac12 \psi'(0) \\
& = 
\frac{1}2 \big\|  f^\frac12  \big\|_{2}^{-1}
\bigg( -\frac1{2} {\rm I}_{\gamma}(f) + {\rm Ent}_{\gamma} ( f )  - 2 \psi'(0) \|f\|_1 \bigg)
\end{align*}
which concludes the proof since $\Lambda'(0)\le 0$. 
\end{proof}

\begin{proof}[Proof of Theorem \ref{t:MainLSI}] 
Take an arbitrary $v$ satisfying the assumptions in Theorem \ref{t:MainLSI}.  
In view of the hypercontractivity inequality in Theorem \ref{t:MainHyper}(1), we apply Lemma \ref{l:Hyper->LogSob} with  
$$
\psi(s) := \big\| P_s\big[ \big( \frac{\gamma_\beta}{\gamma} \big)^\frac12 \big]  \big\|_{L^{q(s)}(\gamma)} = \big\| P_s\big[ f_*^\frac12 \big]  \big\|_{L^{q(s)}(\gamma)}, 
$$
where $f_*:=  \frac{\gamma_\beta}{\gamma}$. 
From \eqref{e:30Apr-4}, we notice that 
\begin{align*}
\psi'(0) & = \frac{1}2 \big\|  f_* \big\|_{1}^{-\frac12}
\bigg( -\frac1{2} {\rm I}_{\gamma}(f_*) + {\rm Ent}_{\gamma} ( f_* ) \bigg) \\
& =
\frac{n}{2}\bigg(\log \beta - 1 + \frac{1}{\beta} \bigg).
\end{align*}
Since $\|f_*\|_{L^1(\gamma)}=1$,  we conclude \eqref{e:ELS} for those $v$ satisfying the conditions in Theorem \ref{t:MainHyper}(1).  
\end{proof}

\subsection{Proof of Corollary \ref{c:FPversion}}
Formally, Lemma \ref{l:LogConVari} is enough to derive Corollary \ref{c:FPversion} from Theorems \ref{t:MainHyper} and \ref{t:MainLSI}. The argument becomes rigorous once we address the assumption $v \in L^2(\gamma_\beta^{-1})$. 
\begin{proof}[Proof of Corollary \ref{c:FPversion}]
	Take any $v\in {\rm FP}(\beta)$, that is, $v = v_*(t_*,\cdot)$ and $v_*$ is a solution to \eqref{e:v*} with some non-negative finite measure $\mu$. 
	For the purpose of proving \eqref{e:RegHC}, we may suppose that $\mu$ has compact support without loss of generality thanks to  a standard approximation argument and the monotone convergence theorem. 
	Then the explicit form of the solution 
	$$
	v_*(t,x) = \frac1{\big(4\pi \beta (1-e^{-2t})\big)}\int_{\mathbb{R}^n} e^{- \frac{|x-e^{-t}y|^2}{4\beta(1-e^{-2t})}}\, d\mu(y)
	$$
	yields that 
	$$
	v(x) = v_*(t_*,x) \le C_\beta \mu(\mathbb{R}^n) \big( \1_{|x|\le 100r_\mu} + e^{- \frac1{2\beta}|x|^2} \big) \in L^2(\gamma_\beta^{-1}),
	$$
	where $r_\mu>0$ is a radius of support of $\mu$. Hence we may apply Theorems \ref{t:MainHyper}(1) and \ref{t:MainLSI}(1) to obtain \eqref{e:RegHC} and \eqref{e:ELS}, respectively. 
\end{proof}

\subsection{Closure properties} \label{S3.4}
By making use of explicit formulae for the flows, it is possible to derive a somewhat stronger statement than the monotonicity of $Q(t)$ (a ``closure property") in terms of supersolutions of Fokker--Planck equations. In fact, we observe 
\begin{align}
&(1-e^{-2s})P_s\big[ h_t \Delta \big( \log\,  \frac{v_t}{\gamma} \big) \big] \label{e:UseExplicitForm1}\\
& = 
P_s\big[h_t x \cdot \nabla \big( \log\,  \frac{v_t}{\gamma} \big)\big] 
-
e^{-s} x \cdot P_s\big[h_t \nabla \big( \log\,  \frac{v_t}{\gamma} \big)\big] 
-
\frac{1-e^{-2s}}p P_s\big[ h_t \Gamma\big( \log\, \frac{v_t}{\gamma} \big)\big], \nonumber
\end{align}
and
\begin{align}
&\frac{(1-e^{-2s})^2}pP_s\big[ h_t \Delta \big( \log\,  \frac{v_t}{\gamma} \big) \big] \label{e:UseExplicitForm2}\\
& = 
P_s\big[ |x|^2 h_t  \big] 
-
e^{-2s}x^2P_s h_t 
-
(1-e^{-2s})P_s h_t \nonumber \\
&\quad -
2\frac{(1-e^{-2s})}p e^{-s}x \cdot P_s\big[ h_t \nabla \big(\log\, \frac{v_t}{\gamma} \big) \big]
-
\frac{(1-e^{-2s})^2}{p^2} P_s\big[ h_t \Gamma\big( \log\,  \frac{v_t}{\gamma} \big)\big], \nonumber
\end{align}
for the term involving $h_t\Delta\log\, \frac{v_t}{\gamma}$, and also 
\begin{align}
&\big|P_s\big[ x h_t  \big]\big|^2  \label{e:UseExplicitForm3}\\
& = 
e^{-2s} |x|^2 \big( P_s h_t \big)^2 
+
2e^{-s} \frac{1-e^{-2s}}{p} P_s h_t x \cdot P_s\big[ h_t \nabla \big( \log\, \frac{v_t}{\gamma}\big) \big]\nonumber\\
& \quad + 
\big( \frac{1-e^{-2s}}{p} \big)^2 \big| P_s\big[ h_t \nabla \big( \log\, \frac{v_t}{\gamma}\big) \big] \big|^2,\nonumber
\end{align}
and
\begin{align}
& P_s\big[ h_t \nabla \big( \log\, \frac{v_t}{\gamma}\big) \big]\cdot P_s\big[ x h_t  \big] \label{e:UseExplicitForm4}\\
& = 
e^{-s} P_sh_t x\cdot P_s\big[ h_t \nabla \big( \log\, \frac{v_t}{\gamma}\big) \big] + \frac{1-e^{-2s}}{p} \big| P_s\big[ h_t \nabla \big( \log\, \frac{v_t}{\gamma}\big) \big] \big|^2.\nonumber
\end{align} 
Applying  \eqref{e:UseExplicitForm1}--\eqref{e:UseExplicitForm4} with changes of variables, we can check that 
$I$, $N$, $R(h_t,h_tx)$, and $R(h_t,h_t \nabla (\log\, \frac{v_t}{\gamma}))$ appeared in Proposition \ref{Prop:Uhmmm2} all vanish identically. Combining this fact and Proposition \ref{Prop:Uhmmm2}, we may derive the following. 
\begin{theorem}\label{t:Closure}
Under the assumptions of Theorem \ref{t:MainHyper}(1) we have
\begin{equation}\label{e:Closure}
\partial_t v_t = \mathcal{L}^*_\beta v_t,\;\;\; v_0 = v\;\;\; \Rightarrow \;\;\; \partial_t \widetilde{v}_t \ge \mathcal{L}^*_{\beta_s} \widetilde{v}_t,
\end{equation}
where $\beta_s>0$ and $\widetilde{v}_t$ are given by \eqref{e:beta2} and \eqref{e:vtilde}, respectively. 
\end{theorem}
It is easy to see that Theorem \ref{t:Closure} implies the monotonicity of $Q(t)$ and thus is a somewhat stronger result compared with Theorem \ref{t:MainHyper}(1). 

Furthermore the property \eqref{e:Closure} can be slightly generalised in the following form 
\begin{equation}\label{e:Closure2}
\partial_t v_t \ge \mathcal{L}_\beta^* v_t,\;\;\; v_0 = v 
\;\;\;\Rightarrow \;\;\; 
\partial_t \widetilde{v}_t \ge \mathcal{L}_{\beta_s}^* \widetilde{v}_t,
\end{equation}
under appropriate smoothness and decay conditions on $v$. Observations of this nature emerged in \cite{BB,BBCrell} in the context of the sharp Young convolution inequality and Brascamp--Lieb inequalities. We follow these papers and refer to \eqref{e:Closure2} as a \textit{closure property} (for supersolutions of Fokker--Planck equations) associated with the operation $v_t \mapsto \widetilde{v}_t$. A closely related closure property for supersolutions of the gaussian heat equation was investigated in \cite{ABBMMS} and, in fact, the result in \cite{ABBMMS} yields \eqref{e:Closure2} in the special case $\beta=1$. 
However there are critical differences distinguishing the case $\beta =1$ and the case $\beta\neq1$.  The most prominent and phenomenological one is the fact that the closure property \eqref{e:Closure} or \eqref{e:Closure2} for $\beta\neq1$ cannot be true unless one imposes  some additional structure (such as $\beta$-semi-log-convexity/concavity). On the other hand, the closure property for $\beta=1$ holds true in general (under very mild decay condition on $v$). It is also clear that the closure property for $\beta=1$ is not enough to obtain deficit estimate \eqref{e:RegHC} and, as we have seen, significant work is required to handle the case $\beta \neq 1$.   


Regarding the reverse inequality \eqref{e:RegRevHC}, in which case we assume $-\infty < q< p <1$, one can very closely follow the argument for $1<p<q$ to see that  
$$
\frac{pp'}{q} \big( \partial_t \widetilde{v}_t - \mathcal{L}_{\beta_s}^* \widetilde{v}_t \big)\ge0
$$ 
as long as $\beta_s >0$.  Notice also that $\beta_s \ge1$ holds in the following cases: (i) $\beta \ge 1$ and $pq >0$, or (ii) $0< \beta \le 1$ and $p q <0$.  As a consequence, we obtain the following. 
\begin{theorem}\label{P:ClosureFPRev}
Suppose that we are under the assumptions of Theorem \ref{t:MainHyper}(2).
\begin{enumerate}
\item 
In the case  $p<0$,  we suppose $\beta > 1$. Then 
$$
\partial_t v_t = \mathcal{L}_{\beta_s}^* v_t ,\;\;\; v_0=v \;\;\; \Rightarrow\;\;\; \partial_t \widetilde{v}_t \ge \mathcal{L}_{\beta_s} \widetilde{v}_t. 
$$
\item 
In the case $0< p <  1-e^{-2s} $,  we suppose $0<\beta < 1$. Then 
$$
\partial_t v_t = \mathcal{L}_{\beta_s}^* v_t ,\;\;\; v_0=v \;\;\; \Rightarrow\;\;\; \partial_t \widetilde{v}_t \ge \mathcal{L}_{\beta_s} \widetilde{v}_t. 
$$
\item 
In the case  $1-e^{-2s}< p <1$,  we suppose $\beta > 1$. Then 
$$
\partial_t v_t = \mathcal{L}_{\beta_s}^* v_t ,\;\;\; v_0=v \;\;\; \Rightarrow\;\;\; \partial_t \widetilde{v}_t \le \mathcal{L}_{\beta_s} \widetilde{v}_t. 
$$
\end{enumerate}
Here $\beta_s>0$ and $\widetilde{v}_t$ are given by \eqref{e:beta2} and \eqref{e:vtilde}, respectively. 
\end{theorem}
The above theorem can be viewed as a somewhat stronger version of Theorem \ref{t:MainHyper}(2).

\subsection{Further remarks on the proof of Theorem \ref{t:MainHyper}} 
\label{section:furtherremarksflow} The Fokker--Plank equation $\partial_t v = \mathcal{L}_\beta^* v$ has a close connection\footnote{For example, when $n=1$ it is straightforward to verify that given a solution $v$ to $\partial_t v = \mathcal{L}_\beta^* v$ on $(0,\infty) \times \mathbb{R}$, if we define $u$ on $(0,\infty) \times \mathbb{R}$ by
\[
u(\tau,y) = \frac{1}{(2\tau + 1)^{1/2}} v\bigg(\frac{1}{2}\log(2\tau + 1), \frac{y}{(2\tau + 1)^{1/2}}\bigg) 
\]
then $\partial_\tau u = \beta \partial_{yy}u$ holds with $u(0) = v(0)$. 
} with the classical heat equation $\partial_t u = \beta \Delta u$ and this raises the obvious question of why we have opted to use Fokker--Planck flow as the basis for our main results. The short yet somewhat opaque answer is that Fokker--Planck flow seems to be significantly better suited to 
framework.

To explain this in more detail, for the sake of simplicity, let us consider the one-dimensional case $n=1$. Our discussion here is based on the fact that one may write
\begin{equation*}
\| P_s[f^\frac{1}{p}] \|_{L^q(\gamma)} = C \| (f\gamma)^\frac{1}{p} * \gamma_\lambda  \|_{L^q(dx)}
\end{equation*}
for suitably chosen constants  $C$ and $\lambda$ (depending on $p$ and $q$); see, for example, \cite[Theorem 5]{BecknerAnnals}. As a result, the claim regarding the regularised hypercontractivity inequality \eqref{e:RegHC} in Theorem \ref{t:MainHyper}(1) is equivalent to
\begin{equation} \label{e:convolutionversion}
\| v^\frac{1}{p} * \gamma_\lambda  \|_{L^q(dx)} \leq \| \gamma_\beta^\frac{1}{p} * \gamma_\lambda  \|_{L^q(dx)}  
\end{equation}
for $v$ satisfying the assumptions in Theorem \ref{t:MainHyper}(1) which are normalised by $\int_{\mathbb{R}} v \, dx = 1$. It is tempting to try to prove \eqref{e:convolutionversion}  using classical heat flow associated with the Laplacian on $\mathbb{R}$. That $v = \gamma_\beta$ is a maximiser indicates that it would be natural to evolve $v$ according to the classical heat equation
\begin{equation} \label{e:betaover2}
\partial_t u = \frac{\beta}{2} \partial_{xx} u, \quad u(t_0) = v 
\end{equation}
since we have for any fixed $t_0\ge0$ that
\[
(t - t_0)^\frac{1}{2} u(t,(t - t_0)^\frac{1}{2}y) \to  \gamma_\beta(y) \quad (t \to \infty).
\]
Moreover, it follows from \cite[Theorem 6]{BB} that 
\[
t \mapsto t^{\frac{1}{2}(\frac{1}{p} + \frac{1}{r} - 1 - \frac{1}{q})} \| u_1(t,\cdot)^{\frac{1}{p}} * u_2(t,\cdot)^{\frac{1}{r}}  \|_{L^q(dx)}
\]
gives rise to a non-decreasing functional if $\frac{1}{p} + \frac{1}{r} \geq 1 + \frac{1}{q}$ and the $u_j$ are non-negative solutions of
\[
\partial_t u_j = \frac{\beta_j}{2} \partial_{xx} u_j
\]
for appropriate diffusion parameters $\beta_1, \beta_2 > 0$ (we refer the reader to \cite[Section 1.3]{BB} for the restrictions on the diffusion parameters).
However, as far as we can tell, it does not seem clear how to make use of this result to prove \eqref{e:convolutionversion} under the conditions of Theorem \ref{t:MainHyper}(1)\footnote{If one is only interested in proving the rather weaker result in Corollary \ref{c:FPversion} (where $v \in \textrm{FP}(\beta)$ and $\beta \geq 1$), then an approach based on \eqref{e:convolutionversion} and the results in \cite{BB} is possible by a careful choice of $r$ and the $\beta_j$.}. The issue seems to be related to the preservation of $\beta$-semi-log-subharmonicity/$\beta$-semi-log-concavity along the lines of Lemma \ref{l:LogPreserve}, and mis-match of that $\beta$ parameter with the $\frac{\beta}{2}$ diffusion parameter in \eqref{e:betaover2}.

In the above sense, our framework seems most harmonious with Fokker--Planck rather than classical heat flow. We also suspect that Fokker-Planck flow seems the more natural candidate for generalising the Ornstein--Uhlenbeck semigroup to more general Markov semigroups; see the forthcoming Section \ref{S5} for further details in this direction.

\subsection{Proof of Corollary \ref{Cor:Matrix}: Hypercontractivity and the LSI}
Let us first prove the statements in part (1) concerning hypercontractivity and the LSI. For this, in fact, it suffices to show \eqref{e:HCMatrixConv} since we may then obtain \eqref{e:LSIMatrixConv} via Lemma \ref{l:Hyper->LogSob}. 
Moreover, we may assume $B = {\rm diag}\, (\beta_1,\ldots,\beta_n)$ since matters are rotationally invariant. 

For the sake of simplicity, we discuss the case $n=2$; the case $n\ge3$ can be obtained by an induction argument. 
Therefore,  if we write $u:= \frac{v}{\gamma}$, then our goal boils down to show
\begin{equation}
\big\| P_s \big[ u^\frac1p \big] \big\|_{ L^q(\gamma) }
\le 
\prod_{\substack{i=1,2 \\ \beta_i \geq 1}}
\beta_i^{ \frac{1}{2p'} } (\beta_i)_s^{ - \frac{1}{2q'} }
\bigg(
\int_{\mathbb{R}^2} 
u\, d\gamma
\bigg)^{ \frac{1}{p} }  
\end{equation}
provided 
\begin{equation}\label{e:Diagonal} 
\nabla^2 \log\, u \ge  {\rm diag}\, (1-\beta_1^{-1},1-\beta_2^{-1}).
\end{equation}  
For this, we employ the standard tensorising argument to reduce matters to the one-dimensional result in Theorem \ref{t:MainHyper}(1). 

Firstly, we consider the case $\beta_1 \geq 1$ and $\beta_2 \ge1$. Write 
\begin{align*}
P_s \big[ u^{ \frac1p } \big](x_1,x_2)
&= 
\int_{\mathbb{R}} \bigg( \int_{\mathbb{R}} 
u( y_1,y_2 )^\frac1p p_s( x_1, y_1 )\, dy_1\bigg) p_s( x_2, y_2 ) \,  dy_2, 
\end{align*}
where the integral kernel is given by 
$$
p_s(x,y) := \frac{1}{ ( 2\pi( 1-e^{-2s} ) )^\frac12 } e^{ - \frac12 \frac{ | e^{-s}x - y |^2 }{ 1-e^{-2s} } },\;\;\; (x,y) \in \mathbb{R}\times \mathbb{R}.  
$$
Then we use Minkowski's integral inequality to see that 
\begin{align*}
&\big\| P_s \big[ u^\frac1p \big] \big\|_{ L^q(\gamma) }^q \\
&= 
\int_{\mathbb{R}} \bigg\|
\int_{\mathbb{R}} \bigg( \int_{\mathbb{R}} 
u( y_1,y_2 )^\frac1p p_s( x_1, y_1 )\, dy_1\bigg) p_s( x_2, y_2 ) \,  dy_2
\bigg\|_{L^q_{x_1}(\gamma)}^q\, d\gamma(x_2)\\
&\le
\int_{\mathbb{R}} 
\bigg( 
\int_{\mathbb{R}} 
\bigg\|
\int_{\mathbb{R}} 
u( y_1,y_2 )^\frac1p p_s( x_1, y_1 )\, dy_1
\bigg\|_{L^q_{x_1}(\gamma)}
p_s( x_2, y_2 ) \,  dy_2
\bigg)^q
\, d\gamma(x_2).
\end{align*}
For each fixed $y_2$, we know that 
$$
\partial_{y_1}^2 \log\, u(y_1,y_2) \ge 1 - \frac{1}{\beta_1}
$$
from the assumption \eqref{e:Diagonal}.  Hence we may apply Theorem \ref{t:MainHyper}(1) on $\mathbb{R}$ to have 
\begin{align*}
\bigg\|
\int_{\mathbb{R}} 
u( y_1,y_2 )^\frac1p p_s( x_1, y_1 )\, dy_1
\bigg\|_{L^q_{x_1}(\gamma)}  &\le 
\beta_1^{ \frac{1}{2p'} } (\beta_1)_s^{ - \frac{1}{2q'} }
\bigg( 
\int_{\mathbb{R}}
u(x_1,x_2)\, d\gamma(x_1)
\bigg)^\frac1p \\
&=: 
\beta_1^{ \frac{1}{2p'} } (\beta_1)_s^{ - \frac{1}{2q'} }
V(x_2)^\frac1p 
\end{align*}
which implies that 
\begin{align*}
&\big\| P_s \big[ u^\frac1p \big] \big\|_{ L^q(\gamma) }^q 
\le 
\big( 
\beta_1^{ \frac{1}{2p'} } (\beta_1)_s^{ - \frac{1}{2q'} }
\big)^q
\int_{\mathbb{R}} 
\bigg( 
\int_{\mathbb{R}} 
V(x_2)^\frac1p 
p_s( x_2, y_2 ) \,  dy_2
\bigg)^q
\, d\gamma(x_2).
\end{align*}
We then notice from \eqref{e:Diagonal} that $ \partial_{x_2}^2 \log\,  u(x_1,x_2) \ge 1 - \frac{1}{\beta_2} $ for each $x_1$ and hence 
$$
\partial_{x_2}^2 \log\, V \ge 1 - \frac{1}{\beta_2}
$$
since the superposition preserves the log-convexity; see for instance the proof of Lemma 1.3 in \cite{EL}. 
This allows us to apply Theorem \ref{t:MainHyper}(1) on $\mathbb{R}$ again to conclude 
\begin{align*}
\big\| P_s \big[ u^\frac1p \big] \big\|_{ L^q(\gamma) }^q 
\le& 
\big( 
\prod_{i=1,2}
\beta_i^{ \frac{1}{2p'} } (\beta_i)_s^{ - \frac{1}{2q'} }
\big)^q
\bigg(
\int_{\mathbb{R}} 
V(x_2)\, d\gamma(x_2)
\bigg)^{ \frac{q}{p} }\\
=& 
\big( 
\prod_{i=1,2}
\beta_i^{ \frac{1}{2p'} } (\beta_i)_s^{ - \frac{1}{2q'} }
\big)^q
\bigg(
\int_{\mathbb{R}^2} 
u\, d\gamma
\bigg)^{ \frac{q}{p} }
\end{align*}
as desired.

The case where either $\beta_1$ or $\beta_2$ (or both) are less than or equal to $1$ may be handled in a similar manner. For example, if $\beta_1 \leq 1$ then rather than applying Theorem \ref{t:MainHyper}(1) to $y_1 \mapsto u(y_1,y_2)$, we simply apply classical hypercontractivity with constant $1$.

Statement (2) can be proved in a similar way except a little more care is needed at the point where we obtained the semi-log-convexity of $V$.  To show statement (2),  we need to ensure the semi-log-concavity of $V$. Although superposition does not always preserve the log-concavity, we can obtain the semi-log-concavity of $V$ as in the argument in Lemma \ref{l:LogPreserve} based on the Pr\'{e}kopa--Leindler inequality.  \qed

\section{Talagrand's inequality : Proof of Theorem \ref{t:MainTalagrand} and Corollary \ref{Cor:Matrix}}\label{S3.5}

For Talagrand's inequality in both Theorem \ref{t:MainTalagrand} and Corollary \ref{Cor:Matrix}, the input $v$ is either semi-log-convex or semi-log-concave and so it turns out that, via a tensorisation argument, it suffices to prove the one-dimensional case in Theorem \ref{t:MainTalagrand}. In particular, the tensorisation argument yields the Talagrand inequalities appearing in Corollary \ref{Cor:Matrix} for general $n$, and these are stronger than the corresponding statements in Theorem \ref{t:MainTalagrand} when $n \geq 2$. 

We first fix the notation and some preliminaries. In \eqref{e:TalCl}, the definition of the Wasserstein distance $W_2$ is as follows. First, take two probability measures $\mu,\nu$  whose second moment is finite. Let $\Pi (\mu,\nu)$ be the set of all couplings of $\mu$ and $\nu$, namely probability measures $\pi$ on $\mathbb{R}^n\times \mathbb{R}^n$ such that $\pi(A\times \mathbb{R}^n) = \mu(A)$ and $\pi(\mathbb{R}^n\times A) = \nu(A)$ hold for all Borel sets $A\subset\mathbb{R}^n$. Then 
\begin{equation}\label{e:DefWassDist}
W_2(\mu,\nu)^2:= \inf_{\pi \in \Pi(\mu,\nu)} \int_{\mathbb{R}^n} |x-y|^2\, d\pi(x,y). 
\end{equation}
It will be important that $W_2(\mu,\nu)$ can be represented by using \emph{Brenier's map} (pushing forward $\mu$ onto $\nu$) which we now describe. For two given probability measures $\mu, \nu$ on $\mathbb{R}^n$ which are absolutely continuous with respect to the $n$-dimensional Lebesgue measure,  Brenier's theorem ensures the existence of a map $T: \mathbb{R}^n\to\mathbb{R}^n$, which is called Brenier's map from $\mu$ to $\nu$,  such that 
\begin{equation}\label{e:Transport}
\int_{\mathbb{R}^n} h(T(x))\, d\mu(x) = \int_{\mathbb{R}^n} h(x)\, d\nu(x)
\end{equation}
holds for all $h \in L^\infty(\mathbb{R}^n)$.  If we slightly abuse notation and write $d\mu (x) = \mu(x)dx$ and $d\nu(x) = \nu(x)dx$ assuming $\mu,\nu$ have densities with respect to Lebesgue measure, then \eqref{e:Transport} leads to the Monge--Amp\'{e}re equation
\begin{equation}\label{e:MA}
\mu(x) = \nu(T(x))  {\rm det}\,  \nabla T(x)
\end{equation} 
for $\mu$-a.e. $x\in\mathbb{R}^n$. The reader is referred to \cite{BGL, Vil1, Vil2} for further details. 

In terms of Brenier's map $T$ pushing forward $\mu$ onto $\nu$, we have 
\begin{equation}\label{e:W2Brenier}
	W_2(\mu,\nu)^2 = \int_{\mathbb{R}^n} |x - T(x)|^2\, d\mu(x). 
\end{equation} 
We refer the reader to \cite{Vil1} for further details.

As promised, we first prove Theorem \ref{t:MainTalagrand} for $n=1$. 
\begin{proposition}\label{t:RegTal}
	Let $\beta>0$ and suppose the twice differentiable $v:\mathbb{R}\to(0,\infty)$ satisfies $\int_{\mathbb{R}} v\, dx =1$ and $\int_{\mathbb{R}} x^2v\, dx<\infty$. 
	
(1) If $\beta > 1$ and $v$ satisfies 
		$$
		0\ge (\log\, v)'' \ge - \frac1{\beta},
		$$
then we have \eqref{e:TIMWeak}.
		
(2) If $0<\beta < 1$ and $v$ is $\beta$-semi-log-concave, then we have \eqref{e:TIMWeak}.
\end{proposition}

\begin{proof}
First of all, we recall that \eqref{e:TIMWeak} is equivalent to 
$$
		\frac12 W_2(\gamma,v)^2 - {\rm Ent}_\gamma\big( \frac{v}{\gamma} \big)
		\le 
		1 + \frac{1}{2} \log \beta - \sqrt{\beta}.
$$

Let us consider $\beta > 1$ first. 
In this case, since $(-\log\, v)'' \ge 0$, considering $v(x)e^{-\frac{\varepsilon}{2} x^2}$ for small enough $\varepsilon >0$ and by an approximation, we may assume $(-\log\, v)'' \ge \varepsilon$.
Let $T$ be Brenier's map from $v$ to $\gamma$. Then we have 
$$
v(x) = \gamma(T(x)) T'(x). 
$$
With this and \eqref{e:W2Brenier} in mind, we see that 
\begin{align*}
&{\rm Ent}_{\gamma}\big(\frac{v}{\gamma}\big) - \frac{1}{2} W_2(v, \gamma)^2 \\
	=& \int_{\mathbb{R}} \big(\frac{1}{2}|x|^2 - \frac{1}{2}|T(x)|^2 + \log\, T'(x) \big) \, dv -\frac{1}{2}\int_{\mathbb{R}} |T(x)-x|^2\, dv\\
	=& -\int_{\mathbb{R}} |T(x)|^2\, dv + \int_ \mathbb{R} \log\, T'(x) \, dv + \int_{\mathbb{R}} xT(x)\, dv \\
	=& -\int_{\mathbb{R}} |x|^2\, d\gamma + \int_ \mathbb{R} \log T'(x)\, dv + \int_{\mathbb{R}} xT(x)T'(x) \gamma(T(x))\, dx. 
\end{align*}
If we notice that $T(x)T'(x) \gamma(T(x))=(-\gamma(T(x)))' $, then integration by parts shows that 
\begin{align*}
{\rm Ent}_{\gamma}\big(\frac{v}{\gamma}\big) - \frac{1}{2} W_2(v, \gamma)^2 
		= \int_ \mathbb{R} \big(-1 - \log \frac{1}{T'(x)} + \frac{1}{T'(x)}\big)\,  dv.
\end{align*}
In this step we used that
	\begin{equation}\label{e:LimIntByParts}
\lim_{|x| \to +\infty} |x|\gamma(T(x)) = 0
	\end{equation}
and this can be justified as follows. 
Let $S$ be Brenier's map pushing forward $\gamma$ to $v$. Then Caffarelli's contraction theorem yields $S' \le \frac{1}{\sqrt{\varepsilon}}$. 
On the other hand, since $T(S(x))=x$ for $x \in \mathbb{R}$, we obtain 
$T'(S(x))S'(x) =1$ for $x \in \mathbb{R}$.
Hence, we enjoy $T' \ge \sqrt{\varepsilon}$ on $\mathbb{R}$ from which we can deduce $|T(x)| \ge c_1 |x| - c_2$ for $|x|$ sufficiently large. Here $c_1, c_2 > 0$ are some constants depending on $\varepsilon$.
Since $|x| \gamma(T(x)) \le |x| \gamma(c_1 |x| - c_2)$ for $|x|$ sufficiently large, we can obtain \eqref{e:LimIntByParts}.

We now appeal to our assumption $(-\log\, v)'' \le  \frac1\beta$ which, by applying Caffarelli's contraction theorem, gives $T' \le \frac{1}{\sqrt{\beta}} \le1$. 
With this in mind, we notice that $x \mapsto -1 - \log x + x$ is monotone increasing in $x \in (1,\infty)$
from which we conclude 
\begin{align*}
{\rm Ent}_{\gamma}\big(\frac{v}{\gamma}\big) - \frac{1}{2} W_2(v, \gamma)^2 \ge  -1 - \frac{1}{2} \log \beta + \sqrt{\beta}.
\end{align*}

Let us next consider $0<\beta < 1$. In this case, our assumption is $(-\log\, v)''\ge \frac1\beta$ and thus if $T$ is Brenier's map from $\gamma$ to $v$, then $T'\le \sqrt{\beta}\le1$ follows from Caffarelli's contraction theorem.

Next note that by using
$$
\gamma(x) = v(T(x)) T'(x)
$$ 
we obtain
\begin{align*}
	{\rm Ent}_\gamma\big( \frac{v}{\gamma} \big)
	=& 
	\int_{\mathbb{R}} \log\, v(T(x))\, d\gamma - \int_{\mathbb{R}} \log\, \gamma(T(x))\, d\gamma \\
	=& 
	\int_{\mathbb{R}} \log\, \gamma\, d\gamma - \int_{\mathbb{R}} \log\, T'(x)\, d\gamma - \int_{\mathbb{R}} \log\, \gamma(T(x))\, d\gamma \\
	=& 
	\frac12\int_{\mathbb{R}} \big( |T(x)|^2 - |x|^2 \big)\, d\gamma - \int_{\mathbb{R}} \log\, T'(x)\, d\gamma\\
	=& 
	\frac12\int_{\mathbb{R}} \big| T(x) - x \big|^2\, d\gamma 
	- \int_{\mathbb{R}} |x|^2\, d\gamma
	+  \int_{\mathbb{R}}xT(x)\, d\gamma - \int_{\mathbb{R}} \log\, T'(x)\, d\gamma.
\end{align*}
For the third term, we perform integration by parts to see that 
\begin{align*}
	\int_{\mathbb{R}}xT(x)\, d\gamma
	=
	\int_{\mathbb{R}}T(x) (-\gamma)'\, dx
	=
	\int_{\mathbb{R}}T'(x) \, d\gamma.
\end{align*}
This step can be justified since $T'\le1$ means that $|T(x)|\gamma(x) \le C|x|\gamma(x)$ for some constant $C>0$ and $|x|$ sufficiently large.
By  \eqref{e:W2Brenier} we may now write 
\begin{align*}
	{\rm Ent}_\gamma\big( \frac{v}{\gamma} \big)
	=&
	\frac12 W_2(\gamma,v)^2 
	- \int_{\mathbb{R}} \big( \log\, T'(x) - T'(x) +1\big)\, d\gamma
\end{align*}
and 
$$
{\rm Ent}_\gamma\big( \frac{v}{\gamma} \big)
	\ge 
	\frac12 W_2(\gamma,v)^2 
	- \big( \log\, \sqrt{\beta} - \sqrt{\beta} +1\big)
$$
follows since $T'\le \sqrt{\beta}$.
\end{proof}

\begin{remark}
The technical condition $(-\log\, v)'' \ge 0$  imposed in the case $\beta > 1$ is only used to justify an integration by parts step in the proof and thus it is reasonable to hope that it could be removed. For instance, when $v$ is symmetric and has a gaussian concentration property, then one may ensure $|T(x)| \ge c_1|x| - c_2$ for some constants $c_1, c_2>0$ and for all large enough $|x|$ from which one can show \eqref{e:LimIntByParts}; see  \cite[Theorem 1.1]{ColFat}.
\end{remark}

\begin{proof}[Proof of Corollary \ref{Cor:Matrix}: Talagrand's inequality]
As in the proof of Corollary \ref{Cor:Matrix} for hypercontractivity,  we show only the assertion in (1) with $n=2$ and $B = {\rm diag}\, (\beta_1,\beta_2)$. 
To this end, we set 
\begin{align*}
v^1 (y_1) := \int_{\mathbb{R}}\, v(y_1, z)\, dz,\;\;\; y_1 \in \mathbb{R}
\end{align*}
and 
\begin{align*}
v(y_2 \,|\, y_1) := \frac{ v(y_1, y_2) }{ \int_{\mathbb{R}}\, v(y_1, z)\, dz},\;\;\; (y_1,y_2) \in \mathbb{R}^2.
\end{align*}
Then $\int_{\mathbb{R}} v^1\, dy_1= 1$ and $\int_{\mathbb{R}} v(y_2\,|\,y_1)\, dy_2 = 1$ for every $y_1 \in \mathbb{R}$ which follows from  $\int_{\mathbb{R}^2} v\, dx = 1$.
Furthermore, we can easily see that 
\begin{align*}
\int_{\mathbb{R}} |y_1|^2 v^1(y_1)\, dy_1 < \infty, \;\;\; \int_{\mathbb{R}} |y_2|^2 v(y_2\,|\,\widetilde{y}_1)\, dy_2 < \infty
\end{align*}
for every $\widetilde{y}_1 \in \mathbb{R}$ since we have  $\int_{\mathbb{R}^2} |x|^2v\, dx<\infty$.
Let $\pi^1 \in \Pi(\gamma, v^1)$ and $\pi(\cdot, \cdot\,|\,y_1) \in \Pi(\gamma, v(\cdot \,|\,y_1))$ be optimal couplings attaining equality in \eqref{e:DefWassDist}, respectively.
We consider a measure $d\pi((x_1,x_2), (y_1, y_2)) := d\pi(x_2, y_2\,|\,y_1)d\pi^1(x_1, y_1)$ on $\mathbb{R}^2 \times \mathbb{R}^2$, then we enjoy that $\pi \in \Pi( \gamma \otimes \gamma, v)$.
Hence by the definition in \eqref{e:DefWassDist}, we see that 
\begin{align*}
&\frac12 W_2(\gamma\otimes \gamma, v)^2 \\
	&\le 
	\frac12
	\int_{\mathbb{R}^2 \times \mathbb{R}^2}\, |x - y|^2\, d\pi(x, y) \\
	&=  
	\frac12 \int_{\mathbb{R} \times \mathbb{R}}\, |x_1 - y_1|^2\, d\pi^1(x_1, y_1) + \frac12 \int_{\mathbb{R}^2} \bigg( \int_{\mathbb{R}^2}\, |x_2 - y_2|^2\, d\pi(x_2, y_2\,|\,y_1) \bigg) d\pi^1(x_1, y_1)\\
	&= 
	\frac12 W_2(\gamma, v^1)^2 + \frac12 \int_{\mathbb{R}} W_2(\gamma, v(\cdot\,|\,y_1))^2\, dv^1(y_1).
\end{align*}
By the assumption for $\nabla^2 \log\, v$,  we have 
$$
\partial_{y_2}^2 \log\, v(y_2\,|\,y_1) \ge - \frac{1}{\beta_2}, \;\;\; \partial_{y_2}^2 \log\, v(y_2\,|\,y_1) \le 0
$$
for every $y_1 \in \mathbb{R}$. We also obtain 
$$
\partial_{y_1}^2 \log\, v^1 \ge - \frac{1}{\beta_1}, \;\;\; \partial_{y_1}^2 \log\, v^1 \le 0
$$
from $0 \ge \partial_{y_1}^2 \log\, v(y_1, y_2) \ge - \frac{1}{\beta_1}$ as in the argument in Corollary \ref{Cor:Matrix}. 
Hence we may apply our one dimensional result in Proposition \ref{t:RegTal} to see that 
$$
\frac12 W_2(\gamma,v^1)^2
		\le 
		{\rm Ent}_\gamma\big( \frac{v^1}{\gamma} \big) +1 + \frac{1}{2} \log \beta_1 - \sqrt{\beta_1}, 
$$
$$\frac12 W_2(\gamma,v(\cdot\,|\,y_1))^2
		\le 
		{\rm Ent}_\gamma\big( \frac{v(\cdot\,|\, y_1)}{\gamma} \big) +1 + \frac{1}{2} \log \beta_2 - \sqrt{\beta_2}
$$
for every $y_1 \in \R$ in the case of $\beta_1,\beta_2\ge1$. 
If either $\beta_1$ or $\beta_2$ are $<1$,  then we replace the corresponding inequality above by the standard Talagrand inequality \eqref{e:TalCl}.
Thus we can deduce that 
\begin{align*}
\frac12 W_2(\gamma, v)^2 
	\le 
	{\rm Ent}_\gamma\big( \frac{v^1}{\gamma} \big) 
	+ \int_{\mathbb{R}} {\rm Ent}_\gamma\big( \frac{v(\cdot\,|\, y_1)}{\gamma} \big) \, dv^1(y_1) 
	- \sum_{ \substack{ i=1, 2:\\ \beta_i \ge1 }} ( - 1 - \frac12 \log\, \beta_i + \sqrt{\beta_i} ).
\end{align*}
Finally, it follows from the additivity of entropy that we have 
\begin{equation}\label{e:AdditivityEnt}
{\rm Ent}_\gamma\big( \frac{v^1}{\gamma} \big) 
	+ \int_{\mathbb{R}} {\rm Ent}_\gamma\big( \frac{v(\cdot\,|\, y_1)}{\gamma} \big) \, dv^1(y_1) 
	= {\rm Ent}_\gamma\big( \frac{v}{\gamma} \big).
\end{equation}
For instance, see \cite[Theorem 22.8]{Vil2} for details of additivity of entropy in general settings. For completeness, we give a proof of \eqref{e:AdditivityEnt}.
Indeed, we enjoy 
\begin{align*}
{\rm Ent}_\gamma\big( \frac{v^1}{\gamma} \big) 
	=& \int_{\mathbb{\R}} \left( \int_{\mathbb{\R}}\, v(y_1, y_2)\, dy_2 \right) \log \left(  \int_{\mathbb{\R}}\, v(y_1, y_2)\, dy_2 \right)\, dy_1 \\
	&+ \int_{\mathbb{R}^2} v(y_1, y_2) \log\, \frac1{\gamma(y_1)}\,dy_1dy_2
\end{align*}
and 
\begin{align*}
\int_{\mathbb{R}} {\rm Ent}_\gamma\big( \frac{v(\cdot\,|\, y_1)}{\gamma} \big) \, dv^1(y_1)
	= & 
	- \int_{\mathbb{\R}} \left( \int_{\mathbb{\R}}\, v(y_1, y_2)\, dy_2 \right) \log \left(  \int_{\mathbb{\R}}\, v(y_1, y_2)\, dy_2 \right)\, dy_1 \\
	& + \int_{\mathbb{R}^2} v(y_1, y_2) \log\, \frac{v(y_1, y_2)}{\gamma(y_2)}\,dy_1dy_2.
\end{align*}
Summing up the two equalities above, we can deduce \eqref{e:AdditivityEnt}, and our proof is complete.
\end{proof}

\section{Applications}\label{S4}
In this section we capitalise on the wealth of connections that the hypercontractivity inequality and the LSI enjoy with other important inequalities. By applying our main results in Theorem \ref{t:MainHyper} (we shall need both the forward and reverse statements) and Theorem \ref{t:MainLSI}, we obtain certain improved versions of the hypercontractivity inequality for the Hamilton--Jacobi equation, the dual version of Talagrand's inequality, the gaussian Poincar\'e inequality, and Beckner's inequality. Also, at the end of the section we discuss improved versions of the dual form of hypercontractivity for the Ornstein--Uhlenbeck semigroup (viewed as a special case of the Brascamp--Lieb inequality). The underlying arguments we use in the proofs in this section are not novel and so our presentation will be somewhat terse.

\subsection{Hypercontractivity of the Hamilton--Jacobi equation and Talagrand's inequality}
Hypercontractivity for the Hamilton--Jacobi equation 
\begin{equation}\label{e:HJ}
\begin{cases}
\partial_\tau u + \frac12 |\nabla u|^2=0,\;\;\;  &(\tau,x) \in (0,\infty) \times\mathbb{R}^n, \\
 u(0,\cdot) = f,\;\;\; & x\in \mathbb{R}^n,
\end{cases}
\end{equation}
was investigated by Bobkov--Gentil--Ledoux in \cite{BGLJMPA} and a close link with \eqref{e:HCClassic} (or, alternatively, \eqref{e:LogSob}) was observed. 
To describe this, we introduce the so-called Hopf--Lax formula
\begin{equation}\label{e:Hopf}
Q_\tau f(x):= \inf_{ y\in \mathbb{R}^n } \big\{ f(y) + \frac1{2\tau}|x-y|^2  \big\},\;\;\; \tau >0,\; x\in \mathbb{R}^n
\end{equation}
for any measurable and real-valued function $f$ on $\mathbb{R}^n$. 
There are a large number of references concerning the theory of the Hamilton--Jacobi equation and we refer the reader to, for example, Evans \cite{Evans} for an excellent introduction to the theory.  In particular, it is known that if $f$ is Lipschitz continuous then Hopf--Lax formula $Q_\tau f$  is also Lipschitz continuous and solves the initial value problem \eqref{e:HJ} in an almost everywhere sense (see \cite[Theorem 6 on p. 128]{Evans}). 

Obtaining a suitable framework in which one has \textit{uniqueness} of ``weak" solutions to Hamilton--Jacobi equations is a highly non-trivial task (classical solutions do not exist in general). For the particular case \eqref{e:HJ} we consider here, \cite[Theorem 8 on p. 134]{Evans} is one possibility in which the notion of weak solution includes the condition that $\partial_\tau u + \frac12 |\nabla u|^2=0$ holds in an almost everywhere sense and $u$ satisfies a certain ``semi-convexity" condition. For more general Hamilton--Jacobi equations, a more suitable framework involves the notion of \emph{viscosity solutions} and goes back to work of Crandall--Lions \cite{CrandallLions} (see also Crandall--Evans--Lions \cite{CEL}). In this framework, if we assume that $f$ is Lipschitz continuous and bounded on $\mathbb{R}^n$, then $Q_\tau f$ is the unique viscosity solution to \eqref{e:HJ} (see, for example, \cite[Theorem 3 on p. 601]{Evans}). Results of this nature have also been obtained for wider classes of initial data and, for later use, we note that Alvarez, Barron and Ishii handled lower semi-continuous initial data which are bounded below by a function of linear growth.
\begin{theorem} [Theorem 5.2 in \cite{ABI}] \label{t:ABI}
Suppose $f:\mathbb{R}^n\to\mathbb{R}$ is lower semi-continuous and satisfies 
\begin{equation}\label{e:LowBound}
f(x) \ge - C (1+|x|),\;\;\; x\in \mathbb{R}^n
\end{equation}
for some $C>0$. Then $Q_\tau f$ is the unique lower semi-continuous viscosity solution to \eqref{e:HJ}. 
\end{theorem}

In the case of gaussian measure, hypercontractivity of the Hamilton--Jacobi equation takes the form\footnote{It is instructive to note that hypercontractivity for the Ornstein--Uhlenbeck semigroup \eqref{e:HCClassic} can be equivalently reformulated in exponential form as $\| e^{P_s f} \|_{e^{2s}} \leq \| e^f \|_1$ for $s \geq 0$ (see, for example, \cite{BE2}), or equivalently again, as $\| e^{P_s f} \|_{ae^{2s}} \leq \| e^f \|_a$ for $a, s \geq 0$. Here, all norms are taken with respect to the gaussian measure $\gamma$.}
\begin{equation}\label{e:HyperHJ}
\big\| e^{ Q_\tau f } \big\|_{ L^{a+\tau }(\gamma) } \le \|e^f\|_{L^a(\gamma)}
\end{equation}
for all $f \in L^\infty(\mathbb{R}^n)$, $\tau\ge0$, and $a\in\mathbb{R}$. Offering two approaches, Bobkov, Gentil and Ledoux obtained \eqref{e:HyperHJ} in \cite{BGLJMPA}; one is based on ideas of Gross \cite{Gross} and a second approach is based on the so-called \textit{vanishing viscosity method}\footnote{The latter approach was used in \cite{CrandallLions} and gave rise to the terminology \emph{viscosity solutions}.}.

Our aim here is to improve the constant in \eqref{e:HyperHJ} in the spirit of Theorem \ref{t:MainHyper}, and it turns out that we are able to handle initial data which are uniformly subharmonic in the sense that
\[
\Delta f\ge n (1 -\frac1\beta)
\]
for some $\beta \geq 1$. Because of a technical reason,  we also impose \eqref{e:LowBound}. 
Whilst there are examples of uniformly subharmonic functions which do not have the property \eqref{e:LowBound}, if we impose the stronger assumption $\nabla^2 f \ge (1-\frac1\beta) {\rm id} $ for some $\beta>1$, then $f$ has a global minimum and hence \eqref{e:LowBound} is satisfied.  

As in \cite{BGL}, we adopt the vanishing viscosity method. The point is that, for each $\varepsilon > 0$, the solution of 
$$
\partial_\tau u^\varepsilon_\tau - \varepsilon \mathcal{L} u^\varepsilon_\tau + \frac12 | \nabla u^\varepsilon_\tau |^2 = 0,\;\;\; u^\varepsilon(0,\cdot) = f,
$$
which is a certain regularisation of \eqref{e:HJ}, can be expressed as\footnote{To see this, note that $w = \exp(-\frac{u^\varepsilon}{2\varepsilon})$ solves the gaussian heat equation $\partial_\tau w = \varepsilon \mathcal{L} w, w(0) = \exp(-\frac{f}{2\varepsilon})$.}
\begin{equation}\label{e:VaniVis}
u_\tau^\varepsilon:= -2\varepsilon \log\, P_{\varepsilon \tau}\big[ e^{ -\frac{f}{2\varepsilon} } \big], 
\end{equation}
and thus one can derive information from  hypercontractivity inequalities for the Ornstein--Uhlenbeck semigroup. Moreover, we expect to have
$$
\lim_{\varepsilon\downarrow0} u^\varepsilon_\tau = Q_\tau f
$$
in some appropriate sense. Whilst there are important papers which investigate this type of problem in much more generality (see, for example, \cite{BBL} and references therein), it shall suffice for us to establish
\begin{equation} \label{e:vanishingviscosity}
\liminf_{\varepsilon\downarrow0} u^\varepsilon_\tau(x) \geq Q_\tau f(x)
\end{equation}
for each $x \in \mathbb{R}^n$, under the assumption that $f$ is lower semi-continuous and bounded from below. To see this, note that
\begin{align*}
P_{\varepsilon \tau}\big[ e^{ -\frac{f}{2\varepsilon} } \big](x)^{2\varepsilon} & \leq C_\varepsilon \bigg( \int_{\mathbb{R}^n} e^{g^\varepsilon_\tau(x,y)} dy \bigg)^{2\varepsilon} \| e^{g^\varepsilon_\tau(x,y)} \|_{L^\infty(dy)}^{1-2\varepsilon} \\
& \leq C_\varepsilon \bigg( \int_{\mathbb{R}^n} e^{-f(y)} dy \bigg)^{2\varepsilon} \| e^{g^\varepsilon_\tau(x,y)} \|_{L^\infty(dy)}^{1-2\varepsilon}
\end{align*}
where $C_\varepsilon = (2\pi(1 - e^{-2\varepsilon \tau}))^{-\varepsilon n}$, and
\begin{align*}
g^\varepsilon_\tau(x,y)  := -f(y) - \frac{2\varepsilon \tau}{2\tau(1 - e^{-2\varepsilon \tau})} |y - e^{-\varepsilon \tau}x|^2.
\end{align*}
We have
\[
\sup_{y \in \mathbb{R}^n} g^\varepsilon_\tau(x,y)  \leq \sup_{y \in \mathbb{R}^n} \{ -f(y) - \frac{1}{2\tau} |y - e^{-\varepsilon \tau}x|^2 \} = - Q_\tau f(e^{-\varepsilon \tau}x).
\]
Since we are assuming that $f$ is lower semi-continuous and bounded from below, by Theorem \ref{t:ABI} we know that $Q_\tau f$ is lower semi-continuous and therefore
\[
\liminf_{\varepsilon \downarrow 0} Q_\tau f(e^{-\varepsilon \tau}x) \geq Q_\tau f(x).
\]
From the above we obtain
\[
\limsup_{\varepsilon \downarrow 0} P_{\varepsilon \tau}\big[ e^{ -\frac{f}{2\varepsilon} } \big](x)^{2\varepsilon} \leq e^{-Q_\tau f(x)}
\]
and hence \eqref{e:vanishingviscosity}.

Armed with the above observations, we obtain the following. 
\begin{theorem}\label{t:RegHJ}
Suppose $a,\tau>0$ and $\beta > 1$ satisfy 
\begin{equation}\label{e:Cond8Nov1}
\beta(1-\frac1a) < 1. 
\end{equation} 
Let $\beta(a) > 1$ be given by
\[
\frac1{\beta(a)}:= 1 - a( 1-\frac1\beta )
\]
For all $ f :\mathbb{R}^n\to\mathbb{R} $ such that \eqref{e:LowBound} and 
\begin{equation}\label{e:Cond8Nov2}
\Delta f\ge (1 -\frac1\beta) {\rm id},\;\;\;  \int_{\mathbb{R}^n} e^{2af} \frac{\gamma}{\gamma_{\beta(a)}}\, d\gamma <\infty, 
\end{equation}
we have that 
\begin{equation}\label{e:HCHJ}
\big\| e^{Q_\tau f}  \big\|_{L^{a+\tau}(\gamma)} \le \big\| e^{Q_\tau \big[ \frac1a \log\, \frac{\gamma_{\beta(a)}}{\gamma} \big]}  \big\|_{L^{a+\tau}(\gamma)} \| e^{f} \|_{L^a(\gamma)}. 
\end{equation}
In particular, \eqref{e:HCHJ} holds for all $f = \frac1a \log\, \frac{v}{\gamma}$ where $v\in {\rm FP}(\beta(a))$. 
\end{theorem}

\begin{proof}
Let $a,\tau>0$ be fixed and set $b_\varepsilon>0$ so that 
\begin{equation}\label{e:Expo12June}
\frac{1+2\varepsilon b_\varepsilon}{1+ 2\varepsilon a} = e^{2\varepsilon \tau},\;\;\; \varepsilon>0.
\end{equation}
Note that we have $\lim_{\varepsilon \downarrow 0} b_\varepsilon = a+\tau$ under this choice.  
Setting
$$
u^\varepsilon_\tau := -2\varepsilon \log\, P_{\varepsilon\tau} \big[ e^{ - \frac{f}{2\varepsilon}} \big]
$$
and
$$
q = -2\varepsilon b_\varepsilon,\;\;\; p = -2\varepsilon a < 0, \;\;\; s= \varepsilon \tau, 
$$
we have $\frac{q-1}{p-1} = e^{2s}$ thanks to \eqref{e:Expo12June}, and also
\begin{align*}
\big\| e^{u^\varepsilon_\tau} \big\|_{L^{b_\varepsilon}(\gamma)} 
&=\bigg( \int_{\mathbb{R}^n} P_{s} \big[ e^{ - \frac{f}{2\varepsilon}} \big]^{q }  \, d\gamma\bigg)^{-\frac{2\varepsilon}{q}}
=
\bigg( \int_{\mathbb{R}^n} P_{s} \big[ (e^{  af})^\frac1p \big]^{q }  \, d\gamma\bigg)^{-\frac{2\varepsilon}{q}}. 
\end{align*}
If we regard $e^{af} = \frac{v}{\gamma}$, then the assumption \eqref{e:Cond8Nov2} can be read as 
$$
\Delta \log\, v = \Delta af - n \ge  \big( a( 1 -\frac1{\beta} ) - 1 \big) = -\frac{n}{\beta(a)},\;\;\; v \in L^2(\gamma_{\beta(a)}). 
$$
Since $\beta(a) > 1$,  we may apply Theorem \ref{t:MainHyper}(2) to see that 
\begin{align*}
\big\| e^{u^\varepsilon_\tau} \big\|_{L^{b_\varepsilon}(\gamma)} 
& \le 
\big\| P_s \big[ \big( \frac{\gamma_{\beta(a)}}{\gamma} \big)^\frac1p \big] \big\|_{L^q(\gamma)}^{-2\varepsilon} \bigg(\int e^{af}\, d\gamma \bigg)^{\frac{-2\varepsilon}{p}} \\
& =
\big\| P_{\varepsilon \tau} \big[ \big( \frac{\gamma_{\beta(a)}}{\gamma} \big)^\frac1{-2\varepsilon a} \big]^{-2\varepsilon} \big\|_{L^{b_\varepsilon}(\gamma)} \|e^{f}\|_{L^a(\gamma)}. 
\end{align*}
Formally, taking $\varepsilon \to 0$ we deduce the desired estimate \eqref{e:HCHJ}. To make this rigorous, we may invoke Fatou's lemma and \eqref{e:vanishingviscosity}.
\end{proof}

  We are interested in the dual form of \eqref{e:TalCl} here. As observed in the work of Bobkov--G\"{o}tze \cite{BoGo}, \eqref{e:TalCl} is indeed equivalent to 
\begin{equation}\label{e:DualTalCl}
\int_{\mathbb{R}^n} e^{ Q_1 f }\, d\gamma \le e^{ \int_{\mathbb{R}^n} f\, d\gamma }
\end{equation}
for all bounded and continuous $f :\mathbb{R}^n\to\mathbb{R}$, and this proceeds by taking $\tau =1$ and $a\to0$ in \eqref{e:HyperHJ}.  Given our improvement of \eqref{e:HyperHJ} in Theorem \ref{t:RegHJ}, it is not a surprise that we obtain an improvement of Talagrand's inequality too. 
\begin{corollary}\label{Cor:Tal}
Let $\tau>0$ and $\beta > 1$. Then 
$$
\big\| e^{Q_\tau f} \big\|_{L^\tau(\gamma)} \le T(\tau,\beta) e^{\int_{\mathbb{R}^n} f\, d\gamma}
$$
whenever $f:\mathbb{R}^n\to \mathbb{R}$ satisfies \eqref{e:LowBound} and  \eqref{e:Cond8Nov2} for all sufficiently small $0<a\ll1$, and
where 
$$
T(\tau,\beta):= \bigg( e^{-1} ( 1 +\tau(1-\frac1\beta) )^{ \frac{1}{ \tau(1-\frac1\beta) } } \bigg)^{ \frac{n}{2}(1-\frac1{\beta} )} \in (0,1).
$$
\end{corollary}
\begin{proof}
By taking a limit $a\to0$ in \eqref{e:HCHJ}, it suffices to show that 
\begin{equation}\label{e:Goal9Nov}
\lim_{a\to0} \big\| e^{ Q_\tau \big[ \frac1{a} \log\, \frac{\gamma_{\beta(a)}}{\gamma} \big] } \big\|_{L^{a+\tau}(\gamma)} = T(\tau,\beta). 
\end{equation}
To this end, we first notice from direct computation that 
$$
Q_\tau\big[ \frac1{a} \log\, \frac{\gamma_{\alpha}}{\gamma} \big](x) = \frac1{2\tau} \frac{\delta_{a,\alpha}}{\delta_{a,\alpha}+1} |x|^2 - \frac{n}{2a}\log\, \alpha,\;\;\; \delta_{a,\alpha}:= \frac{\tau}{a} (1-\frac1{\alpha})
$$
holds true in general for $\tau,a>0$ and $\alpha\ge1$. 
Hence
$$
\big\| e^{ Q_\tau \big[ \frac1{a} \log\, \frac{\gamma_{\alpha}}{\gamma} \big] } \big\|_{L^{a+\tau}(\gamma)} = \alpha^{-\frac{n}{2a}} \big( ( \frac{a}{\tau} + 1 )\frac{1}{\delta_{a,\alpha}+1} - \frac{a}{\tau} \big)^{ -\frac{n}{2(a+\tau)} }
$$
as long as $( \frac{a}{\tau} + 1 )\frac{1}{\delta_{a,\alpha}+1} - \frac{a}{\tau}>0$.  By taking $\alpha = \beta(a)\ge1$, we see that $\delta_{a,\beta(a)} = \tau(1-\frac1\beta)$ and hence, recalling the definition of $\beta(a)$ in \eqref{e:HCHJ}, 
\begin{align*}
\lim_{a\to0} \big\| e^{ Q_\tau \big[ \frac1{a} \log\, \frac{\gamma_{\alpha}}{\gamma} \big] } \big\|_{L^{a+\tau}(\gamma)}
=&
\lim_{a\to0} \big( 1 - a ( 1 - \frac1\beta) \big)^{ \frac{n}{2a} } \big( ( \frac{a}{\tau} + 1 )\frac{1}{ \tau(1-\frac1\beta) +1} - \frac{a}{\tau} \big)^{ -\frac{n}{2(a+\tau)} }\\
=& 
e^{ -\frac{n}{2} (1-\frac1\beta) } \big( 1 + \tau(1-\frac1\beta) \big)^{ \frac{n}{2\tau} } 
=
T(\tau, \beta)
\end{align*}
as claimed.
\end{proof}

\subsection{Poincar\'{e} and Beckner inequalities}
Next we consider improvements of the gaussian Poincar\'{e} inequality
\begin{equation}\label{e:Poi}
\int_{\mathbb{R}^n} f^2\, d\gamma - \bigg( \int_{\mathbb{R}^n} |f|\, d\gamma \bigg)^2 \le \int_{\mathbb{R}^n} |\nabla f|^2\, d\gamma
\end{equation}
for real-valued functions $f\in L^2(\gamma)$ whose gradient belongs to $L^2(\gamma)$.  

Define the constant
\[
D_n(\beta) = \frac{n}{2}\bigg(\log \beta - 1 + \frac{1}{\beta}\bigg).
\]
\begin{corollary}\label{Cor:RegPoi}
Let $\beta>0$ and let $\gamma f^2 \in L^2(\gamma_\beta^{-1})$. When $\beta > 1$, assume that $\gamma f^2$ is $\beta$-semi-log-subharmonic, and when $\beta < 1$, assume that $\gamma f^2$ is $\beta$-semi-log-concave. Then 
\begin{equation}\label{e:RegPoi}
\frac{1}{2}(1 + D_n(\beta))\int_{\mathbb{R}^n} f^2\, d\gamma - \frac12 \bigg( \int_{\mathbb{R}^n} |f|\, d\gamma \bigg)^2 \le \int_{\mathbb{R}^n} |\nabla f|^2\, d\gamma.
\end{equation}
\end{corollary}
\begin{remark}
Note that this result offers a gain over  the classical inequality \eqref{e:Poi} for sufficiently small or sufficiently large $\beta$. More precisely, this is possible when $D_n(\beta) > 1$ or equivalently,
in terms of Lambert $W$ functions, when
$$
\beta \notin \bigg[- \frac{1}{W_{-1}(e^{-(1 + 2/n)})}, - \frac{1}{W_0(e^{-(1 + 2/n)})} \bigg].
$$
\end{remark}
\begin{proof}[Proof of Corollary \ref{Cor:RegPoi}]
We will prove a more general inequality. For all $p \in[1,2)$ we show
\begin{equation}\label{e:PoiGoal}
\|f\|_{L^2(\gamma)}^2\log\, \frac{\|f\|_{L^2(\gamma)}^2}{\| f \|_{L^p(\gamma)}^2}
\le 
\frac{2(2-p)}{p}  \int_{\mathbb{R}^n} | \nabla f|^2\, d\gamma - \frac{2-p}{p} D_n(\beta) \|f\|_{L^2(\gamma)}^2,
\end{equation}
from which we can see that 
$$
\|f\|^2_{L^2(\gamma)}- \| f \|_{L^p(\gamma)}^2
\le 
\frac{2(2-p)}{p}  \int_{\mathbb{R}^n} | \nabla f|^2\, d\gamma - \frac{2-p}{p} D_n(\beta) \|f\|_{L^2(\gamma)}^2,
$$
since $ x\log\, x \ge x-1 $.  By choosing $p=1$, we clearly derive \eqref{e:RegPoi}. 

To prove \eqref{e:PoiGoal}, we appeal to the argument of \cite[Proposition 5.1.8]{BGL}. 
Let us introduce $\phi(r):= \log\, ( \| f \|_{L^{\frac1r}(\gamma)} )$, $r\in (0,1]$. This $\phi$ is convex thanks to H\"{o}lder's inequality and hence, for $p \in [1,2)$, we obtain
\begin{align*}
\frac12\log\, \frac{\| f \|_{L^p(\gamma)}^2}{\|f\|_{L^2(\gamma)}^2} =  \phi(\tfrac1p) - \phi(\tfrac12) 
& \ge (\tfrac1p - \tfrac12) \phi'(\tfrac12) \\
& = \frac{\tfrac12 - \tfrac1p }{\|f\|_{L^2(\gamma)}^2} {\rm Ent}_{\gamma}( f^2 ) \\
& \geq \frac{\tfrac12 - \tfrac1p}{\|f\|_{L^2(\gamma)}^2} \bigg(\frac12 {\rm I}_{\gamma}(f^2) - D_n(\beta) \|f\|_{L^2(\gamma)}^2 \bigg)
\end{align*}
through use of Theorem \ref{t:MainLSI}. Since ${\rm I}_{\gamma}(f^2) = 4 \int | \nabla f |^2\, d\gamma$, we quickly obtain \eqref{e:PoiGoal}.
\end{proof}

Beckner's inequality \cite{Beckner} for each $p \in [1,2)$
\begin{equation}\label{e:GPI}
\frac1{2-p} \bigg[ \int_{\mathbb{R}^n} f^2\, d\gamma - \bigg( \int_{\mathbb{R}^n} |f|^p\, d\gamma\bigg)^{\frac2p} \bigg] \le \int_{\mathbb{R}^n} | \nabla f |^2\, d\gamma
\end{equation}
is a generalisation of the Poincar\'{e} inequality (take $p=1$) and recovers the LSI in \eqref{e:LogSob} (take $p \to 2$). Beckner obtained \eqref{e:GPI} by proving
\begin{equation} \label{e:BecknerPre}
\int_{\mathbb{R}^n} f^2\, d\gamma - \int_{\mathbb{R}^n} |P_sf|^2\, d\gamma
\le  (1 - e^{-2s}) \int_{\mathbb{R}^n} | \nabla f |^2\, d\gamma
\end{equation}
and combining this with hypercontractivity of the Ornstein--Uhlenbeck semigroup \eqref{e:HCClassic}. In the same manner, replacing \eqref{e:HCClassic} with our result in Theorem \ref{t:MainHyper}, we obtain the following.
\begin{corollary}\label{Cor:RegBeckner}
Let $\beta>0$, $1 < p < 2$ and let $\gamma f^p \in L^2(\gamma_\beta^{-1})$. When $\beta > 1$, assume that $\gamma f^p$ is $\beta$-semi-log-subharmonic, and when $\beta < 1$, assume that $\gamma f^p$ is $\beta$-semi-log-concave. Then 
$$
\frac1{2-p} \bigg[ \int_{\mathbb{R}^n} f^2 \, d\gamma - B(p,\beta) \bigg( \int_{\mathbb{R}^n} f^p \, d\gamma \bigg)^{\frac2p} \bigg] \le  \int_{\mathbb{R}^n} | \nabla f |^2\, d\gamma
$$
Here
$$
B(p,\beta):= \int_{\mathbb{R}^n} |P_{s} \big[ \big( \frac{\gamma_\beta}{\gamma_1}\big)^\frac1p \big]|^2\, d\gamma 
=
\beta^{\frac{n}{p'}} ( 1 + (\beta -1 ) \frac{2}{p'} )^{- \frac{n}{2}},
$$
where $s := -\frac12 \log\, (p-1)$.
\end{corollary}

\begin{remark}
We have $B(p,\beta) < 1$ whenever $\beta > 1$ and $p \in (1,2)$ and so we see an improvement over \eqref{e:GPI} in such cases. 
Also, it is easy to verify that $\partial_p B(2,\beta) = \frac{1}{2}D_n(\beta)$ and thus, as one would expect, the limiting case $p \to 2$ in Corollary \ref{Cor:RegBeckner} recovers our regularised LSI in Theorem \ref{t:MainLSI} (and thus gives a slightly different route compared to that in Section \ref{S3.3}).
In the limiting case $p \to 1$, one sees that $B(p,\beta) \to 1$ and hence Corollary \ref{Cor:RegBeckner} does not yield any improvement for the Poincar\'{e} inequality. 
\end{remark}

\subsection{Regularised hypercontractivity in dual form}
Via duality, the hypercontractivity inequality \eqref{e:HCClassic} on $\mathbb{R}$ (for simplicity) can be equivalently viewed as a special case of Brascamp--Lieb inequality 
\begin{equation}\label{e:BL}
\int_{\mathbb{R}^2} e^{ -\pi x\cdot Q x  } \prod_{j=1,2} f_j(L_jx)^{c_j}\, dx\le \mathrm{H}(c_1,c_2) \prod_{j=1,2} \bigg( \int_{\mathbb{R}} f_j \bigg)^{c_j}
\end{equation}
for nonnegative $f_1,f_2 \in L^1( \mathbb{R} )$. Here, the positive semi-definite transformation $Q$ and exponents $c_1,c_2\in (0,1)$ are given by 
\begin{equation}\label{e:Relation}
Q:= \frac1{ 2\pi (1-e^{-2s}) } 
\begin{pmatrix}
1 - ( 1-e^{-2s} ) \frac1p & -e^{-s} \\
-e^{-s} & 1 - ( 1-e^{-2s} )\frac1{q'}
\end{pmatrix},\;\;\;
c_1:= \frac1p,\;\;\; c_2 := \frac1{q'}.
\end{equation}
The condition $e^{2s} = \frac{q-1}{p-1}$ is equivalent to
\begin{equation} \label{e:HCscondition}
c_1+c_2 = 1 + (1-e^{-2s})c_1c_2,
\end{equation}
and the constant $\mathrm{H}(c_1,c_2)$ is given by
\[
\mathrm{H}(c_1,c_2)  := (2\pi)^{ 1 - \frac12 (c_1 + c_2 ) } \sqrt{ 1 - e^{-2s} }.
\]
We note that in order to obtain \eqref{e:BL} from \eqref{e:HCClassic}, it suffices to choose
\begin{equation}\label{e:BLLink}
f = \big( \frac{ f_1 }{\gamma} \big)^\frac1p
\end{equation}
and apply H\"{o}lder's inequality. Similarly, the reverse hypercontractivity inequality \eqref{e:RevHC} is equivalent to the inverse Brascamp--Lieb inequality
\begin{equation}\label{e:RevBL}
\int_{\mathbb{R}^2} e^{ -\pi x\cdot Q x  } \prod_{j=1,2} f_j(L_jx)^{c_j}\, dx\ge\mathrm{H}(c_1,c_2) \prod_{j=1,2} \bigg( \int_{\mathbb{R}} f_j \bigg)^{c_j}
\end{equation}
 for all positive $f_1,f_2 \in L^1( \mathbb{R} )$. Here, we reuse the same notation as above, but we note that positive semi-definiteness of $Q$ and the positivity of $c_1,c_2$ do not necessarily persist in this case; we refer to \cite{BWAnn,BW} for further discussion along these lines.
 
The equivalences alluded to above rest on the fact that we are referring to general input functions. When restricting to special classes of input functions there is no reason to expect such equivalences to continue to be valid and typically one might expect the dual version to be weaker. For instance, it would be reasonable to consider the dual statement \eqref{e:BL} where both $f_1$ and $f_2$ are regularised in a similar manner. Indeed, this would fit into the framework of regularised Brascamp--Lieb inequalities established by Bennett--Carbery--Christ--Tao \cite{BCCT} using heat flow (see also \cite{V} for an alternative approach based on mass transportation, and also \cite{BN} for regularised inverse Brascamp--Lieb inequalities). Obtaining \eqref{e:BL} for such inputs from \eqref{e:RegHC} simply relies on H\"older's inequality (and, in fact, does not make use of the regularity imposed on one of the inputs). On the other hand, based on the form of the duality function in $L^p$ duality, there seems to be a much less clear path towards obtaining the reverse implication. The following result improves  \eqref{e:BL} and \eqref{e:RevBL} and is obtained by making use of our regularised hypercontractivity inequalities in \eqref{e:RegHC} and \eqref{e:RegRevHC}  (via H\"older and \eqref{e:BLLink}). The improved constant which arises in this way takes the form
\[
\mathcal{H}(c_1,c_2)  := \mathrm{H}(c_1,c_2) \big\| P_s \big[ \big( \frac{\gamma_\beta}{\gamma} \big)^{c_1} \big] \big\|_{L^{ (\frac1{ c_2 })' }(\gamma)}
\]
\begin{corollary}\label{Cor:BL}
Let $\beta>0$,  $s>0$ and $c_1,c_2 \in \mathbb{R}\setminus\{0\}$ satisfy \eqref{e:HCscondition}.

\begin{enumerate}
\item 
In the case $c_1,c_2 \in (0,1)$, we suppose $\beta > 1$. Then 
\begin{align*}
\int_{\mathbb{R}^2} e^{ -\pi x\cdot Q x  } \prod_{j=1,2} f_j(L_jx)^{c_j}\, dx
\le \mathcal{H}(c_1,c_2)\prod_{j=1,2} \bigg( \int_{\mathbb{R}} f_j \bigg)^{c_j}
\end{align*}
for any nonnegative $f_2 \in L^1(\mathbb{R})$ and $f_1 \in L^1(\mathbb{R}) $ satisfying $(\log\, f_1)'' \ge - \frac1\beta$.
\item 
In the case $c_1 c_2 <0$,  we suppose $\beta > 1$. Then 
\begin{align*}
\int_{\mathbb{R}^2} e^{ -\pi x\cdot Q x  } \prod_{j=1,2} f_j(L_jx)^{c_j}\, dx
\ge 
\mathcal{H}(c_1,c_2) \prod_{j=1,2} \bigg( \int_{\mathbb{R}} f_j  \bigg)^{c_j}
\end{align*}
for any positive $f_2 \in L^1(\mathbb{R})$ and $f_1 \in L^1(\mathbb{R}) $ satisfying $(\log\, f_1)'' \ge - \frac1\beta$.
\item 
In the case $c_1,c_2 >1$, we suppose $\beta < 1$. Then 
\begin{align*}
\int_{\mathbb{R}^2} e^{ -\pi x\cdot Q x  } \prod_{j=1,2} f_j(L_jx)^{c_j}\, dx
\ge 
\mathcal{H}(c_1,c_2)\prod_{j=1,2} \bigg( \int_{\mathbb{R}} f_j  \bigg)^{c_j}
\end{align*}
for any positive $f_2 \in L^1(\mathbb{R})$ and $f_1 \in L^1(\mathbb{R}) $ satisfying $(\log\, f_1)'' \le - \frac1\beta$.
\end{enumerate}
\end{corollary}
This result indicates that it would be reasonable to try and extend the regularised Brascamp--Lieb inequalities in \cite{BCCT} and \cite{BN}, in which the input functions are given by solutions to heat equations at time $t=1$, and allow inputs which are semi-log-convex/semi-log-concave in the appropriate sense.

\section{Further results}\label{S5}
\subsection{Hypercontractivity on general measure} \label{subsection:hypergeneral}
Given our regularised hypercontractivity inequality on gaussian space, it is of course natural to pursue the case of more general measure spaces. Whilst significant parts of our proof of Theorem \ref{t:MainHyper} readily generalise, certain parts make use of the specific nature of the gaussian measure, { namely the commutation property $ \nabla P_s = e^{-s} P_s \nabla $}. In this subsection we isolate exactly which parts of the argument offer up difficulties in extending to more general measures and we suggest possible ways to overcome these.

First, we provide the following set-up. Take a smooth, non-negative,  and convex potential $V$ on $\mathbb{R}^n$ and fix the probability measure 
\[
d\mathfrak{m}(x) = Z^{-1}e^{ -V(x) }dx, \quad Z:= \int_{\mathbb{R}^n} e^{-V(x)}\, dx, 
\]
which is invariant and symmetric with respect to the operator 
\[
\mathcal{L} = \mathcal{L}_V:= \Delta - \nabla V \cdot \nabla.
\] 
Then we denote a diffusion semigroup corresponding to $\mathcal{L}$ by $P_s:L^2(\mathfrak{m}) \to L^2(\mathfrak{m})$, $s\ge0$, which satisfies 
$$
\partial_s P_s f = \mathcal{L} P_sf,\;\;\; \lim_{s\downarrow0} P_sf =f 
$$
for all $f \in L^2(\mathfrak{m})$.  We remark that $P_s$ can be extended to  a bounded operator on $L^p(\mathfrak{m})$ for all $1\le p\le\infty$, and moreover it is a contraction. We refer the reader to the book \cite{BGL} for a comprehensive treatment.  
We will frequently use the notation 
$$
d\mathfrak{m}_\beta(x) = Z_\beta^{-1} e^{ - \frac{1}{\beta} V(x)}dx, \quad Z_\beta:= \int_{\mathbb{R}^n} e^{ - \frac{1}{\beta}V(x) }\, dx, \quad \mathcal{L}_{\beta} := \beta\Delta - \nabla V \cdot \nabla 
$$
for $\beta>0$.  In the gaussian case $V(x) = \frac12|x|^2$, we have $\mathfrak{m}_\beta = \gamma_\beta$ and the notation $\mathcal{L}_{\beta}$ is consistent with notation we introduced in Section 1.  

In this general setting, the hypercontractivity inequality \eqref{e:HCClassic} can be generalised as follows.  Suppose $\nabla^2 V \ge K {\rm id}$ for some $K>0$ and let $1 < p< q <\infty$ and $s>0$ satisfy $\frac{q-1}{p-1} = e^{ 2Ks }$. Then  
\begin{equation}\label{e:HCPotential}
\big\| P_s\big[ f^\frac1p \big] \big\|_{L^q(\mathfrak{m})} \le \bigg(\int_{\mathbb{R}^n} f\, d\mathfrak{m} \bigg)^{\frac1p}
\end{equation}
for all non-negative $f \in L^1(\mathfrak{m})$.
It is reasonable to expect to improve the constant in \eqref{e:HCPotential} via the Fokker--Planck equation 
\begin{equation}\label{e:FPPotential}
\partial_t v_t = \mathcal{L}_\beta^* v_t := \beta \Delta v_t  + \nabla\cdot (v_t \nabla V )
\end{equation}
in the spirit of \eqref{e:RegHC}.  
Here $\mathcal{L}_{\beta}^*$ means the dual of $\mathcal{L}_\beta$ with respect to the $L^2(dx)$ inner product. 
Note that the assumption $\nabla^2 V \ge K {\rm id}$ is indeed known to be related to the commutativity of $P_s$ and $\nabla$.  In the case of $V(x) = \frac1{2} |x|^2$, it is straightforward from the explicit form of $P_sf$ in \eqref{e:PsFormula} to see that 
\begin{equation}\label{e:GradEstGauss}
\nabla ( P_sf ) = e^{-s} P_s\big[ \nabla f\big]:= e^{-s} ( P_s\big[ \partial_1 f \big],\ldots,  P_s\big[ \partial_n f \big] ). 
\end{equation}
For general potential $V$, the convexity condition on $V$ is in fact equivalent to a weak form of \eqref{e:GradEstGauss}, the so-called \textit{gradient estimate} 
\begin{equation}\label{e:GradEstPotensial}
|\nabla ( P_sf ) | \le e^{-Ks} P_s\big[ |\nabla f|\big]
\end{equation}
for all smooth and compactly supported functions $f$.
We refer to \cite{BE} for details on this equivalence.
In what follows, we will impose a slightly stronger assumption than the gradient estimate \eqref{e:GradEstPotensial}, namely
\begin{equation}\label{e:GradEstStrong}
|\nabla ( P_sf ) | \le e^{-Ks} |P_s\big[ \nabla f\big]|
\end{equation}
for all smooth $f \in L^1(\mathfrak{m})$.

Let us consider the case $\beta\in(0,1)$. 
We generalise the notion of semi-log-concavity by assuming that the solution $v_t$ to \eqref{e:FPPotential} satisfies 
\begin{equation}\label{e:LogConvex}
\nabla^2  \log\, v_t \le -\frac1{\beta} \nabla^2 V
\end{equation}
for all $t \geq 0$. Under this set-up, let us discuss to what extent our proof of Theorem \ref{t:MainHyper} can be generalised.

Firstly, for technical reasons, we will restrict our attention to the case $q=2$ and thus we consider $1<p<q=2$ satisfying $\frac{q-1}{p-1} = e^{2Ks}$ for some $K>0$. Along the lines of our proof of Theorem \ref{t:MainHyper}, our goal is to show that the functional
\[
Q(t):= \int_{\mathbb{R}^n} P_s [ h_t ]^q \, d\mathfrak{m}
\]
is non-decreasing for $t > 0$, where 
$
h_t:= ( \frac{v_t}{\mathfrak{m}})^\frac1p.
$

Observe that in the proof of Theorem \ref{t:MainHyper} we made use of identities \eqref{e:AssumpG=} and \eqref{e:AssumpF=}, both of which follow from \eqref{e:GradEstGauss} and are unique to the case $V(x) = \frac12 |x|^2$. 
With this in mind,  for general potential $V$,  here we \emph{impose} an inequality version of \eqref{e:AssumpG=} and \eqref{e:AssumpF=} as follows. 
We say that the solution $v_t$ to \eqref{e:FPPotential} satisfies \textit{gradient type estimates} if \eqref{e:GradEstStrong} holds, the inequality
\begin{align}
\int_{\mathbb{R}^n}  \nabla P_{2s}  f \cdot \mathbf{g} \, d\mathfrak{m}
\le 
e^{-2Ks} \int_{\mathbb{R}^n}  P_{2s}\big[ \nabla f \big] \cdot \mathbf{g} \, d\mathfrak{m} \label{e:AssumpG<} 
\end{align}
is satisfied for $( f,\mathbf{g} ) = ( h_t, h_t \nabla V ), ( h_t, h_t \nabla \big( \log\, \frac{v_t}{\mathfrak{m}}\big) )$, and the inequality
\begin{align}
\int_{\mathbb{R}^n}   \big( {\rm div}\,  P_{2s} \mathbf{f} \big){g} \, d\mathfrak{m}
\ge 
e^{-2Ks} \int_{\mathbb{R}^n}    P_{2s}\big[ {\rm div}\, \mathbf{f} \big] {g} \, d\mathfrak{m} \label{e:AssumpF>}
\end{align}
is satisfied for $( \mathbf{f}, g ) = (h_t \nabla V, h_t) , ( h_t \nabla \big(\log\, \frac{v_t}{\mathfrak{m}}\big),  h_t)$. 

 The relevancy of our gradient type estimates can be seen from the following facts. If we have \eqref{e:AssumpG<} for $(f,\mathbf{g})$,  then  
\begin{align}\label{e:ConsSep26G<}
\int_{\mathbb{R}^n} 
P_s f P_s \big[  \mathbf{g} \cdot \nabla V \big] \, d\mathfrak{m}  
\le \int_{\mathbb{R}^n}
P_s f P_s \big[ {\rm div}\, \mathbf{g} \big] \, d\mathfrak{m} 
+ 
e^{-2Ks} \int_{\mathbb{R}^n} 
P_s \big[\nabla f\big] \cdot P_s \mathbf{g}  \, d\mathfrak{m}.  
\end{align}
Also, if we have \eqref{e:AssumpF>} for $(\mathbf{f}, g)$,  then 
\begin{align}\label{e:ConsSep26F>}
\int_{\mathbb{R}^n} 
P_s \mathbf{f} \cdot  P_s \big[ g \nabla V \big] \, d\mathfrak{m}  \ge
\int_{\mathbb{R}^n} 
P_s \mathbf{f} \cdot P_s \big[ \nabla g \big]  \, d\mathfrak{m} 
+ 
e^{-2Ks} \int_{\mathbb{R}^n} 
P_s \big[{\rm div}\, \mathbf{f} \big] P_s g  \, d\mathfrak{m}.  
\end{align}
These two properties can be seen by integration by parts. In fact, since $P_s$ is self-adjoint with respect to $d\mathfrak{m}$, we have that 
\begin{align*}
\int_{\mathbb{R}^n} 
P_s f P_s \big[ \mathbf{g} \cdot \nabla V  \big]  \, d\mathfrak{m} 
=&
\int_{\mathbb{R}^n} 
P_{2s} f 
(\mathbf{g}\cdot \nabla V)
\, d\mathfrak{m} \\
=& 
\int_{\mathbb{R}^n} 
P_{2s} f  \mathbf{g} \cdot \nabla(-\mathfrak{m}) \, dx \\
=& 
\int_{\mathbb{R}^n} 
P_s f   P_s\big[ {\rm div}\, \mathbf{g} \big] \, d\mathfrak{m}
+ 
\int_{\mathbb{R}^n} 
\nabla \big( P_{2s} f \big) \cdot \mathbf{g} \, d\mathfrak{m},
\end{align*}
which shows the first claim.

We remark that the proof of Proposition \ref{Prop:Uhmmm2} continues to work for general potentials.   
Hence, armed with \eqref{e:ConsSep26G<} and \eqref{e:ConsSep26F>}, we can modify and generalise the proof of Theorem \ref{t:MainHyper} under the assumptions above.  
In fact, in the proof of Theorem \ref{t:MainHyper}, we used identities \eqref{e:ConsSep26G=} and \eqref{e:ConsSep26F=} for the terms
\begin{align*}
&\Upsilon_1(p,s,\beta) \int_{\mathbb{R}^n} P_s \big[ h_{  t} |\nabla V|^2 \big]P_s h_{  t}^{q-1}\, d\mathfrak{m}, \\
& ( 1-\frac1\beta )^2\int_{\mathbb{R}^n} | P_s\big[ h_{  t}\nabla V \big] |^2 P_sh_{  t}^{q-2} \, d\mathfrak{m},\\
&\Upsilon_2(p,s,\beta) \int_{\mathbb{R}^n} P_s\big[ h_{  t}\nabla \big( \log\, \frac{v_t}{\mathfrak{m}} \big) \big] \cdot P_s\big[ h_{  t}\nabla V \big] P_sh_{  t}^{q-2}\, d\mathfrak{m},\\ 
&\Upsilon_3(p,s,\beta) \int_{\mathbb{R}^n} P_s\big[ h_{  t}\nabla V\cdot \nabla \log\, \frac{v_t}{\mathfrak{m}} \big] P_sh_{  t}^{q-1}\, d\mathfrak{m}
\end{align*}
and other parts of the proof did not make use of the special nature of $V(x) = \frac12|x|^2$. The signs of the coefficients $\Upsilon_j(p,s,\beta)$, $j=1,2,3$, are key to establishing the non-decreasingness of $Q$. 
Note that $\Upsilon_1(p,s,\beta) \le 0$ and $ (1-\frac1\beta)^2, \Upsilon_2(p,s,\beta) = -(1-\frac1\beta) \ge 0$ under $0<\beta<1$.  
In the case $q=2$, we have 
$$
\Upsilon_3(p,s,\beta) = 2e^{-2Ks} - 1 + \frac{(1-e^{-2Ks})^2}{p}(1-\frac1\beta)
$$
and hence if $s < \frac{1}{2K}\log\, 2$, then there exists some $\beta_*\in(0,1)$ for which $\Upsilon_3(p,s,\beta_*)>0$. 

In summary, if the potential $V$ satisfies \eqref{e:GradEstStrong} for some $K>0$, the exponents satisfy $0<s<\frac{1}{2K}\log\, 2$, $1<p<q=2$, $\frac{q-1}{p-1} = e^{2Ks}$, and the following holds for some $\beta = \beta_*\in(0,1)$ sufficiently close to $1$, then $Q$ can be shown to be non-decreasing on $(0,\infty)$. The input $v_0 \in L^2(\mathfrak{m}_{\beta}^{-1})$ and its evolution $v_t$ under \eqref{e:FPPotential} are semi-log-concave in the sense that \eqref{e:LogConvex} holds and satisfy gradient type estimates. If also
\[
\lim_{t \to \infty} v_t = \bigg( \int_{\mathbb{R}^n} \frac{v_0}{\mathfrak{m}}\, d\mathfrak{m} \bigg) \mathfrak{m}_\beta
\]
in an appropriate sense, we obtain the regularised hypercontractivity inequality
\begin{equation*}
\big\| P_s \big[ \big( \frac{v_0}{\mathfrak{m}} \big)^\frac1p \big] \big\|_{ L^q(\mathfrak{m} ) } \le \big\| P_s \big[ \big( \frac{\mathfrak{m}_\beta}{\mathfrak{m}} \big)^\frac1p \big] \big\|_{ L^q(\mathfrak{m} ) } \bigg( \int_{\mathbb{R}^n} \frac{v_0}{\mathfrak{m}}\, d\mathfrak{m} \bigg)^\frac1p.
\end{equation*}
Admittedly, the conditions we impose are somewhat forbidding and, as far as we can tell, it is not clear how to identify concrete examples outside the gaussian case. In the following subsection, we change tack somewhat and more fruitfully we obtain regularised LSI outside the gaussian case via mass transportation rather than flow monotonicity.

\subsection{Regularised LSI for general measures}\label{S5-1}
On a general measure space, the gaussian LSI \eqref{e:LogSob} can be generalised as follows.  
Provided $\nabla^2 V \ge K {\rm id}$ for some $K>0$,  then 
\begin{equation}\label{e:LogSobPotential}
{\rm Ent}_{\mathfrak{m}} (f) \le \frac1{2K} {\rm I}_{\mathfrak{m}} (f)
\end{equation} 
holds for all non-negative locally Lipschitz $f \in L^1(\mathfrak{m})$, where 
\begin{align*}
{\rm Ent}_{\mathfrak{m}} (f) & := \int_{\mathbb{R}^n} \, f\log\, f\, d\mathfrak{m} - \int_{\mathbb{R}^n} \, f \, d\mathfrak{m} \log\, \int_{\mathbb{R}^n} \, f \, d\mathfrak{m}, \\
{\rm I}_{\mathfrak{m}} (f) & := \int_{\mathbb{R}^n} \frac{|\nabla f|^2}{f}\, d\mathfrak{m}.
\end{align*}
See, for example, \cite{BGL} for further details.  At this level of generality, the cases of equality for \eqref{e:LogSobPotential} was recently investigated by Ohta--Takatsu \cite{OhTa}.  In particular,  one can see from \cite{OhTa} that $f = \frac{\mathfrak{m}_\beta}{\mathfrak{m}}$ does not attain equality in \eqref{e:LogSobPotential} unless $\beta=1$. 
Hence, as in the spirit of Theorem \ref{t:MainLSI}, we may expect to improve the constant in \eqref{e:LogSobPotential} under an appropriate log-convexity assumption. To avoid difficulties encountered in the previous subsection, here we investigate regularised LSI via a mass-transport approach. As mentioned earlier, Cordero-Erausquin \cite{Cor} obtained a mass-transport proof of the classical LSI and this inspired the following argument.

For the sake of simplicity we discuss the one-dimensional case. 
\begin{theorem}\label{t:LSIPo}
Let $K>0$ and $\beta > 1$.  
Suppose $v$ and $V$ are symmetric on $\mathbb{R}$ and satisfy 
$$V''\ge K,\;\;\; (\log\,v)'' \ge - \frac{K}{\beta},\;\;\; \int_{\mathbb{R}} v\, dx =1, \;\;\; \lim_{|x| \to +\infty} |V'(x)|v(x) = 0.$$ 
Then we have 
\begin{equation}\label{e:LSIPo}
{\rm Ent}_{\mathfrak{m}} \big( \frac{v}{\mathfrak{m}} \big)
- \frac1{2K} {\rm I}_{\mathfrak{m}} \big( \frac{v}{\mathfrak{m}} \big)
\le
{\rm Ent}_{\mathfrak{m}} \big( \frac{\mathfrak{m}_\beta}{\mathfrak{m}} \big)
- \frac1{2K} {\rm I}_{\mathfrak{m}} \big( \frac{\mathfrak{m}_\beta}{\mathfrak{m}} \big)
+ (1-\frac1{\beta}) \int_{\mathbb{R}}  \frac{V''}{K} (v -\mathfrak{m}_\beta)\, dx.  
\end{equation}

In particular, if we additionally assume $V''\le L$, then 
\begin{equation}\label{e:LSIPo2}
{\rm Ent}_{\mathfrak{m}} \big( \frac{v}{\mathfrak{m}} \big)
- \frac1{2K} {\rm I}_{\mathfrak{m}} \big( \frac{v}{\mathfrak{m}} \big)\\
\le
{\rm Ent}_{\mathfrak{m}} \big( \frac{\mathfrak{m}_\beta}{\mathfrak{m}} \big)
- \frac1{2K} {\rm I}_{\mathfrak{m}} \big( \frac{\mathfrak{m}_\beta}{\mathfrak{m}} \big)
+
(1-\frac1\beta) \frac{L-K}{K}. 
\end{equation}

\end{theorem}

\begin{remark}
It is not immediately clear if \eqref{e:LSIPo2}  improves upon \eqref{e:LogSobPotential} or not. 
However, if one notices that 
$$
V(x) \le V(0) + \frac1{2K} |V'(x)|^2
$$
from $V''\ge K$, then one can see that 
$$
{\rm Ent}_{\mathfrak{m}} \big( \frac{\mathfrak{m}_\beta}{\mathfrak{m}} \big)
- \frac1{2K} {\rm I}_{\mathfrak{m}} \big( \frac{\mathfrak{m}_\beta}{\mathfrak{m}} \big)
\le  
- \frac12 \log\, \big( \frac{2\pi \beta}{L} \big) + V(0) + \frac{L}{2K} (1-\frac1{\beta}) + \log\, Z, 
$$
provided $K\le V''\le L$. 
In particular, this shows that $ {\rm Ent}_{\mathfrak{m}} \big( \frac{\mathfrak{m}_\beta}{\mathfrak{m}} \big)
- \frac1{2K} {\rm I}_{\mathfrak{m}} \big( \frac{\mathfrak{m}_\beta}{\mathfrak{m}} \big)
\to -\infty $ as $\beta\to\infty$. Therefore there exists sufficiently large $\beta = \beta(V) \ge1$ such that \eqref{e:LSIPo2} improves \eqref{e:LogSobPotential}. 
\end{remark}
\begin{proof}[Proof of Theorem \ref{t:LSIPo}]
Let $T:\mathbb{R}\to\mathbb{R}$ be Brenier's map from $v$ to $\mathfrak{m}_\beta$ in which case \eqref{e:Transport} and \eqref{e:MA} can be read as 
\begin{equation}\label{e:MA2}
\int_{\mathbb{R}} h(T(x))\, dv = \int_{\mathbb{R}} h(x)\, d\mathfrak{m}_\beta,\;\;\; v(x) = \mathfrak{m}_\beta(T(x))  T'(x). 
\end{equation} 
Then we have from \eqref{e:MA2} that 
\begin{align*}
{\rm Ent}_{\mathfrak{m}} \big( \frac{v}{\mathfrak{m}} \big)
=& 
\int_{\mathbb{R}} \big(\log\, v- \log\, \mathfrak{m} \big)\, dv \\
=& 
\int_{\mathbb{R}} \big(\log\, \mathfrak{m}_\beta(T(x)) + \log\, T'(x) -\log\, \mathfrak{m}(x) \big)\, dv \\
=& 
-\log\, Z_\beta + \log\, Z - \frac1\beta \int_{\mathbb{R}} V(T(x))\, dv+ \int_{\mathbb{R}} V(x)\, dv + \int_{\mathbb{R}} \log\, T'(x)\, dv.
\end{align*}
In the case of $v = \mathfrak{m}_\beta$, we have that 
$$
{\rm Ent}_{\mathfrak{m}} \big( \frac{\mathfrak{m}_\beta}{\mathfrak{m}} \big)
=
-\log\, Z_\beta + \log\, Z + (1 - \frac1\beta) \int_{\mathbb{R}} V(T(x))\, dv.
$$
Regarding the Fisher information, from the definition 
$$
{\rm I}_{\mathfrak{m}} \big( \frac{v}{\mathfrak{m}} \big)
=
\int_{\mathbb{R}} \big| (\log\, v)' + V' \big|^2\, dv,
\;\;\; 
{\rm I}_{\mathfrak{m}} \big( \frac{\mathfrak{m}_\beta}{\mathfrak{m}} \big)
=
(1-\frac{1}{\beta} )^2 \int_{\mathbb{R}}   |V'|^2\, d\mathfrak{m}_\beta. 
$$
Hence 
\begin{align*}
R(v)
& :=
{\rm Ent}_{\mathfrak{m}} \big( \frac{v}{\mathfrak{m}} \big)
- \frac1{2K} {\rm I}_{\mathfrak{m}} \big( \frac{v}{\mathfrak{m}} \big)
- {\rm Ent}_{\mathfrak{m}} \big( \frac{\mathfrak{m}_\beta}{\mathfrak{m}} \big)
+ \frac1{2K} {\rm I}_{\mathfrak{m}} \big( \frac{\mathfrak{m}_\beta}{\mathfrak{m}} \big)\\
& =
\int_{\mathbb{R}} \big(V(x) - V(T(x)) \big) \, dv + \int_{\mathbb{R}} \log\, T'(x)\, dv\\
&\qquad -\frac1{2K} \int_{\mathbb{R}} \big| (\log\, v)' + V' \big|^2\, dv
+ \frac1{2K} (1-\frac{1}{\beta} )^2 \int_{\mathbb{R}}   |V'|^2\, d\mathfrak{m}_\beta. 
\end{align*}
We know that $V''\ge K$ implies 
$$
V(x) - V(T(x)) \le V'(x) (x-T(x)) - \frac{K}{2} | T(x) - x |^2
$$
from which we obtain 
\begin{align*}
R(v)
\le&
-\int_{\mathbb{R}} (T(x)-x) V'(x)  \, dv - \frac{K}{2} \int_{\mathbb{R}} |T(x)-x|^2\, dv  + \int_{\mathbb{R}} \log\, T'(x)\, dv\\
&-\frac1{2K} \int_{\mathbb{R}} \big| (\log\, v)' + V' \big|^2\, dv
+ \frac1{2K} (1-\frac{1}{\beta} )^2 \int_{\mathbb{R}}   |V'|^2\, d\mathfrak{m}_\beta. 
\end{align*}
We focus on 
\begin{align*}
&- \frac{K}{2} \int_{\mathbb{R}} |T(x)-x|^2\, dv-\frac1{2K} \int_{\mathbb{R}} \big| (\log\, v)' + V' \big|^2\, dv\\
=&
-\frac12 \int_{\mathbb{R}} \big| \sqrt{K} ( T(x)-x ) + \frac{1}{\sqrt{K}} \big( (\log\, v)'(x) +V'(x) \big)  \big|^2\, dv\\
&+ 
\int_{\mathbb{R}} \big(  (\log\, v)'(x) +V'(x) \big)( T(x)-x )\, dv.
\end{align*}
Notice simply that 
$$
\int_{\mathbb{R}} (\log\, v)'( T(x)-x )\, dv 
=
- \int_{\mathbb{R}} (T'(x) - 1)\, dv. 
$$
This integration by parts can be justified from assumptions $\lim_{|x| \to +\infty} |V'(x)|v(x) = 0$, $V'' \ge K$, and the fact that $V$ is symmetric.
In fact, as we will see later,  we have from $V'' \ge K$ and the symmetric assumption on $V$ that  $|T(x)| \le |x|\le \frac1K |V'(x)|$. This yields that $v(x)|x|, v(x)|T(x)| \to 0$ as $|x| \to +\infty$.
Therefore we obtain 
\begin{align}\label{e:Simple}
R(v)
\le&
\int_{\mathbb{R}} \big( \log\, T'(x) - T'(x) +1 \big)  \, dv 
+ \frac1{2K} (1-\frac{1}{\beta} )^2 \int_{\mathbb{R}}   |V'|^2\, d\mathfrak{m}_\beta\\
&
-\frac12 \int_{\mathbb{R}} \big| \sqrt{K} ( T(x)-x ) + \frac{1}{\sqrt{K}} \big( (\log\, v)' +V' \big)  \big|^2\, dv. \nonumber 
\end{align}
We then reorganise terms in the third term as 
\begin{align*}
&\big| \sqrt{K} ( T(x)-x ) + \frac{1}{\sqrt{K}} \big( (\log\, v)' +V' \big)  \big|^2\\
& =
\big| \frac{\sqrt{K}}{\beta} ( T(x)-x ) + \frac{1}{\sqrt{K}} \big( (\log\, v)' + \frac{V'}{\beta} \big)  + (1-\frac1{\beta})\sqrt{K} ( T(x)-x ) + (1-\frac1{\beta}) \frac{V'}{\sqrt{K}} \big|^2\\
& = 
\big| \frac{\sqrt{K}}{\beta} ( T(x)-x ) + \frac{1}{\sqrt{K}} \big( (\log\, v)' + \frac{V'}{\beta} \big) \big|^2 \\
& \qquad + 2  (1-\frac1{\beta}) (\log\, v)' \big(  T(x)-x  +  \frac{V'}{K}  \big)  
+
K (1-\frac1{\beta^2}) \big| T(x)-x  +  \frac{V'}{K}  \big|^2. 
\end{align*}
We further rearrange $\big| T(x)-x  +  \frac{V'}{K}  \big|^2$ as
\begin{align*}
&\big| T(x)-x  +  \frac{V'}{K}  \big|^2\\
& =
\big| T(x)-x  +  \frac{V'(x)}{K} - \frac{V'(T(x))}{K}  \big|^2
+2 \big(T(x)-x  +  \frac{V'}{K}\big)  \frac{V'(T(x))}{K}
- 
\big|  \frac{V'(T(x))}{K} \big|^2. 
\end{align*}
Since $\beta\ge1$, we obtain  
\begin{align*}
&\big| \sqrt{K} ( T(x)-x ) + \frac{1}{\sqrt{K}} \big( (\log\, v)' +V' \big)  \big|^2\\
& \ge
2  (1-\frac1{\beta}) (\log\, v)' \big(  T(x)-x  +  \frac{V'}{K}  \big) \\
& \qquad +
2 (1-\frac1{\beta^2})\big(T(x)-x  +  \frac{V'}{K}\big)  V'(T(x)) 
-
\frac1{K}(1-\frac1{\beta^2}) |V'(T(x))|^2. 
\end{align*}
Now we appeal to the assumption that $V,v$ are symmetric to see that 
\begin{equation}\label{e:Hiroshi}
\big(T(x)-x  +  \frac{V'}{K}\big)  V'(T(x))
\ge 
\frac{ |V'(T(x))|^2}{K}
\end{equation}
as follows. Consider the case $x>0$. First notice that 
if we define 
$\psi(x):= -x + \frac1KV'(x)$, then 
$$
\psi'(x)= -1 + \frac1KV''(x) \ge0
$$
and hence $\psi$ is monotone increasing.  From the convexity assumptions on $V$ and $v$, Caffarelli's contraction theorem \cite{Caf,Kol} shows that $0\le T' \le 1$. Also the symmetry of $v,V$ yields that $T(0) = 0$. These two imply $0\le T(x)\le x$ and hence 
$\psi(x) \ge \psi(T(x))$. In other words, 
$$
T(x) - x +\frac1{K} V'(x)\ge \frac1KV'(T(x)). 
$$
Also the symmetry of $V$ yields $V'(0) = 0$ and hence $\psi(x) \ge \psi(0) = 0$ which turns into $V'(x) \ge Kx\ge0$. 
Therefore we see \eqref{e:Hiroshi} when $x\ge0$. In a similar way, one can show \eqref{e:Hiroshi} even when $x\le 0$. 

We then apply \eqref{e:Hiroshi} to obtain  
\begin{align*}
&\big| \sqrt{K} ( T(x)-x ) + \frac{1}{\sqrt{K}} \big( (\log\, v)' +V' \big)  \big|^2\\
& \ge
2  (1-\frac1{\beta}) (\log\, v)' \big(  T(x)-x  +  \frac{V'}{K}  \big) 
+
\frac{1}{K}(1-\frac1{\beta^2}) |V'(T(x))|^2.  
\end{align*}
Overall we derive 
\begin{align*}
R(v)
\le&
\int_{\mathbb{R}} \big( \log\, T'(x) - T'(x) +1 \big)  \, dv 
+ \frac1{2K} (1-\frac{1}{\beta} )^2 \int_{\mathbb{R}}   |V'|^2\, d\mathfrak{m}_\beta\\
&
-\frac12 \int_{\mathbb{R}} \bigg( 2  (1-\frac1{\beta}) (\log\, v)' \big(  T(x)-x  +  \frac{V'}{K}  \big) 
+
\frac{1}{K}(1-\frac1{\beta^2}) |V'(T(x))|^2 \bigg) \, dv \\
=& 
\int_{\mathbb{R}} \big( \log\, T'(x) - T'(x) +1 \big)  \, dv 
+ \frac1{2K} (1-\frac{1}{\beta} )^2 \int_{\mathbb{R}}   |V'|^2\, d\mathfrak{m}_\beta\\
&
-(1-\frac1{\beta}) \int_{\mathbb{R}} v'(x)  \big(  T(x)-x  +  \frac{V'}{K}  \big) \, dx 
-
\frac{1}{2K} (1-\frac1{\beta^2}) \int_{\mathbb{R}} |V'(T(x))|^2 \, dv.  
\end{align*}
We finally use \eqref{e:MA2} to see that 
\begin{align*}
&\frac1{2K} (1-\frac{1}{\beta} )^2 \int_{\mathbb{R}}   |V'|^2\, d\mathfrak{m}_\beta-
\frac{1}{2K} (1-\frac1{\beta^2}) \int_{\mathbb{R}}  |V'(T(x))|^2 \, dv \\
=& 
- \frac1{K} (1-\frac{1}{\beta} ) \int_{\mathbb{R}}   V'' \, d\mathfrak{m}_\beta, 
\end{align*}
where we also use $\lim_{|x| \to +\infty} |V'(x)|v(x) = 0$ to validate the integration by parts step.
From this we conclude
\begin{align*}
R(v)
\le&
\int_{\mathbb{R}} \big( \log\, T'(x) - \frac{1}{\beta}T'(x) +\frac{1}{\beta} \big)  \, dv 
+ (1-\frac1{\beta})\frac1{K}\int_{\mathbb{R}}  V'' (v -\mathfrak{m}_\beta)\, dx
\end{align*}
and since $0\le T'(x)\le 1$ we thus obtain \eqref{e:LSIPo}. 
\end{proof}

\section*{Acknowledgements}
The authors thank Jonathan Bennett, Fumio Hiroshima, Shin-ichi Ohta, and Asuka Takatsu for sharing their deep insight from several mathematical fields. 
The authors would also like to express their appreciation to Yair Shenfeld for extremely generous discussions which, in particular, led us to pursue the possibility of extending our results to the setting of semi-log-subharmonicity as well as the results in Corollary \ref{Cor:Matrix} in which the dependence on the underlying dimension is weakened. The conjectural estimate for Talagrand's inequality in \eqref{e:TIMWeak} under the small-covariance assumption was also raised by him.

This work was supported by JSPS Kakenhi grant numbers 18KK0073, 19H00644 and 19H01796 (Bez), Grant-in-Aid for JSPS Research Fellow no. 17J01766 and JSPS Kakenhi grant numbers 19K03546, 19H01796 and 21K13806 (Nakamura), and JST,  ACT-X Grant Number JPMJAX200J, Japan, and JSPS Kakenhi grant number 19H01796 (Tsuji).


\end{document}